\documentclass[a4paper, 11pt, reqno]{amsart}

\usepackage[left=3cm, right=3cm, top=3cm, bottom=3cm]{geometry}
\usepackage{amsmath, amssymb, amsthm, textcomp, verbatim, graphicx, xcolor, mathtools, mathrsfs}
\usepackage{wrapfig, comment}
\usepackage[]{mdframed} %to place text in a box
\usepackage{bm} %for boldmath 
\usepackage{xparse}
\usepackage{pifont}
\usepackage[utf8]{inputenc}
\usepackage[T1]{fontenc}
\usepackage[all]{xy}
\usepackage[]{hyperref}
\hypersetup{colorlinks = true, linkcolor = blue, filecolor = blue, urlcolor = blue, citecolor=blue}
\theoremstyle{definition} \newtheorem{defn}{Definition}[section]
\theoremstyle{remark} \newtheorem{rem}[defn]{Remark}
\theoremstyle{remark} \newtheorem*{rem*}{Remark}
\theoremstyle{plain} \newtheorem{thm}[defn]{Theorem}
\theoremstyle{plain} \newtheorem{lem}[defn]{Lemma}
\theoremstyle{plain} \newtheorem{prop}[defn]{Proposition}
\theoremstyle{plain} \newtheorem{cor}[defn]{Corollary}
\theoremstyle{remark} 
\theoremstyle{remark} 
\theoremstyle{remark} 

\numberwithin{equation}{section} %numbers equation according to the sections; requires amsmath

\newcommand{\mc}{\mathcal}
\newcommand{\x}{\times}
\newcommand{\C}{\mathbb{C}}
\newcommand{\D}{\mathbb{D}}
\newcommand{\R}{\mathbb{R}}
\newcommand{\Q}{\mathbb{Q}}
\newcommand{\Z}{\mathbb{Z}}
\newcommand{\N}{\mathbb{N}}

\newcommand{\pr}{\mathbb{P}}
\newcommand{\E}{\mathbb{E}}
\newcommand{\bH}{\mathbb{H}}

\newcommand{\Poi}{\operatorname{Poi}}
\newcommand{\ind}{\mathbf{1}}

\newcommand{\clpl}{\color{purple} }

\begin{document}
	\title[Real layering field of Brownian loop soup]{The real layering field of Brownian loop soup and the Gaussian multiplicative chaos}

    %\markleft{The }
    %\markright{real}

    \author{Sayantan Maitra}

    \address{Division of Science, New York University Abu Dhabi, Abu Dhabi, United Arab Emirates.}

    \email{sayantanmaitra123@gmail.com}

	\keywords{Brownian loop soup; Conformal invariance; Gaussian Multiplicative Chaos; Wiener-It\^{o} chaos expansion}

	%\subjclass[2020]{}

    \begin{abstract}
        We consider the random field defined by the layering numbers of the Brownian loop soup in a bounded simply connected domain in the complex plane. We call this the layering field and show that, after a suitable renormalization, it converges to the subcritical Gaussian multiplicative chaos. The main technique for our proof is the Wiener-It\^{o} chaos expansion. We also calculate the $n$-point functions of the layering field, show their conformal covariance and discuss their behavior near the boundary of the domain. 
    \end{abstract}
    
	\maketitle
	
	\section{Introduction}

    The Brownian loop measure $\mu^{loop}$ is a $\sigma$-finite conformally invariant  measure on the set of all unrooted Brownian loops in the complex plane. The Brownian loop soup (BLS), denoted by $\mc{L}^{\lambda}$, is the Poisson point process having intensity $\lambda \mu^{loop}$, $\lambda>0$. Given the collection of loops in $\mc{L}^{\lambda}$ in a domain $D\subseteq \C$ and a point $z 
    \in D$, let $N_{\lambda, D}(z)$ be the number (with a $+$ or $-$ sign) of loops whose interiors contain $z$. We call this the \textit{layering number} of the BLS at the point $z$. For a  real number $\beta$, we call the collection 
    \begin{align}\label{eq:layering_field_0}
        e^{\beta N_{\lambda, D}(z)}, \quad z \in D
    \end{align}
    the \textit{layering field} of the BLS in D. Our goal in this article is to show that, as $\lambda \to \infty$ and $\beta \to 0$, this field converges to the \textit{Gaussian Multiplicative Chaos} (GMC). 
    
    We shall make all the above terms more precise and mathematically rigorous in the next section. Let us first discuss some of the background of this problem. Since its introduction by Lawler and Werner \cite{lw04}, the Brownian loop soup has received much attention. Apart from the fact that it can be seen as the scaling limit of the random walk loop soup on discrete $\Z^2$ lattice \cite{lt07}, various authors have studied its connection to other models in probability theory and statistical physics. For instance, it is known that when the intensity is low, the outer boundaries of the clusters of Brownian loops in the soup is a Conformal Loop Ensemble (CLE) \cite[Theorem 3]{werner06}. The CLE is a conformally invariant collection of simple (i.e. non-self-intersecting) and non-crossing loops in the plane. Also, when $\lambda=\frac{1}{2}$, the occupation time of the BLS has been shown to have the same law as that of the square of the Gaussian free field (GFF). This is the continuum version of an isomorphism result due to Le Jan \cite{lejan10}; see \cite[Section 3.4]{wp21} for a discussion on this issue. Another interesting object, which has many similarities with the layering field defined above, is the \textit{Winding field}. This is constructed from the winding numbers at a point $z$, say $N_w(z)$, which is the sum over all the winding numbers corresponding to the loops (that cover $z$) in a realization of the BLS. In an article by van de Brug, Camia and Lis \cite{vcl18}, the authors study the (imaginary) winding field (i.e. objects of the form $e^{i \beta N_w(z)}, z \in D, \beta \in \R$). They show that, roughly speaking, this object exists as a random distribution taking values in a Sobolev space of negative index. For a more comprehensive review of the Brownian loop soup and its connections to various areas of mathematical physics, we refer the reader to the article \cite{camia17} and the references therein. 

    Our motivation for the present paper comes from the works of Camia, Gandolfi, Kleban \cite{cgk16} and Camia, Gandolfi, Peccati and Reddy \cite{cgpr21}. These studied the imaginary layering field $\{e^{i \beta N_{\lambda, D}(z)}, z \in D\}$. The former computed their $n$-point correlation functions and showed the conformal covariance of these $n$-point functions and the latter proved that with an appropriate renormalization, these fields converge to a (tilted) imaginary Gaussian Multiplicative Chaos by making use of a technique known as \textit{Wiener-It\^{o} chaos expansion}. 
    
    The present article continues this line of research and provides the complementary results for the real-valued layering field. However, there are some key differences with the imaginary case, which we note here. While the imaginary multiplicative chaos exists as a random element in the space $H^{s}$ (the Sobolev space with some index $s<0$) 
    \cite{jsw20}, the real GMC, which we shall denote by $M_{\xi}$, is a random measure and our proof requires a slightly different set of arguments. Furthermore, unlike in the imaginary case, the unboundedness of the real-valued field introduces new challenges for our analysis. More importantly, in dimension 2, $M_{\xi}$ is known to be non-degenerate (i.e. not identically zero) only when $0<\xi< 2$. This is known as the \textit{sub-critical phase} of the GMC. In the critical and super-critical phases (i.e. when $\xi\ge 2$), one can still make sense of the multiplicative chaos using other techniques, but the GMC has different behaviors compared to the sub-critical phase. At this point, our result only covers a part of the sub-critical phase. In other words, using the standard technique of chaos decomposition, we are only able to connect the layering field with the sub-critical GMC.   
    
    \textbf{Organization:} The plan for this paper is as follows. In Section \ref{sec:prelim} we briefly review the definitions and other relevant notions that we are going to study. The main results are stated in Section \ref{sec:prelim_main_results}. The first two main results (Theorems \ref{thm:main_1} and \ref{thm:main_2}) describe the construction and prove the existence of the Poissonian and Gaussian realizations of the layering model. Their proofs are sketched in Section \ref{sec:PLF_GLF_exist}. The final main result (Theorem \ref{thm:main_3}) connects the Poisson layering field with the Gaussian one, and its proof can be found in Section \ref{sec:PLF_to_GLF}. Since many of the calculations required for these proofs are routine, we have postponed them to the Appendix sections. 

    \textbf{Acknowledgment:} I am grateful to Federico Camia for suggesting this problem to me and for his guidance throughout the project. I also thank New York University Abu Dhabi for supporting this research.

    \section{Preliminaries and results}\label{sec:prelim}
    
	\subsection{The loop measure}\label{sec:prelim_loop_measure}

    Let $D$ be either a simply connected open domain in $\C$ with smooth (i.e. $C^1$) boundary or $\C$ itself. In this section, we are going to recall the definition of the unique (up to a multiplicative constant) non-trivial measure $\mu_D$ on simple loops in $D$ that satisfies the \textit{conformal restriction property}. This basically means that the collection of measures $\{\mu_D\}_{D\subset \C}$ behaves well under conformal maps; see below for a precise definition.
    
    We shall start by defining the \textit{Brownian loop measure} $\mu^{loop}$. For us, a \textit{loop} is a continuous curve $\gamma : [0, t_{\gamma}] \to \C$ such that $\gamma(0) = \gamma(t_{\gamma})$. The time-length $t_{\gamma}$ is a positive number associated with $\gamma$. We are not going to distinguish two loops that can be obtained as a time-shift of one another.\footnote{Note however that, we do distinguish a loop and its reverse. In other words, each unrooted loop has an orientation. But this only contributes a factor of 2 to most of our calculations.} This is called \textit{unrooting} the loops and the collection of all unrooted loops in $D$ is denoted by $\mc{K}(D)$. This is a metrizable space and has a Borel $\sigma$-algebra $\mc{B}(\mc{K}(D))$ on it. We refer the reader to \cite[Chapter 5]{lawler05} or \cite{lw04} for more details regarding the topology of this space.
    	
	Now we are ready to define the Brownian loop measure. For a fixed time $t >0$ and $z \in D$, let $\mu^{br}_D(z,z;t)$ denote the law of the Brownian bridge of time length $t$ from $z$ to itself, conditioned to stay inside $D$. This induces a probability measure on $\mc{K}(D)$, which we also denote by $\mu^{br}_D(z,z;t)$. Then the \textit{Brownian loop measure} $\mu^{loop}_D$ is defined on the measurable set $(\mc{K}(D), \mc{B}(\mc{K}(D)))$ as follows:
	\begin{align}\label{eq:loop_measure}
		\mu^{loop}_D(\cdot) = \int_{D} \int_0^{\infty} \frac{1}{2\pi t^2} \mu^{br}_D(z,z;t)(\cdot) \, dt \, dz. 
	\end{align}
    This is not a finite measure. In particular, since the above time integral is unbounded at $0$, the total mass of $\mu^{loop}_D$ is infinite. It is also known that this is conformally invariant, in the sense that, if $f$ is a conformal map from $D$ onto another domain $D'\subset \C$, then $f\circ\mu^{loop}_D = \mu^{loop}_{D'}$ (Stating this precisely requires some care, as one needs to use Brownian time change under the conformal transformation. See \cite[Proposition 6]{lw04}). 
	
	Given a $\gamma \in \mc{K}(D)$, we shall use the notation $\bar{\gamma}$ to denote its \textit{hull}, i.e. the complement of the unique component of $D\setminus \gamma([0, t_{\gamma}])$ that touches the boundary $\partial D$  (or, the complement of the unique unbounded component, in case $D$ is unbounded).     
    
    For this paper we will only need to consider the boundaries of the loops under the Brownian loop measure. It is known that, with probability 1, the boundary $\partial \bar{\gamma}$ of a Brownian loop $\gamma$ defines a simple (i.e. non self-intersecting) loop in $D$ (see the discussion in p. 148 and Proposition 6 in \cite{werner08}). Let $S(D)$ be the set consisting of all simple loops in $D$. In this collection we only consider the traces of the loops and ignore their time parametrization. By a slight abuse of notations, we denote by $\gamma$ the image (or the trace) of the loop $\gamma$. The $\sigma$-algebra on $S(D)$, denoted by $\mc{S}(D)$, is generated by collections of loops that are contained in annular regions in $D$. In particular, all the sets of loops that we shall consider in the sequel (e.g. those defined in \eqref{eq:loop_sets} and \eqref{eq:loop_sets_disk}) are measurable with respect to $\mc{S}(D)$. 
    
    Since each loop in $\mc{K}(D)$ defines one in $S(D)$ (via the boundary of the hull of the first loop), the Brownian loop measure $\mu^{loop}_D$ induces a measure on $S(D)$, which we simply call $\mu_D$. The resulting family $\{ \mu_D \}_{D}$ of measures on simple loops, indexed by domains $D \subset \C$, has the following two properties, known as \textit{conformal restriction property}:
    \begin{itemize}
        \item If $D' \subset D$, then the restriction of $\mu_{D}$ to $D'$ is $\mu_{D'}$.

        \item If $f: D \to D'$ is a conformal equivalence between two domains $D, D' \subset \C$, then 
        $f\circ \mu_{D} = \mu_{D'}$, in the sense that,  for all measurable subsets $A$ of $S(D')$,
        \begin{align}\label{eq:conf_inv}
            (f\circ \mu_{D})(A) := \mu_D (\gamma \in S(D) \mid f(\gamma) \in A) = \mu_{D'}(A). 
        \end{align}       
    \end{itemize}
    It was shown in \cite[Theorem 2]{werner08} that the above two conditions fully characterize the family of measures, up to multiplication by a positive constant. In other words, on the set of all planar simple loops, the measures $\mu_D$ induced by the Brownian loop measures $\mu^{loop}_D$ form the unique family satisfying the conformal restriction property.\footnote{We should point out that, aside from drawing boundaries of Brownian loops, there are other ways of obtaining the measure $\mu$ (up to a constant factor). See \cite{werner06} for a discussion on how this can be done using the boundaries of the critical percolation clusters.}
    
    Let us fix some useful notations for the sequel. Henceforth, we shall write $d(D)$, $d(\gamma)$ to denote the diameters of the domain $D$ and the (image of) a loop $\gamma$. For $z, w \in D$ and $R> \delta>0$ we define the following naturally occurring collections of simple loops and their mass under the measure $\mu_D$, as follows, 
	\begin{align}\label{eq:loop_sets}
		A_{\delta, D}(z) & = \{\gamma \in S(D) \mid z\in \bar{\gamma}, d(\gamma) \ge \delta\}, \quad \alpha_{\delta, D}(z) = \mu_D(A_{\delta}(z)) 
		\nonumber \\
		A_{\delta, D}(z,w) & = \{\gamma \in S(D) \mid z,w\in \bar{\gamma}, d(\gamma) \ge \delta\}, \quad \alpha_{\delta, D}(z, w)= \mu_D(A_{\delta}(z, w)) 
		\nonumber \\
		A_{\delta, D}(z|w) & = \{\gamma \in S(D) \mid z\in \bar{\gamma}, w\notin \bar{\gamma}, d(\gamma) \ge \delta\}, \quad \alpha_{\delta, D}(z| w)= \mu_D(A_{\delta}(z| w))
		\nonumber \\
		A_{\delta, R, D}(z) & = \{\gamma \in S(D) \mid z\in \bar{\gamma},d(\gamma) \in (\delta, R)\}, \quad \alpha_{\delta, R, D}(z)= \mu_D(A_{\delta,R}(z)),
	\end{align}
    and 
    \begin{align}\label{eq:loop_sets_disk}
        \bar{A}_{\delta, D}(z) & = \{\gamma \in S_D \mid z\in \bar{\gamma}, \gamma \nsubseteq B(z, \delta) \}, \quad \bar{\alpha}_{\delta, D}(z) = \mu_D(\bar{A}_{\delta, D}(z)) 
        \nonumber \\
        \bar{A}_{\delta, R, D}(z) & = \{\gamma \in S_D \mid z\in \bar{\gamma},\gamma \nsubseteq B(z, \delta) , \gamma \subseteq B(z, R) \}, \quad \bar{\alpha}_{\delta, R, D}(z)= \mu_D(\bar{A}_{\delta,R, D}(z))
    \end{align}
    where $B(z, \delta)$ denotes the open disk around $z$ with radius $\delta$.  $R$ and $\delta$ are sometimes called the \textit{infrared} and \textit{ultraviolet} cutoffs, respectively. Note that the above two methods of implementing the cutoff $\delta$ are equivalent, since we clearly have $\bar{A}_{\delta, D}(z) \subsetneq A_{\delta, D}(z) \subsetneq \bar{A}_{\delta/2, D}(z)$
    for all $\delta>0$ and $z \in D$, and therefore, 
    \begin{align}\label{eq:diam_disk_rel_2}
            \bar{\alpha}_{\delta, D}(z) \le \alpha_{\delta, D}(z) \le \bar{\alpha}_{\delta/2, D}(z).
    \end{align}
    However, cutoffs based on diameter of loops seem to behave well when two or more points are concerned, e.g. when calculating two-point functions of the layering fields. Therefore, most of our subsequent calculations will make use of the quantities defined in \eqref{eq:loop_sets}.         

    We recall from \cite[Lemma A1]{cgk16} that the masses of some of the collections appearing above can be computed exactly.
    \begin{lem}
        Let $z \in D$ and $0<\delta<R<\infty$ be such that $B(z, R) \subset D$. Then we have,
        \begin{align}\label{eq:mu_size_uv_ir}
            \alpha_{\delta, R, D}(z) =  \bar{\alpha}_{\delta, R, D}(z) = \frac{1}{5} \log \frac{R}{\delta},
        \end{align}       
    \end{lem}
    Let us also record a useful consequence of the above. If $D$ is a bounded domain, $z \in D$ and $\delta>0$ is small, we have $A_{\delta, D}(z) \subset A_{\delta, d(D), \C}(z)$, since every loop in $D$ must also be a loop in $\C$ of diameter less than that of $D$ itself. Thus, 
    \begin{align}\label{eq:mu_D_size}
        \alpha_{\delta, D}(z) \le \alpha_{\delta, d(D), \C } = \frac{1}{5} \log \frac{d(D)}{\delta},
    \end{align}
    for all small $\delta>0$. 

    Before stating the next result, we mention the so-called \textit{thinness property}, first introduced in \cite{nw11}. Roughly speaking, a measure on some class of compact curves in the plane is said to be \textit{thin}, if it gives only finite mass to the collection of large curves intersecting a bounded region. %It is known that the Brownian loop measure $\mu^{loop}$ (defined in \eqref{eq:loop_measure}) does  not enjoy the thinness property {\clpl (reference?)}.
    Lemma 4 of \cite{nw11} proved that the measure $\mu = \mu_{\C}$, defined on the outer boundaries the Brownian loops in the entire complex plane, is thin. More formally, for any non-empty bounded set $B\subset \C$, this property can be written as 
    \begin{align} \label{eq:thinness}
        \mu (\gamma \mid d(\gamma) > R, \gamma \cap B \neq \emptyset ) < \infty, \text{ for all } R>0.
    \end{align}
    
    Given a domain $D$, we shall often use the notation $d_z = d_{z, D}= d(z, \partial D)$ for the shortest distance between a point $z\in D$ and the boundary of $D$. The following result sheds light on the behavior of the loop measure near $\partial D$. We note that, by its virtue, for nice enough domains $D$ one can extend the quantity $\alpha_{d_z, D}(z)$ to the whole $\bar{D}$. This observation will be useful later for describing the behavior of the layering field at the boundary (see Corollary \ref{cor:layering_field_bdry}). We also note that, even though we can prove our main theorems only for bounded domains, the following proposition can be stated for some unbounded domains too.  

     \begin{prop}\label{prop:alpha_near_bdry}
         Let $\alpha_{\delta, D}(z)$ and $\bar{\alpha}_{\delta, D}(z)$ be as in \eqref{eq:loop_sets} and \eqref{eq:loop_sets_disk}.
         \begin{itemize}
             \item[(i)] For each $r>0$, there are finite constants $C_{\bH}(r), \bar{C}_{\bH}(r)>0$ such that for all $z \in \bH$,             \begin{align}\label{eq:alpha_H_near_bdry_0}
                 \alpha_{r\cdot d_z, \bH}(z) = C_{\bH}(r),  \text{ and } \bar{\alpha}_{r\cdot d_z, \bH}(z) = \bar{C}_{\bH}(r). 
             \end{align}
             Moreover, $C_{\bH}(r), \bar{C}_{\bH}(r) \sim \log(1/r)$ as $r \to 0$ and $C_{\bH}(r), \bar{C}_{\bH}(r) \downarrow 0$ as $r \to \infty$.

             \item[(ii)] For any point $z_0$ on the boundary of the unit disk $\D$ and any sequence $\{z_n\}_n$ in $\D$ converging to $z_0$, we have $\limsup_{n \to \infty} \alpha_{d_{z_n}, \D}(z_n) = C_{\D} < \infty$ and the constant $C_{\D}$ does not depend on $z_0$.    

             \item[(iii)] Let $D\subset \C$ be a bounded simply connected domain with smooth boundary. For any point $z_0$ on the boundary $\partial D$ of $D$ and a sequence $\{z_n\}_n$ in $D$ converging to $z_0$, we have, $ c_D(z_0) = \limsup_{n \to \infty} \alpha_{d_{z_n}, D}(z_n) <\infty$. Moreover, the constants $c_D(z_0)$ are bounded away from $0$ and $+\infty$ over $\partial D$.
          \end{itemize}
     \end{prop}

    The proof is postponed to Appendix \ref{app:proof_prop:alpha_near_bdry}.     Finally, let $\alpha_{D}(z, w)$ be the $\mu_D$-mass of all loops in $D$ covering both points $z$ and $w$. From Lemma 4.4 of \cite{cgpr21} we know that $\alpha_D$ blows up logarithmically at the diagonal. In other words, the function
     \begin{align}
             g(z, w) = \alpha_D(z, w) - \frac{1}{5} \log_+ \frac{1}{|z-w|}, \quad z, w \in D
         \end{align}
    is continuous on $D\x D$. Let us rephrase this in the following result. 
     \begin{lem}\label{lem:alpha_cont}
         The function $\alpha_D(\cdot, \cdot)$ on $D\x D$ is continuous away from the diagonal. 
     \end{lem}
     We mention that the crucial property for proving the above result is thinness of $\mu$. See \cite[p. 932]{cgpr21} for the relevant calculations.

	\subsection{The loop soup and its layering field}\label{sec:prelim_layering}

    Let $(\Omega, \mc{F}, \pr)$ be a probability space and $D$ a bounded simply connected domain in $\C$ with a $C^1$ boundary. Given a parameter $\lambda >0$, the \textit{Brownian loop soup} on $D$ is a Poisson point process $\mc{L}_{\lambda} =  \mc{L}_{\lambda, D}$ on simple loops in D (i.e. on $S(D)$) with intensity $\lambda \mu_D$.\footnote{The BLS is usually defined as a Poisson point process on  unrooted  (not necessarily simple) loops in $D$, having the intensity $\lambda\mu^{loop}_D$. But since our concern is with the layering by loops, which is only dependent on their outer boundaries, by an abuse of terminology, we refer to the point process on simple loops as the Brownian loop soup.  } More precisely, for each measurable subset $A$ of $S(D)$, $\mc{L}_{\lambda}(A)$ is Poisson random variable with intensity $\lambda \mu_D(A)$ and for disjoint measurable subsets $A_1, \ldots, A_k$, the random variables $\mc{L}_{\lambda}(A_1), \ldots, \mc{L}_{\lambda}(A_k)$ are independent. Note that, almost surely one can find loops $\gamma_1, \gamma_2, \ldots \in S(D)$ such that $\mc{L}_{\lambda} = \sum_{i =1}^{\infty} \delta_{\gamma_i}$ and we write $\gamma \in \mc{L}_{\lambda}$ if $\gamma = \gamma_i$ for some $i$. 
	
	We can now define the layering field on $D$. A loop $\gamma \in S(D)$ is said to \textit{cover} or \textit{layer} a point $z \in D$ if $z \in \bar{\gamma}$. Given a BLS $\mc{L}_{\lambda}$, to each loop $\gamma \in \mc{L}_{\lambda}$ we now assign a sign $X_{\gamma}$, which are i.i.d. random variables taking values $+ 1$ or $-1$, each with probability $1/2$. Equivalently, on $\Omega$ we can define two independent Brownian loop soups, $\mc{L}_{\lambda/2}^+$ and $\mc{L}_{\lambda/2}^-$, each with intensity $\frac{\lambda}{2} \mu_D$ and assign the value $+1$ to all loops realized in the former point process and $-1$ to those in the latter. That these two descriptions are indeed equivalent can be seen from the relation   
    \begin{align}
       \sum_{\gamma \in \mc{L}_{\lambda}\cap A} X_{\gamma} \stackrel{d}{=} \mc{L}_{\lambda/2}^+(A) - \mc{L}_{\lambda/2}^-(A), \quad A \in \mc{S}(D).
    \end{align} 
    
    %For a loop $\gamma$ in $\mc{L}_{\lambda/2}^+$ (or $\mc{L}_{\lambda/2}^-$), we let $X_{\gamma} =+1$ (resp. $-1$). 
 
    The \textit{layering number} $N_{\lambda, D}(z)$ is the random variable that counts the number of signed loops in the BLS (with their signs) covering $z$. But due to the fact that the tiny loops covering $z$ have infinite mass (as is evident from the definition  \eqref{eq:loop_measure} of the Brownian loop measure), in almost every realization of $\mc{L}_{\lambda/2}^{\pm}$ there are infinitely many tiny loops covering every point of $D$. This leads to $N_{\lambda, D}(z)$ being almost surely undefined. To address this, we can remove these tiny loops from consideration by introducing a parameter $\delta>0$, the so-called \textit{ultraviolet cutoff}, and restrict ourselves to the collection $A_{\delta, D}(z)$ of loops of diameter at least $\delta$ which cover $z$. With these notions, finally we can define the main object of our study. 
	\begin{defn}\label{def:real_layering_field}
	   \begin{itemize}
          \item[(i)]  The \textit{layering number} of the BLS at $z \in D$ with ultraviolet cutoff $\delta>0$ is, 
	\begin{align}\label{eq:layering_number_uv}
		 N^{\delta}_{\lambda, D}(z) : = \sum_{\gamma \in \mc{L}_{\lambda} \cap A_{\delta, D}(z) } X_{\gamma} =   \mc{L}_{\lambda/2}^+ (A_{\delta, D}(z)) - \mc{L}_{\lambda/2}^-(A_{\delta, D}(z)).
	\end{align}  
    We note that, since almost surely there are only finitely many loops covering $z$ with diameter larger than $\delta$ (because $\alpha_{\delta, D}(z) = \mu_D(A_{\delta}(z)) < \infty$), the above quantity is almost surely finite. 
	       \item[(ii)] For $\beta \in \R$ and $\delta>0$, we define the \textit{(real) layering field} on $D$ as
		\begin{align}\label{eq:real_layering_field}
			V^{\delta}_{\beta, \lambda, D}(z) : = e^{\beta N^{\delta}_{\lambda, D}(z)}, \quad z \in D.
		\end{align}		
         We shall also use the following notations frequently in the sequel to denote the renormalized layering fields:
	\begin{align}\label{eq:real_LF_normalized}
		\tilde{V}^{\delta}_{\lambda,\beta, D} (z) = \delta^{2 \Delta(\lambda,\beta)} V^{\delta}_{\lambda,\beta, D} (z) = e^{\beta N^{\delta}_{\lambda, D}(z) + 2\Delta(\lambda, \beta) \log \delta} 
	\end{align}   	
	with $\Delta(\lambda, \beta) = \frac{\lambda}{10} (\cosh(\beta)-1)$. Note that $\Delta (\lambda, \beta) \ge 0$ for all $\lambda >0, \beta \in \R$.

        \item[(iii)] Similarly, when $0<\delta<R< d(z, \partial D)$ for some $z \in D$, we define
	\begin{align}\label{eq:real_LF_uv_ir}
		N^{\delta, R}_{\lambda, D}(z) = \sum_{\gamma \in \mc{L}_{\lambda}\cap A_{\delta, R, D}(z)} X_{\gamma} \text{ and } V^{\delta, R}_{\beta, \lambda, D} (z) = e^{\beta N^{\delta, R}_{\lambda, D}(z)}. 
	\end{align} 
    These are the layering number and the layering field at $z$ with the ultraviolet cutoff $\delta$ and infrared cutoff $R$. $\tilde{V}^{\delta, R}_{\beta, \lambda, D}$ is similarly as above.
	   \end{itemize}
	\end{defn}	   
    We will often ignore the parameters $\lambda$, $\beta$ and the domain $D$ when they are fixed and clear from the context. 

    In the following we provide yet another description of the layering numbers. This will help us connect the Poisson layering fields with Gaussian fields on $D$. Suppose $S_{\pm}(D) = \{\pm 1\}\x S(D)$ is the collection of all marked simple loops in $D$ with the product $\sigma$-algebra $\mc{S}_{\pm} (D)$. Then the \textit{loop soup} on $D$ with intensity parameter $\lambda >0$ can be thought of as a Poisson random measure (PRM) $N_{\lambda}$ on $S_{\pm}(D)$ with intensity $\nu_{\lambda} : = \frac{1}{2}(\delta_{+1} + \delta_{-1} ) \otimes \lambda \mu_D$. More precisely, for each measurable $A \subset S_{\pm}(D)$, $N_{\lambda}(A)$ is Poisson random variable with intensity $\nu_{\lambda}(A)$ and for disjoint sets $B_1, \ldots, B_k \in \mc{S}_{\pm} (D)$, the random variables $N_{\lambda}(A_1), \ldots, N_{\lambda}(A_k)$ are independent. It is easy to see that the layering number defined in \eqref{eq:layering_number_uv} is equivalent to the following representation. For $\delta \ge 0 $, $z \in D$ and a marked loop $x = (\epsilon, \gamma) \in S_{\pm} (D)$, let 
        \begin{align}\label{eq:locally_exploding_kernel}
            h^{\delta}_z(x) = h^{\delta}_z (\epsilon, \gamma) = \epsilon \ind_{A_{\delta}(z)}(\gamma)
        \end{align}
    be the \textit{locally exploding kernels}, according to the terminology introduced in \cite{cgpr21}. Note that, when $\delta =0$, we use the notation $h_z = h^0_z$.
    With these notations in hand, we can rewrite the layering number at $z$ with the ultraviolet cutoff $\delta>0$ as,
	\begin{align}\label{eq:layering_number_uv_1}
			N^{\delta}_{\lambda, D}(z)  = N_{\lambda} (h^{\delta}_z) = \int_{S_{\pm}(D)} h^{\delta}_z (x ) N_{\lambda}(dx).
	\end{align} 	    
	
	\subsection{The Gaussian layering field and the multiplicative chaos}\label{sec:prelim_glf_gmc}
	
	In this section we introduce the other main object of our study, viz. the Gaussian layering model. However, it will be beneficial to first recall the main ingredients of the theory of Gaussian multiplicative chaos. 
 
    As before, we assume that $D\subset \C$ is a bounded simply connected domain with smooth boundary. Let $G_D(z,w)$, $z,w \in D$, denote the Green's function on $D$. Recall that, at least formally, this can be defined as the expected local time at $w$ of the Brownian motion starting from $z \in D$, before it hits the boundary $\partial D$ of $D$. In notations, we can write 
    \begin{align}
        \int_{D'} G_D(z, w) \,dw = \frac{1}{2}\E_z \left[\int_0^{\tau_D} \ind\{B_s \in D'\} \, ds \right],
    \end{align}
    where $D'$ is an open subset of $D$, $(B_t)_{t \ge 0}$ is a planar Brownian motion started at $z$ and $\tau_D$ is the first time it exits $D$. It is known that $G_D$ is a symmetric function on $D\x D$ with a logarithmic singularity at $z=w$ (i.e. it blows up like a constant multiple of $\log (|z-w|^{-1})$ as $w \to z$). The Green's function can also be defined as the fundamental solution to the Laplace equation on $D$ with Dirichlet boundary condition. 
    
    The \textit{Gaussian free field} (GFF) $X$ on $D$ is usually defined as a generalized centered Gaussian process with the covariance given by the Green's function $G_D(\cdot, \cdot)$. More generally, we can define the GFF $X$ to be a Gaussian process with covariance given by 
	\begin{align}\label{eq:gff_cov}
		K(z, w) = \log_+ \frac{1}{|z-w|} + g(z,w), \quad z, w \in D
	\end{align}
	where $g: D\x D \to \R$ is continuous and bounded. The reader is referred to, for instance, the lecture notes \cite[Chapter 3]{wp21} for a rigorous introduction to the continuum GFF.
	
	%{\clpl Q. Do we not want to impose the Dirichlet boundary condition on $X$? In that case, the correlation function $K$ should have $G_D$, instead of $\log_+$. }
	
	The \textit{Gaussian multiplicative chaos} (GMC) on $D$, denoted hereafter by $M_{\xi, D}$, can formally be viewed as a random measure on $D$ defined as an exponential of the GFF $X$ (cf. \cite{rv14})
	\begin{align}\label{eq:GMC_formal}
		M_{\xi, D} (dz) =  e^{\xi X(z)} \, dz,
	\end{align}
	where $\xi>0$ is a parameter. Since the GFF $X$ can only be shown to exist as a generalized function (i.e. as a Schwartz distribution), it is not possible to rigorously make sense of its exponential.  
    To precisely define the GMC, one introduces a cutoff or an approximation of the GFF $X$. There are several equivalent ways of approximating the GFF, and we can fix one for the time being. Suppose that there is a sequence $\{X_n\}_{n \ge 1}$ of centered Gaussian processes on $D$ which approximates the GFF $X$ in a suitable sense. Then we can define the random measures $ M^n_{\xi, D }, n \ge 1$, on $\mc{B} (D)$ by,
	\begin{align}\label{eq:GMC_approx}
		M^n_{\xi, D} (dz) =  e^{\xi X_n(z) - \frac{\xi^2}{2}\E(X_n(z)^2)} \, dz.
	\end{align} 
    It was proved by Kahane that the above sequence of random measures converges almost surely to a random measure $M_{\xi, D}$ in the topology of weak converges, and that this limit is unique in the sense that it does not depend on the specific approximation of the GFF. This shows that the object defined in \eqref{eq:GMC_formal} can be given meaning, provided that we can find a suitable approximation of the process $X$. Moreover, Kahane showed that $M_{\xi, D}$ is non-degenerate (i.e. it does not vanish almost surely) if and only if $\xi < 2$. This is known as the \textit{sub-critical regime} of the GMC. In all other cases (i.e. in the \textit{critical}, $\xi =2$, and \textit{super-critical}, $\xi >2$, cases) the GMC $M_{\xi, D}$ is fully degenerate. Let us also mention here that, when $\xi <2$, $M_{\xi}$ has finite moments of order $p \in (0, 4/\xi^2)$, in the sense that 
    \begin{align}
        \E(M_{\xi}(E)^p) < +\infty, \text{  for all compact } E \subset D.
    \end{align}
    In particular, $M_{\xi}$ has finite second moments only when $\xi < \sqrt{2}$. For these reasons, in our main result (Theorem \ref{thm:main_2}) we restrict ourselves to this section of the subcritical regime. We refer the reader to the survey article \cite{rv14} for a thorough review of this area.  

    Now let us construct the Gaussian layering field. This is analogous to the Poissonian layering field from Definition \ref{def:real_layering_field}. Recall that $S(D)$ is the collection of all simple loops in $D$ with the $\sigma$-algebra $\mc{S}(D)$ and $S_{\pm} (D)$ is the set of all signed simple loops with the product $\sigma$-algebra $ \mc{S}_{\pm} (D)$. Let $G_0$ be a centered \textit{Gaussian random measure} (or, white noise) on $S(D)$ with \textit{control} $\mu_D$. By this we mean that, for all $B \in \mc{S}(D)$, $G_0(B)$ is a centered Gaussian random variable and	
	$\E [G_0(B) G_0(C)] = \mu_D(B\cap C)$
	for all $B, C \in \mc{S}(D)$. Let $G$ denote a Gaussian random measure on $S_{\pm} (D)$ with control $\nu = \frac{1}{2}(\delta_{+1} + \delta_{-1}) \otimes \mu_D$.  This is the natural extension of $G_0$ to $S_{\pm} (D)$ and we are sometimes going to use the relationship,
    \begin{align}
        G(h^{\delta}_z) \stackrel{d}{=} G_0(A_{\delta, D}(z)), \quad( \delta >0 ,  z \in D),
    \end{align}
    where $h^{\delta}_z$ was defined in \eqref{eq:locally_exploding_kernel}.
    
    \begin{defn}\label{def:GLF}
    For $\delta>0$ and $\xi \in \R$, the \textit{Gaussian layering field (GLF) on $D$ with ultraviolet cutoff $\delta$} is
	\begin{align}\label{eq:W_delta_def}
		W^{\delta}_{\xi}(z) = e^{\xi G(h^{\delta}_z)} = e^{\xi G_0 (A_{\delta}(z))}, \quad z \in D.
	\end{align}	
	With $\Delta_{\xi} = \xi^2/20$, we use the notation 
    \begin{align}\label{eq:W_delta_renorm_def}
        \tilde{W}^{\delta}_{\xi}(z) = \delta^{2\Delta_{\xi}} W^{\delta}_{\xi}(z) = e^{\xi G_0 (A_{\delta}(z)) + 2\Delta_{\xi} \log \delta}, \quad z \in D,
    \end{align}
    for the renormalized GLF.
    \end{defn}

	\subsection{Main results}\label{sec:prelim_main_results}
	
	We are now ready to state our main results. Recall from Definition \ref{def:real_layering_field} the real valued layering field of the Brownian loop soup in $D$, denoted by $V^{\delta}_{\lambda, \beta} = V^{\delta}_{\lambda, \beta, D}$. In the following, we identify the renormalized fields defined in \eqref{eq:real_LF_normalized} with the measures $\tilde{V}^{\delta}_{\beta}(z)\, dz$ on $D$. Let $M_+(D)$ be the space of all non-negative measures on $(D, \mc{B}(D))$.  We equip $M_+(D)$ with the topology of weak convergence of measures, i.e. for $\nu, \nu_n \in M_+(D)$, $n \ge 1$, $\nu_n \to \nu$ if for all continuous bounded function $\varphi: D \to \R$ (notation: $\varphi\in C_b(D)$), $\nu_n(\varphi) = \int_D \varphi \, d\nu_n \to \int_D \varphi \, d\nu = \nu(\varphi)$ as $n \to \infty$. Roughly speaking, taken together, the next results show that as $\delta \to 0, \beta \to 0$ and $\lambda \to \infty$ the random measures $\tilde{V}^{\delta}_{\lambda, \beta}$ converges to a "tilted" Gaussian multiplicative chaos on $D$. 
	
	\begin{thm}\label{thm:main_1}
		Let $\delta>0$, $\lambda >0$, $\beta \in \R$ and $D$ be a bounded simply connected domain in $\C$. If $\Delta(\lambda, 2\beta) = \frac{\lambda}{10}(\cosh (2\beta)-1)< 1$, there is a random Borel measure $V_{\lambda, \beta} =V_{\lambda, \beta, D}$ in $M_+(D)$ such that as $\delta \to 0$, $\tilde{V}^{\delta}_{\lambda, \beta, D} \to V_{\lambda, \beta}$ in probability as members of $M_+(D)$.   
	\end{thm}

    We shall call the random measure $V_{\lambda, \beta, D}$, obtained above, the \textit{(real) Poisson layering field on $D$}. The next theorem concerns the Gaussian random variables $\tilde{W}^{\delta}_{\xi, D}$ introduced in Definition 
    \ref{def:GLF}. It shows that we can obtain a well-defined object out of them by removing the ultraviolet cutoff $\delta$ and provides the relation of this limiting object with the Gaussian multiplicative chaos. As above, here $\tilde{W}^{\delta}_{\xi}$ are identified with the measures $\tilde{W}^{\delta}_{\xi}(dz) = \tilde{W}^{\delta}_{\xi}(z) \, dz$.
 
	\begin{thm}\label{thm:main_2}
		 Suppose $0<\xi<2$ and $D\subset \C$ is bounded simply connected domain with smooth boundary. Then the following hold.
		\begin{enumerate}
            \item There is a random element $W_{\xi} = W_{\xi,D}$ in $M_+(D)$, such that as $\delta \to 0$, $\tilde{W}^{\delta}_{\xi} \to W_{\xi}$ weakly in probability. 
            		
			\item $W_{\xi, D}$ is almost surely absolutely continuous with respect to the Gaussian multiplicative chaos $M_{\xi, D}$ on $D$. Moreover, the Radon-Nikodym derivative of $W$ with respect to $M$ is non-random and is given by the formula,
			\begin{align}\label{eq:W_M_main}
				\frac{d W_{\xi, D}}{d M_{\xi, D}}(z)  = e^{\frac{\xi^2}{2}\Theta_D(z)}, \quad z \in D,
			\end{align} 
			where $\Theta_D(z) = \frac{1}{5} \log d_z + \alpha_{ d_z, D}(z)$ and recall that $d_z = d(z, \partial D)$.
		\end{enumerate}			
	\end{thm}	
    
    $W_{\xi}$ will be called the \textit{(real) Gaussian layering field} (GLF).  Let us observe that, given Proposition \ref{prop:alpha_near_bdry}, the expression of the Radon-Nikodym derivative in \eqref{eq:W_M_main} implies that the field $W_{\xi, D}$ vanishes near the boundary of $D$. This should be compared with the analogous expression \cite[eq. (1.6)]{cgpr21} which appears in the case of imaginary layering field. In that situation, again because of the behavior of $\Theta_D$ near $\partial D$, the imaginary layering field blows up with respect to the imaginary GMC.
    \begin{cor}\label{cor:layering_field_bdry}
        The right side of \eqref{eq:W_M_main} converges to $0$ as $z \to \partial D$. Thus,  on $\partial D$, the Gaussian layering field $W_{\xi, D}$ is singular with respect to the GMC .
    \end{cor}
    
     In the above statement, we only require that $z$ goes towards the boundary $\partial D$ of $D$, in the sense that $d(z, \partial D) \to 0$. This is more general than the assumption that $z \to z_0$ for some point $z_0 \in \partial D$. Since $\bar{\D}$ is compact, all sequences have convergent subsequences, and thus the statement follows from Proposition \ref{prop:alpha_near_bdry}. Finally, we state our main result that links the Poisson and the Gaussian  layering fields. 
    
    \begin{thm} \label{thm:main_3}
        Let $0<\xi< \sqrt{2}$ be fixed and $D \subset \C$ be bounded and dimply connected. As $\lambda \to \infty$ and $\lambda \beta^2 \to \xi^2$, the Poisson layering fields $V_{\lambda, \beta, D}$ converge weakly to the Gaussian layering field $W_{\xi, D}$.  
			%\begin{align}
			%	V_{\lambda, \beta, D} \implies W_{\xi, D} \text{  as } \lambda \to \infty \text{ and } \lambda \beta^2 \to \xi^2.
			%\end{align}			
    \end{thm}
    By \textit{weak convergence} here we mean that the convergence is in the sense of finite-dimensional distributions, i.e. if $\varphi_1, \ldots, \varphi_n \in C_b(D)$, we have,
			\begin{align}\label{eq:f.d.d.}
				(V_{\lambda, \beta}(\varphi_1), \ldots, V_{\lambda, \beta}(\varphi_n)) \implies (W_{\xi}(\varphi_1), \ldots, W_{\xi}(\varphi_n)),
			\end{align}
	as $\lambda \to \infty$, $\lambda \beta^2 \to \xi^2$.

    \begin{rem}
        Before finishing this section, we highlight the important difference between the hypotheses of the above theorems regarding the parameter $\xi$. Since we will exploit the standard multiplicative chaos theory to show the existence of the Gaussian layering field, we can state Theorem \ref{thm:main_2} for the whole subcritical range $\xi \in (0,2)$. On the other hand, Theorem \ref{thm:main_3} can only be stated for the range $\xi \in (0, \sqrt{2})$ as the technique of chaos expansion requires that the fields be bounded in $L^2$. 
    \end{rem}
   	
   	\section{Existence of the Poisson and Gaussian layering fields}\label{sec:PLF_GLF_exist}

    This section is dedicated to the proofs of Theorem \ref{thm:main_1} and Theorem \ref{thm:main_2}. As our treatments of the Poisson and the Gaussian cases are similar, some aspects of their existence proofs can be unified. To this end, we explicitly  calculate the one and two-point functions of $V^{\delta}_{\beta}= V^{\delta}_{\beta, \lambda, D}$ and $W^{\delta}_{\xi} = W^{\delta}_{\xi, D}$ and also discuss their ultraviolet limits (i.e. when $\delta \to 0$). Using these observations, we prove the existence of $V_{\lambda, \beta, D}$ in Section \ref{sec:PLF_exists} and and that of $W_{\xi, D}$ in Section \ref{sec:GLF_exists}. Since the intensity $\lambda >0$ and the domain $D$ are fixed in this section, let us ignore them whenever possible. 
	
	\subsection{Correlation functions} \label{sec:corr_fns}
	
	First, let us consider the Poisson fields. Recall that $\Delta(\lambda, \beta) = \Delta (\beta) = \frac{\lambda}{10}(\cosh(\beta)-1)$ and we introduce the following new notations: $d_{z,w} = d^D_{z,w} : = |z-w| \wedge d(z, \partial D)$ and $ d_{w,z} = d^D_{w,z} := |z-w| \wedge d(w, \partial D)$ for $z, w \in D$.
	
	\begin{lem}\label{lem:1,2point}
		Let $\beta \in \R$, $z, w \in D$ be two distinct points. Then,
		\begin{itemize}
			\item[(i)] For every $\delta>0$, we have,
			\begin{align}\label{eq:1_point_0}
				\E (V^{\delta}_{\beta}(z)) = & e^{\lambda  \alpha_{\delta, D} (z) (\cosh(\beta) -1) }= e^{10\Delta(\beta) \alpha_{\delta}(z)}.
			\end{align}
			And when $d_z = d(z, \partial D)> R> \delta>0$ we have,
			\begin{align}\label{eq:1_point_UV_IR}
				\E (V^{\delta, R}_{\beta}(z)) = & \left(R/\delta\right)^{2\Delta(\beta)}
			\end{align}
            where $V^{\delta, R}_{\beta}$ was defined in \eqref{eq:real_LF_uv_ir}.

            \item[(ii)] As $\delta \to 0$, the limit
			\begin{align}\label{eq:conv_1_pt_0}
				\langle V_{\beta}(z) \rangle_D : = \lim_{\delta \to 0}\E (\tilde{V}^{\delta}_{\beta}(z)) = d_z^{2\Delta(\beta)} e^{10\Delta(\beta) \alpha_{d_z, D}(z)}
			\end{align}
			exists.
			
			\item[(iii)] Suppose $|z-w|>\delta, \delta' >0$. Then for all $\beta, \beta' \in \R$ we have,
			\begin{align}\label{eq:2point_0}
				\E \left( V^{\delta}_{\beta}(z) V^{\delta'}_{\beta'}(w)\right) 
				= e^{ \lambda \alpha_{\delta}(z|w)(\cosh(\beta) -1)} e^{ \lambda \alpha(z,w) \left(\cosh(\beta+\beta')-1\right)} e^{\lambda\alpha_{\delta'}(w|z)(\cosh(\beta') -1) }
			\end{align}		
            The notations $\alpha_{\delta}(z|w), \alpha(z,w) = \alpha_{|z-w|}(z, w)$ etc were defined in \eqref{eq:loop_sets}.

            \item[(iv)] The limit
			\begin{align}\label{eq:conv_2_pt_0}
				&\langle V_{\beta}(z) V_{\beta}(w) \rangle_D := \lim_{\delta, \delta' \to 0} \E( \tilde{V}^{\delta}_{\beta}(z) \tilde{V}^{\delta'}_{\beta}(w)) 			
				\nonumber \\
				= & (d_{z,w}d_{w,z})^{2 \Delta(\beta)}	 e^{10 \Delta(2\beta)  \alpha(z,w) } e^{10 \Delta(\beta)\alpha_{d_{z,w}}(z|w)} e^{10 \Delta(\beta)\alpha_{d_{w,z}}(w|z)}			
			\end{align}
			exists finitely. 
		\end{itemize}		       
	\end{lem}	
	
	The proofs of the above formulae are in Appendix \ref{app:proof_lem:1,2point}. We next show that the convergence obtained in Lemma \ref{lem:1,2point} for the 2-point function holds in $L^1(D\x D)$ when $\Delta(2\beta, \lambda)<1$. Since $\Delta(2\beta, \lambda) \to  \frac{\xi^2}{5}$ 
        %\frac{\lambda}{10} \left( \frac{(2\beta)^2}{2} + \frac{(2\beta)^4}{4!} + \cdots\right) \to \frac{2}{10} \xi^2  = \frac{\xi^2}{5} $ 
        as $\lambda \to \infty, \lambda \beta^2 \to \xi^2$, this condition therefore requires $0<\xi<\sqrt{5}$, and in particular covers the range mentioned in Theorems \ref{thm:main_2}  and \ref{thm:main_3}. (See Remark 1.2 in \cite{cgpr21} for a comment on this restriction.) 	
	\begin{lem}\label{lem:L^1_conv}
		Assume that $\beta \in \R$ is such that $\Delta(2\beta)  < 1$. As $\delta,\delta' \to 0$, we have
		\begin{align}\label{eq:L^1_conv_0}
			\int_D \int_D \, dz \, dw\, |\E ( \tilde{V}^{\delta}_{\beta}(z) \tilde{V}^{\delta'}_{\beta}(w)) - \langle V_{\beta}(z) V_{\beta}(w) \rangle | \to 0.
		\end{align}
	\end{lem}
	   The above is reminiscent of \cite[Theorem 3.3]{cgpr21} and \cite[Proposition 5.3]{vcl18} and can be proved similarly. We nevertheless present its proof in Appendix \ref{app:proof_lem:L^1_conv} for the reader's convenience. 
       
       Although we do not need to consider general $n$-point functions for the proof of our main results, for the sake of completeness we include their formulae in the following proposition and show that they satisfy a conformal covariance property. The proof follows the same line of arguments as that of \cite[Theorem 4.1]{cgk16} and is included in Appendix \ref{app:proof_prop:n_point}.
	
	\begin{prop}\label{prop:n_point}
	(i)	For $n \ge 2$, $\beta_1, \ldots, \beta_n \in R$ and $z_1, \ldots, z_n \in D$ the limit
			\begin{align}\label{eq:conv_n_pt_0}
				\lim_{\delta \to 0} \E\left(\prod_{j=1}^{n} \tilde{V}^{\delta}_{\beta_j, D}(z_j)\right)=\lim_{\delta \to 0} \delta^{2\sum_{j=1}^n \Delta(\beta_j)} \E\left(\prod_{j=1}^n V^{\delta}_{\beta_j , D}(z_j)\right) = : \phi_D(z_1, \ldots, z_n; \bm{\beta})
			\end{align}
			exists finitely, with $\bm{\beta} = (\beta_1, \ldots, \beta_n)$ 
   
     (ii)  Let $D, D' \subset \C$ be two conformally equivalent bounded open domains and suppose $f: D \to D'$ is a conformal map. Also let $\bm{\beta} = (\beta_1, \ldots, \beta_n)$, $z_i \in D$ and $z_i':= f(z_i) \in D'$ with $i =1, \ldots, n$. Then, 
		\begin{align} \label{eq:conf_cov_n_pt}
			\phi_{D'}(z'_1, \ldots, z'_n;\bm{\beta}) = \left(\prod_{j=1}^n |f'(z_j)|^{2\Delta(\beta_j)}\right) \phi_D(z_1, \ldots, z_n;\bm{\beta}) .
		\end{align}		
	\end{prop}	
	
    Let us now turn to the Gaussian layering fields $\tilde{W}^{\delta}_{\xi}$ with parameter $\xi$ and ultraviolet cutoffs $\delta>0$ defined in \eqref{eq:W_delta_renorm_def}. In the following proposition we collect various useful facts about these fields, including their one and two point functions. These are going to be useful when we prove that $\tilde{W}^{\delta}_{\xi}$ converges to a certain random measure $W_{\xi}$ in the ultraviolet limit (Theorem \ref{thm:main_2}). Although the computations and the proofs of these results are routine, for the sake of completeness we have included them in Appendix \ref{app:proof_prop_GLF_facts}.     
    
	\begin{prop}\label{prop:GLF_facts}
		Let $\xi > 0$ be such that $\Delta_{\xi} = \xi^2/20 < 1/4$. Then, 
		\begin{itemize}
			\item[(i)] For every $z \in D$, as $\delta \to 0$, we have the convergence,
			\begin{align} \label{eq:W_1_pt_1}
				\E[\tilde{W}^{\delta}_{\xi}(z)] = \delta^{2\Delta_{\xi}} e^{\frac{\xi^2}{2} \alpha_{\delta, D} (z) }\to d_z^{2\Delta_{\xi} } e^{\frac{\xi^2}{2} \alpha_{d_z, D}(z) } = : \langle W_{\xi}(z)\rangle.
			\end{align}  

            \item[(ii)] For $z, w \in D$, %$z \neq w$ and $d_{z,w} = |z-w| \wedge d(z, \partial D)$, $d_{w,z} = |z-w| \wedge d(w, \partial D)$,
            we have
            \begin{align}\label{eq:W_2_pt}
                \lim_{\delta, \delta' \to 0} \E (\tilde{W}^{\delta}_{\xi}(z) \tilde{W}^{\delta'}_{\xi}(w)) = &(d_{z,w}d_{w,z})^{2\Delta_{\xi}} e^{\frac{\xi^2}{2}  \alpha_{d_{z,w},D}(z) } e^{\frac{\xi^2}{2}  \alpha_{d_{w,z}, D}(w)} e^{\xi^2  \alpha_{|z-w|,D}(z,w)}
			\nonumber \\
			= :& \langle W_{\xi}(z) W_{\xi}(w)\rangle.
            \end{align}
            \item[(iii)] As $\delta, \delta' \to 0$ we have,
            \begin{align}\label{eq:L^1_conv_W_0}
                \int_{D\x D} \, dz \, dw \, |\E (\tilde{W}^{\delta}_{\xi}(z) \tilde{W}^{\delta'}_{\xi}(w)) - \langle W_{\xi}(z) W_{\xi}(w)\rangle | \to 0.
            \end{align}
          \end{itemize}		
    \end{prop}
	
	\subsection{Proof of Theorem \ref{thm:main_1}}\label{sec:PLF_exists}
	
	Here we will show that the layering fields $\tilde{V}^{\delta}_{\beta}$, with the ultraviolet cutoff $\delta>0$, converge to a random measure $V_{\beta}$ as $\delta \to 0$. But first, we record the simple fact that, for every measurable subset $B$ of $D$, as $\delta\to 0$ the random variables $\tilde{V}^{\delta}_{\beta}(B)$ converge in mean. 
	
	\begin{lem}\label{lem:mean_conv}
		For all $B \in \mc{B}(D)$, 
		\begin{align}\label{eq:mean_conv}
			\lim_{\delta \to 0} \E (\tilde{V}^{\delta}_{\beta}(B)) = \int_B d_z^{2\Delta(\beta)} e^{10\Delta(\beta) \alpha_{d_z, D}(z)} \, dz <\infty.
		\end{align} 
	\end{lem}
	
	\begin{proof}
		By Fubini's theorem, \eqref{eq:1_point_0} and \eqref{eq:conv_1_pt_0} we have, for all $\delta>0$, 
		\begin{align}
			\E (\tilde{V}^{\delta}_{\beta}(B)) 
            %=  \int_B \E( \tilde{V}^{\delta}_{\beta}(z)) \, dz 
			=  \int_{B \cap \{d_z > \delta\}} d_z^{2\Delta(\beta)} e^{10 \Delta(\beta) \alpha_{d_z, D}(z)} dz + \int_{B \cap \{d_z \le \delta\}} \delta^{2\Delta(\beta)} e^{10 \Delta(\beta) \alpha_{\delta, D}(z)} dz.
		\end{align}
		Thus, using \eqref{eq:mu_D_size} we get,
		\begin{align}
			\left| \E (\tilde{V}^{\delta}_{\beta}(B)) - \int_{B \cap \{d_z > \delta\}} d_z^{2\Delta(\beta)} e^{10 \Delta(\beta) \alpha_{d_z, D}(z)} dz \right| 			
			\le & \int_{B \cap \{d_z \le \delta\}} \delta^{2\Delta(\beta)} e^{10 \Delta(\beta) \alpha_{\delta, D}(z)} dz
			\nonumber \\
			%\le & \int_{B \cap \{d_z \le \delta\}} \delta^{2\Delta(\beta)} e^{10 \Delta(\beta) \alpha_{\delta, d(D), \C}(z)} dz
			%\\
			\le & \int_{B \cap \{d_z \le \delta\}} \delta^{2\Delta(\beta)} e^{10 \Delta(\beta) \frac{1}{5} \log \frac{d(D)}{\delta}} dz
			\nonumber \\
			= & d(D)^{2\Delta(\beta)} |B \cap \{ z \in D \mid d_z \le \delta\}|.
		\end{align}
		As $\delta \to 0$ the r.h.s. in the above converges to $0$. This gives the expression for the limit in \eqref{eq:mean_conv}. It is easy to see that it is finite, since $D$ is bounded and for $B \in \mc{B}(D)$ we have
		\begin{align}
			\int_B d_z^{2\Delta(\beta)} e^{10\Delta(\beta) \alpha_{d_z, D}(z)} \, dz \le \int_B d_z^{2\Delta(\beta)} e^{10\Delta(\beta) \frac{1}{5} \log \frac{d(D)}{d_z}} \, dz = d(D)^{2\Delta(\beta)} |B| < \infty,
		\end{align}
         again using \eqref{eq:mu_D_size} to get the upper bound.
	\end{proof}

    The next proposition finishes the proof of Theorem \ref{thm:main_1}.     
    
	\begin{prop}\label{prop:layering_field_conv}
		Let $D \subset \C$ be a bounded simply connected domain with smooth boundary. Also suppose $\beta \in \R$ is such that $\Delta(2\beta) < 1$. Then, for all $B \in \mc{B}(D)$,
		\begin{align}\label{eq:layering_field_conv_0}
			\lim_{\delta \to 0}\tilde{V}^{\delta}_{\beta}(B) =: V_{\beta}(B) 
		\end{align}
		exists in $L^2(\Omega)$ and is finite. Moreover, the above defines a random element $V_{\beta} = V_{\beta, \lambda, D}$ in $M_+(D)$, such that in probability 
		\begin{align}\label{eq:layering_field_conv_1}
			\tilde{V}^{\delta}_{\beta} \to V_{\beta} \text{, as } \delta \to 0
		\end{align}
		as elements of $M_+(D)$ (i.e. in the topology of weak convergence).
	\end{prop}
	
	\begin{proof}	
		   The first part of this proof borrows the arguments from \cite[Theorem 3.3]{cgpr21} and \cite[Theorem 5.1]{vcl18}. Fix a set $B \in \mc{B}(D)$ and note that for each fixed $\delta>0$, $\tilde{V}^{\delta}_{\beta}(B)$ has finite second moment. Let us prove that $\{\tilde{V}^{\delta}_{\beta}(B)\}_{\delta}$ is Cauchy in $L^2(\Omega)$. To this end, let $\delta, \delta'>0$ and note that we can apply Fubini's theorem to get,
		\begin{align}\label{eq:V_delta_cauchy_1}
			 \E \left[  (\tilde{V}^{\delta}_{\beta} (B) - \tilde{V}^{\delta'}_{\beta} (B) )^2\right]
			%\nonumber \\			
			%= & \E \left[\left( \int_D (\tilde{V}^{\delta}_{\beta}(z) - \tilde{V}^{\delta'}_{\beta}(z)) \ind_B(z) \, dz   \right)^2\right]
			%\nonumber \\
			= & \E \left[ \int_D \int_D (\tilde{V}^{\delta}_{\beta}(z) - \tilde{V}^{\delta'}_{\beta}(z)) (\tilde{V}^{\delta}_{\beta}(w) - \tilde{V}^{\delta'}_{\beta}(w)) \ind_B(z) \ind_B(w)  \, dz  \, dw \right]
			\nonumber \\
			= &  \int_D \int_D \E [(\tilde{V}^{\delta}_{\beta}(z) - \tilde{V}^{\delta'}_{\beta}(z)) (\tilde{V}^{\delta}_{\beta}(w) - \tilde{V}^{\delta'}_{\beta}(w))] \ind_B(z) \ind_B(w)  \, dz  \, dw 
			\nonumber \\
			\le &    \int_D \int_D | \E [(\tilde{V}^{\delta}_{\beta}(z) - \tilde{V}^{\delta'}_{\beta}(z)) (\tilde{V}^{\delta}_{\beta}(w) - \tilde{V}^{\delta'}_{\beta}(w))] | \, dz  \, dw. 
		\end{align}
		As $\delta, \delta' \to 0$, by Lemma \ref{lem:L^1_conv}
		\begin{align}\label{eq:V_delta_cauchy_2}
			&\int_D \int_D |\E [(\tilde{V}^{\delta}_{\beta}(z) - \tilde{V}^{\delta'}_{\beta}(z)) (\tilde{V}^{\delta}_{\beta}(w) - \tilde{V}^{\delta'}_{\beta}(w))]| \, dz  \, dw
			\nonumber \\
			= &  \int_D \int_D | \E( \tilde{V}^{\delta}_{\beta}(z) \tilde{V}^{\delta}_{\beta}(w)) - \E(\tilde{V}^{\delta'}_{\beta}(z)\tilde{V}^{\delta}_{\beta}(w) )- \E(\tilde{V}^{\delta}_{\beta}(z)\tilde{V}^{\delta'}_{\beta}(w)) + \E(\tilde{V}^{\delta'}_{\beta}(z) \tilde{V}^{\delta'}_{\beta}(w))  | \, dz  \, dw
			\nonumber \\
			\le & \int_D \int_D \left[ | \E( \tilde{V}^{\delta}_{\beta}(z) \tilde{V}^{\delta}_{\beta}(w)) - V_{\beta}(z,w)| + |\E(\tilde{V}^{\delta'}_{\beta}(z)\tilde{V}^{\delta}_{\beta}(w) ) - V_{\beta}(z,w) | \right.
			\nonumber \\
			& \hspace{3cm} \left. + | \E(\tilde{V}^{\delta}_{\beta}(z)\tilde{V}^{\delta'}_{\beta}(w))  - V_{\beta}(z,w) | + |\E(\tilde{V}^{\delta'}_{\beta}(z) \tilde{V}^{\delta'}_{\beta}(w)) - V_{\beta}(z,w) | \right]\, dz  \, dw 
			\nonumber \\
			\to & 0.
		\end{align}	
		Since $L^2(\Omega)$ is complete, there is a random variable $V_{\beta}(B)$ satisfying \eqref{eq:layering_field_conv_0}. From Lemma \ref{lem:mean_conv} we can observe that this random variable has the mean:
		\begin{align}\label{eq:layering_field_conv_3}
			\E (V_{\beta}(B)) = \int_B d_z^{2\Delta(\beta)} e^{10\Delta(\beta) \alpha_{d_z, D}(z)} \, dz.
		\end{align}
		
		\iffalse
		Since \eqref{eq:layering_field_conv_1} holds for any $f \in L^{\infty}(D)$, it holds for $f = \ind_B$ for any $B \in \mc{B}(D)$ and in particular for $\ind_{D}$ itself. We have that,
		\begin{align}
			\tilde{V}^{\delta}_{\beta} (B) = \tilde{V}^{\delta}_{\beta} (\ind_B) \to V(B), \text{ in } L^2(\Omega)
		\end{align}
		as $\delta \downarrow 0$. Clearly $V$ defines a measure on $D$ {\clpl (Clearly $V$ is finitely additive. Need to prove its countable additivity.)}. Next we argue that this $V$ is indeed a random measure on $D$ i.e. a $M_+(D)$-valued random variable.
		\fi
		
		We next want to show that the limit in \eqref{eq:layering_field_conv_0} uniquely defines an element of $M_{+}(D)$ and that \eqref{eq:layering_field_conv_1} holds in $M_+(D)$. Although this is standard in the literature, %; see for example \cite[p. 621-622]{rv10}. 
        for the sake of completeness we include a detailed proof by adapting the argument from \cite[Section 6]{bersetycki17}. Let us define the following $\pi$-system of Borel subsets of $D$,
		\begin{align}
			\mc{R} = \{ (a_1, b_1] \x (a_2, b_2], [a_1, b_1) \x [a_2, b_2) \subset D \mid a_i< b_i \text{ and } a_i, b_i \in \Q \text{ for } i =1,2\} 
		\end{align} 
		and note that $\sigma(\mc{R}) = \mc{B}(D)$. We now claim that, almost surely	\begin{align}\label{eq:layering_field_conv_7}
			V_{\beta}(D) = \sup\{ V_{\beta}(\cup_{i=1}^k B_i ) \mid B_1, \ldots, B_k \in \mc{R} \text{ disjoint}\}.
		\end{align}
		Note that, for every $\delta>0$, the measure $\tilde{V}^{\delta}_{\beta}$ is monotone: if $B \subset C$ then $\tilde{V}^{\delta}_{\beta}(B) \le \tilde{V}^{\delta}_{\beta}(C)$. Therefore, so is $V_{\beta}$. Hence the above LHS is greater than the r.h.s. a.s. Let ($D_n$, $n \ge 1$) be an increasing sequence of subsets of $D$ such that $D = \cup_n D_n$ and that each $D_n$ is a finite union of disjoint members of $\mc{R}$. Using \eqref{eq:layering_field_conv_3} and the monotone convergence theorem we obtain
		\begin{align}
			\E (V_{\beta}(D_n)) \to \E (V_{\beta}(D)), \text{ as } n \to \infty.
		\end{align}
		This shows that both sides of \eqref{eq:layering_field_conv_7} are equal a.s.
		
		Since $\mc{R}$ is countable and \eqref{eq:layering_field_conv_0} holds for each $B \in \mc{R}$, by a diagonalization argument we can find a sequence $\delta_k \downarrow 0$ such that, for all $B \in \mc{R}$,
		\begin{align}\label{eq:layering_field_conv_5}
			\tilde{V}^{\delta_k}_{\beta} (B) \to V_{\beta}(B), \text{ as } k \to \infty, 
		\end{align}
		outside of a null event $\mc{N} \subset \Omega$. The r.h.s. of the above can be extended additively: if $B_1, B_2 \in \mc{R}$ are disjoint then we let $V_{\beta}(B_1\cup B_2) = V_{\beta}(B_1) + V_{\beta}(B_2)$. Also we can assume without loss of generality that $\tilde{V}^{\delta_k}_{\beta}(D) \to V_{\beta}(D)$ on $\Omega\setminus \mc{N}$. 		
		
		Let us use this to show that for every $\omega \in \Omega\setminus \mc{N}$ the family $\{\tilde{V}^{\delta_k}_{\beta}(\omega)\}_{k \in \N}$ is tight in the space $M_{+}(D)$. Fix a $\omega \in \Omega\setminus \mc{N}$. We know that \eqref{eq:layering_field_conv_5} holds for this $\omega$ and in particular, this implies that $V_{\beta}(\omega)(D)  < \infty$. By \eqref{eq:layering_field_conv_7}, for a given $\epsilon>0$ we can find a set $D' = D'(\omega, \epsilon)$ which can be written as a finite union of sets belonging to $\mc{R}$, such that $\bar{D'} \subset D$ and which has the property that,
		\begin{align}
			V_{\beta}(\omega) (D \setminus D') : = V_{\beta}(\omega)(D) - V_{\beta}(\omega)(D')< \epsilon.
		\end{align}
		By our assumptions, $\tilde{V}^{\delta_k}_{\beta}(\omega)(D') \to V_{\beta}(\omega)(D')$ as $k \to \infty$. Therefore we can find $K = K (\omega, \epsilon) \in \N$ such that 
		\begin{align}
			\sup_{k \ge K} \tilde{V}^{\delta_k}_{\beta}(\omega) (D\setminus D')  \le \epsilon.
		\end{align}
		This proves tightness of $\{\tilde{V}^{\delta_k}_{\beta}(\omega)\}_{k}$.
		
		Because of tightness and the fact that $V_{\beta}(\omega)(D) < \infty$, by Prokhorov's theorem (cf. \cite[Theorem 8.6.2]{bogachev07}) we can find limit points of $\{ \tilde{V}^{\delta_k}_{\beta}(\omega)\}_k$ in the weak convergence topology. Suppose $V'_{\beta}(\omega)$ and $V''_{\beta}(\omega)$ are two such limit points. Note that we clearly have, $V'_{\beta}(\omega)(B) = V''_{\beta}(\omega)(B)$ for all $B \in \mc{R} $ by construction. Since $\mc{R}$ generates the Borel $\sigma$-algebra on $D$, we have $V'_{\beta}(\omega) = V''_{\beta}(\omega)$. We shall denote this limit by $V_{\beta}(\omega)$. Clearly, $V_{\beta}$ is a $M_+(D)$-valued random variable as each $\tilde{V}^{\delta_k}_{\beta}$ is so. This, together with \eqref{eq:layering_field_conv_0}, proves the assertion \eqref{eq:layering_field_conv_1}. 	
	\end{proof}

    \subsection{Proof of Theorem \ref{thm:main_2}}\label{sec:GLF_exists}

    (1) Fix $\varphi \in C_b(D)$. We want to show that 
    \begin{align}\label{eq:W_conv_0}
			W_{\xi,D}(\varphi) = \lim_{\delta \to 0} \tilde{W}^{\delta}_{\xi, D} (\varphi) \text{ in } L^2(\Omega).
    \end{align}	
	 For this, it is enough to prove that $\{\tilde{W}^{\delta}_{\xi}(\varphi)\}_{\delta>0}$ is a Cauchy sequence in $L^2(\Omega)$. By \eqref{eq:L^1_conv_W_0} we have,  
		\begin{align}\label{eq:W_conv_8}
			\E \left[(\tilde{W}^{\delta}_{\xi} (\varphi) - \tilde{W}^{\delta'}_{\xi}(\varphi))^2\right]
			%\nonumber \\
			%= & \E \left[ \left( \int_D (\tilde{W}^{\delta}_{\xi}(z) - \tilde{W}^{\delta'}_{\xi}(z)) \varphi(z) \, dz \right) ^2 \right]
			%\nonumber \\
			= &  \E \left[ \int_D \int_D (\tilde{W}^{\delta}_{\xi}(z) - \tilde{W}^{\delta'}_{\xi}(z)) (\tilde{W}^{\delta}_{\xi}(w) - \tilde{W}^{\delta'}_{\xi}(w)) \varphi(z)\varphi(w) \, dz \, dw \right]
			\nonumber \\
			\le & \lVert \varphi \rVert_{\infty}^2 \int_D \int_D |\E[(\tilde{W}^{\delta}_{\xi}(z) - \tilde{W}^{\delta'}_{\xi}(z)) (\tilde{W}^{\delta}_{\xi}(w) - \tilde{W}^{\delta'}_{\xi}(w))] |  \, dz \, dw 
			\nonumber \\
			\to & 0
		\end{align}
		as $\delta, \delta' \to 0$. This proves \eqref{eq:W_conv_0}. The rest of the proof is similar to that of Theorem \ref{thm:main_1}.

    (2) Let us define 
			\begin{align}\label{eq:M_approx}
				M^{\delta}_{\xi, D}(dz) := e^{\xi G(h^{\delta}_z) - \frac{\xi^2}{2} \E(G(h^{\delta}_z)^2)} \, dz
			\end{align} 
        and note here that for every $\delta>0$, by the definition of $W^{\delta}_{\xi}$ in \eqref{eq:W_delta_def} and that of $M^{\delta}_{\xi} = M^{\delta}_{\xi, D}$ above, we have
		\begin{align}\label{eq:W_M_1}
			\tilde{W}^{\delta}_{\xi}(z) =   e^{\xi G(h^{\delta}_z) - \frac{\xi^2}{2}\E(G(h^{\delta}_z)^2)} e^{2\Delta_{\xi} \log \delta + \frac{\xi^2}{2} \E(G(h^{\delta}_z)^2) }
			%\nonumber \\
			=  M^{\delta}_{\xi} (z) e^{2\Delta_{\xi} \log \delta + \frac{\xi^2}{2} \E(G(h^{\delta}_z)^2) }.
		\end{align}
		
		We first claim that $M^{\delta}_{\xi}$ converges to a non-degenerate Gaussian multiplicative chaos $M_{\xi}$ on $D$ as $\delta \to 0$, by showing that the correlation kernel of the underlying Gaussian field is of $\sigma$-positive type (see \cite[p. 321]{rv14}). Let $z \in D$ be fixed. Since $A_{\frac{1}{k+1}, \frac{1}{k}, D}(z), k \ge 1$ are pairwise disjoint collections of loops and $A_{\frac{1}{N}, D}(z) = \cup_{k=1}^{N-1} A_{\frac{1}{k+1}, \frac{1}{k}, D}(z)$, we can rewrite the random variable $G_0(A_{\frac{1}{N}}(z))$ as the sum of independent random variables, as follows
		\begin{align}
			G_0(A_{\frac{1}{N}}(z)) = \sum_{k=1}^{N-1} G_0 (A_{\frac{1}{k+1}, \frac{1}{k}}(z)).
		\end{align}
		By our assumptions, $\{G_0 (A_{\frac{1}{k+1}, \frac{1}{k}}(z))\}_{z \in D}$ is a Gaussian process with correlations:
		\begin{align}
			\E [ G_0 (A_{\frac{1}{k+1}, \frac{1}{k}}(z)) G_0 (A_{\frac{1}{k+1}, \frac{1}{k}}(w)) ] = \alpha_{\frac{1}{k+1}, \frac{1}{k}, D}(z,w)
		\end{align}
		for $z, w \in D$. Clearly the above is non-negative. Thus, to prove our claim that $M^{\delta}_{\xi} \to M_{\xi}$, we only have to argue that $(z, w) \mapsto \alpha_{\frac{1}{k+1}, \frac{1}{k}, D}(z,w)$ is continuous for every $k$. It is easy to see that as $w\to z$ we have $\alpha_{\frac{1}{k+1}, \frac{1}{k}, D}(z,w) \to \alpha_{\frac{1}{k+1}, \frac{1}{k}, D}(z)$, which is finite. And for $z \neq w$, the continuity follows from Lemma \ref{lem:alpha_cont}.
		
		It only remains to derive the form of the Radon-Nikodym derivative of $W_{\xi}$ with respect to $M_{\xi}$. From \eqref{eq:W_M_1} we have, for every $\delta>0$
		\begin{align}
			\frac{d \tilde{W}^{\delta}_{\xi}}{d M^{\delta}_{\xi}}(z) = e^{2\Delta_{\xi} \log \delta + \frac{\xi^2}{2} \E(G(h^{\delta}_z)^2) } = e^{2\Delta_{\xi}\log \delta  + \frac{\xi^2}{2} \alpha_{\delta, D}(z) } = e^{\frac{\xi^2}{2}\Theta_{\delta, D}(z)}
		\end{align}
		Since for small enough $\delta$, i.e. when $\delta < d_z$, we 
		we have $\Theta_{\delta, D}(z) =  \frac{1}{5} \log \delta +  \alpha_{\delta, D}(z) = \frac{1}{5} \log \delta+ \frac{1}{5} \log d_z -  \frac{1}{5} \log \delta +  \alpha_{d_z, D}(z) 
		=\frac{1}{5} \log d_z +  \alpha_{d_z, D}(z) = \Theta_D(z)$,  \eqref{eq:W_M_main} follows.        
	
	\section{Convergence of the Poisson layering field to the Gaussian layering field}\label{sec:PLF_to_GLF}

    The purpose of this section is to prove Theorem \ref{thm:main_2}. The main tool for this proof is the Wiener-It\^{o} chaos expansion and it  follows a similar line of arguments from \cite{cgpr21}. As with the previous section, the proofs of most of the auxiliary results are postponed to Appendix \ref{app:proofs_PLF_to_GLF}. 
	
	\subsection{Chaos expansions}	
	
	We will need the following notations to precisely state the Wiener-It\^{o} chaos expansion for the layering fields $V_{\beta, \lambda}$ and $W_{\xi}$. From Section \ref{sec:prelim_layering}, recall that $S_{\pm} = S_{\pm} (D)$ denotes the collection of all marked simple loops in $D$. Let $\mc{N} = \mc{N}(S_{\pm}(D), \mc{S}_{\pm}(D))$ denote the collection of all counting measures (i.e. integer valued $\sigma$-finite measures) on the measurable space $S_{\pm}$ and let $\mc{B}(\mc{N})$ be the Borel $\sigma$-algebra on $\mc{N}$ with respect to the topology of weak convergence. Then clearly the PRM $N_{\lambda}$ with intensity $\nu_{\lambda} = \frac{1}{2}(\delta_{-1} + \delta_{+1} ) \otimes \lambda \mu_D$ is a $\mc{N}$ valued random variable. By $\pr_{\lambda}$ we denote the law of $N_{\lambda}$ and by $\hat{N}_{\lambda} = N_{\lambda} - \nu_{\lambda}$ we denote the compensated PRM.
 
    We state the chaos decomposition result below for a general member of $L^2(\mc{N}, \pr_{\lambda})$. For a measurable function $Y: (\mc{N}, \mc{B}(\mc{N}))\to \R$, let us define the difference operator $D_x$ for every $x \in S_{\pm}$ as
	\begin{align}
		(D_x Y ) (\eta) : = Y (\eta  + \delta_x) - Y (\eta), \quad \eta \in \mc{N}.
	\end{align}
	For $q \ge 2$ and $x_1, \ldots, x_q \in A$, one defines by recursion,
	\begin{align}
		(D^q_{x_1, \ldots, x_q} Y )(\eta) : =[ D_{x_1}(D^{q-1}_{x_2, \ldots, x_q} Y)] (\eta), \quad \eta \in \mc{N}.
	\end{align}
    With these notations we have,
       
	\begin{prop}[The general Wiener-It\^{o} chaos decomposition] \label{prop:wi_chaos_poisson}
		 Suppose $Y \in L^2(\pr_{\lambda})$ i.e. $\E[(Y(N_{\lambda}))^2] < \infty$. Then there exists a sequence $f_q \in L^2(S_{\pm}^q, \mc{S}_{\pm}^q, \nu_{\lambda}^q)$, $q \ge 1$, of symmetric functions such that the following series converges in $L^2(\Omega, \pr)$,
		\begin{align}\label{eq:chaos_exp_PRM_gen}
			Y(N_{\lambda}) = \E [Y(N_{\lambda})] + \sum_{q=1}^{\infty} I^{N_{\lambda}}_q (f_q),
		\end{align} 
		where, for every $q \ge 1$, $I^{N_{\lambda}}_q$ is the $q$-fold Wiener-It\^{o} integral defined as,
		\begin{align}\label{eq:chaos_exp_PRM_gen_1}
			I^{N_{\lambda}}_q (f_q) = \int_{S_{\pm}^q} \ind\{ x_k \neq x_l, \forall k \neq l\} f_q(x_1, \ldots, x_q) \hat{N}_{\lambda}(dx_1)\cdots\hat{N}_{\lambda}(dx_q).
		\end{align}
        Moreover, the kernels $f_q$ have the following explicit expression,
        \begin{align}\label{eq:fq_diff_op}
			f_q(x_1, \ldots, x_q) = \frac{1}{q!} \E [(D^q_{x_1, \ldots, x_q} Y) (N_{\lambda})]
		\end{align}
		where $x_1, \ldots, x_q \in A$.
	\end{prop}
 
    We refer the reader to Theorem 18.10 and Eq. (18.8) of \cite{lp18} for the proof of the above result.
    %; see also \cite[Theorem 8.2.1]{pt11}. 
    For us, $Y$ will have the following specific form. Suppose a measurable map $h : S_{\pm} \to \R$ and a parameter $\beta \in \R$ are fixed. Then we will assume throughout this section that,
	\begin{align}
		Y(\eta) = Y_{h, \beta} (\eta) := e^{\beta \eta (h)}, \quad \eta \in \mc{N}
	\end{align}
	whenever the integral $\eta(h) = \int_{S_{\pm}} h(x) \eta (dx)$ exists. The following result can be easily obtained by induction.
	\begin{lem}
		For $Y$ as above and $\eta(h) < \infty$, we have
		\begin{align}\label{eq:diff_op_exp}
			(D^q_{x_1, \ldots, x_q} Y ) (\eta) = e^{\beta \eta(h)} \prod_{i=1}^q (e^{\beta h(x_i)}-1) = Y(\eta) (e^{\beta h (\cdot)}-1)^{\otimes q} (x_1, \ldots, x_q),
		\end{align}
		where $q \ge 1$, $x_1, \ldots, x_q \in S_{\pm}$.
	\end{lem}
 
	Now observe that, by their definition (see \eqref{eq:locally_exploding_kernel}), $h^{\delta}_z \in L^1(\nu_{\lambda})$. Therefore $N_{\lambda}(h^{\delta}_z ) < \infty$ a.s. and the layering field of the loop soup at the point $z$ with cutoff $\delta >0$
	\begin{align}
		Y^{\delta}(N_{\lambda}) := Y_{h^{\delta}_z, \beta} (N_{\lambda}) = \exp\left( \beta N_{\lambda}(h^{\delta}_z)\right)  = e^{\beta N^{\delta}(z)} = V^{\delta}_{\beta}(z)
	\end{align} is defined a.s. The above two results now imply the following representation.		\begin{lem}\label{lem:chaos_decomp_layering_uv}	For $\delta>0$ and $z \in D$ we have,	\begin{align}\label{eq:chaos_decomp_layering_uv}
			V^{\delta}_{\beta}(z) = e^{\beta N_{\lambda}(h^{\delta}_z)} = e^{10\Delta(\beta) \alpha_{\delta, D}(z) } \left( 1 + \sum_{q=1}^{\infty} \frac{1}{q!} I_q^{N_{\lambda}} \left[(e^{\beta h^{\delta}_z(\cdot)}-1)^{\otimes q}\right]\right).
		\end{align}
	\end{lem}

    \begin{proof}
        Note that, here $Y = Y_{h^{\delta}_z, \beta}$. By Lemma \ref{lem:1,2point}, $\E[Y(N_{\lambda})] = \E\left[ e^{\beta N_{\lambda}(h^{\delta}_z)}\right] = e^{10\Delta(\beta) \alpha_{\delta, D}(z) }$. Plugging the expression \eqref{eq:fq_diff_op} into \eqref{eq:chaos_exp_PRM_gen_1} we have, for $q\ge 1$
		\begin{align}
			I^{N_{\lambda}}_q (f_q) = & \frac{1}{q!} \int_{A^q} \E[(D^q_{x_1, \ldots, x_q} Y)(N_{\lambda})] \ind\{x_k \neq x_l \forall k\neq l\} \hat{N}_{\lambda}(dx_1) \cdots \hat{N}_{\lambda}(dx_q)
			\nonumber \\
			= &  \frac{1}{q!} \int_{A^q} \E[e^{\beta N_{\lambda}(h^{\delta}_z)}] (e^{\beta h^{\delta}_z(\cdot)}-1)^{\otimes q} (x_1, \ldots, x_q) \ind\{x_k \neq x_l \forall k\neq l\} \hat{N}_{\lambda}(dx_1) \cdots \hat{N}_{\lambda}(dx_q)
			\nonumber \\
			= & e^{10\Delta(\beta) \alpha_{\delta, D}(z) } \frac{1}{q!} I^{N_{\lambda}}_q [(e^{\beta h^{\delta}_z(\cdot)}-1)^{\otimes q}]
		\end{align} 
		where we used \eqref{eq:diff_op_exp} for the second equality. Now \eqref{eq:chaos_exp_PRM_gen} gives the required formula \eqref{eq:chaos_decomp_layering_uv}.
    \end{proof}
    
    The first main result of this subsection is the following. Suppose the random measure $V_{\lambda, \beta, D} = V_{\beta}$ is the limiting layering field of $\tilde{V}^{\delta}_{\beta}$ as $\delta \downarrow 0$, as it appears in Theorem \ref{thm:main_1}. For $\varphi \in C_b(D)$, let us use the notation 
	\begin{align}\label{eq:chaos_decomp_layering_1.5}
			V_{\lambda, \beta} (\varphi) : = \int_D \varphi(z) V_{\lambda, \beta} (z) \, dz
	\end{align}
	to denote the integral of $\varphi$ with respect to $V_{\lambda, \beta}$. Then, we clearly have 
    $\langle V_{\lambda, \beta} (\varphi) \rangle = \int_D \varphi(z) \langle V_{\lambda, \beta} (z) \rangle \, dz$ where $\E [V_{\lambda, \beta} (z)]  = \langle V_{\lambda, \beta}(z) \rangle$ was defined in \eqref{eq:conv_1_pt_0}. Also, recall from \eqref{eq:locally_exploding_kernel} that $h_z (x)  = h_z(\epsilon, \gamma) = \epsilon \ind_{A_{0, D}(z)}(\gamma)$ when $x = (\epsilon, \gamma) \in S_{\pm}$.
	
	\begin{prop}[Chaos decomposition of the Poisson layering field]\label{prop:chaos_decomp_layering}
		For each $\varphi \in C_b(D)$,  $V_{\lambda, \beta}(\varphi)$ admits a chaos decomposition 
		\begin{align}\label{eq:chaos_decomp_layering_1}
			V_{\lambda, \beta} (\varphi) = \E [V_{\lambda, \beta} (\varphi)] + \sum_{q=1}^{\infty} I^{N_{\lambda}}_q (f^{\varphi}_{q, \lambda, \beta}), \text{ in } L^2(\pr),
		\end{align} 
		where, for $x_1, \ldots, x_q \in S_{\pm}$,
		\begin{align}\label{eq:chaos_decomp_layering_2}
			f^{\varphi}_{q, \lambda, \beta}(x_1, \ldots, x_q) = \frac{1}{q!}\int_D \varphi(z) \E [V_{\lambda, \beta} (z)] \prod_{i=1}^q \left(e^{\beta h_z(x_i)} -1 \right) \, dz.
		\end{align} 		
	\end{prop}	
    The proof of the above result is in Appendix \ref{app:proof_chaos_decomp_layering}. Similarly, in the next result we state the Wiener-It\^{o} chaos expansion of the limiting GLF, whose existence is guaranteed by Theorem \ref{thm:main_2}. For each $q\ge 1$, let us define the multiple Gaussian integral as
	\begin{align}\label{eq:mult_Gaussian_integrals}
		I^G_q (h^{\otimes q}) = \int_{{S_{\pm}}^q} \prod_{i=1}^q h(x_i) \, G(dx_1) \cdots G(dx_q).
	\end{align} 
    where $h$ is a function on $S_{\pm}$. The proof of the following result is in Appendix \ref{app:proof_prop_GLF_chaos_exp}.
    
	\begin{prop}[Chaos expansions of the Gaussian layering field]\label{prop:GLF_chaos_exp}
	    For all $\varphi \in C_b(D)$, 
			\begin{align}\label{eq:W_conv_1}
				W_{\xi}(\varphi) = \E [W_{\xi}(\varphi)] + \sum_{q=1}^{\infty} I^G_q(w^{\varphi}_{q, \xi}) \text{ in } L^2(\pr),
			\end{align}
			where $\E[W_{\xi}(\varphi)] : = \int_D \varphi(z) \E[W_{\xi}(z)] \, dz$ and
			\begin{align}\label{eq:W_conv_2}
				w^{\varphi}_{q, \xi} (x_1, \ldots, x_q) = & \frac{\xi^q}{q!} \int_D \varphi(z)\E[W_{\xi}(z)]  (h_z)^{\otimes q} (x_1, \ldots, x_q) \, dz
				\nonumber \\
				= & \frac{\xi^q}{q!} \int_D \varphi(z)d_z^{2\Delta_{\xi} } e^{\frac{\xi^2}{2} \alpha_{d_z, D}(z) }   (h_z)^{\otimes q} (x_1, \ldots, x_q) \, dz
			\end{align}		
	\end{prop}
    
	\subsection{Proof of Theorem \ref{thm:main_3}}	
	
	The following proposition collects some preliminary observations. Its proof is contained in Appendix \ref{app:proof_prop:prelim_main_thm}.
	\begin{prop}\label{prop:prelim_main_thm}
		Let the assumptions of Theorem \ref{thm:main_3} hold and suppose $\varphi \in C_b (D)$. Recall the notations $f^{\varphi}_{q, \lambda, \beta}$ and $w^{\varphi}_{q, \xi}$ from the chaos expansions \eqref{eq:chaos_decomp_layering_1} and \eqref{eq:W_conv_1} of $V_{\lambda, \beta}(\varphi)$ and $W_{\xi}(\varphi)$ respectively. Then, 
		\begin{itemize}
			\item[(a)] As $\lambda \to \infty$ and $\beta \to 0$ such that $\lambda \beta^2 \to \xi^2$, we have the convergence
			\begin{align}
				\langle V_{\lambda, \beta}(\varphi)\rangle \to \langle W_{\xi}(\varphi)\rangle,
			\end{align} 
			where the quantities above have been defined in \eqref{eq:chaos_decomp_layering_1.5} and in Proposition \ref{prop:GLF_chaos_exp}.
			
			\item[(b)] $\sum_{q=1}^{\infty} q! \lVert w^{\varphi}_{q, \xi} \rVert^2_{L^2(\nu^q)} <\infty$,
			
			\item[(c)] As $\lambda \to \infty$, $\lambda \beta^2 \to \xi^2$ we have for all $q \ge 1$,
			\begin{align}\label{eq:f_w_conv_0}
				\lVert \lambda^{\frac{q}{2}} f^{\varphi}_{q, \lambda, \beta} - w^{\varphi}_{q, \xi}\rVert_{L^2(\nu^q)} \to 0.
			\end{align}		
			
			\item[(d)] As $N\to \infty$, we have,
			\begin{align}
				\limsup_{\lambda\uparrow \infty, \lambda\beta^2 \to \xi ^2} \sum_{q= N+1}^{\infty} q! \lVert \lambda^{q/2} f^{\varphi}_{q,\lambda, \beta}\rVert_{L^2(\nu^q)} \to 0.
			\end{align}
		\end{itemize}		
	\end{prop} 
	
	\begin{proof}[Proof of Theorem \ref{thm:main_3}]
		Note that the statements in the proposition above are exactly analogous to those of \cite[Eqs. (6.5) - (6.8)]{cgpr21}. Observe also that these were sufficient to establish \cite[Theorem 6.2]{cgpr21} and the calculations appearing in \cite[p. 920-921]{cgpr21} do not depend on the specific nature of the field (i.e. whether they are imaginary of real valued). Therefore the convergence of $V_{\lambda, \beta}(\varphi)$ to $W_{\xi}(\varphi)$ can be proved with the help of the same arguments without any change. Then the convergence of finite dimensional distributions as required by Theorem \ref{thm:main_3} follows by an application of the Cramer-Wold device.
	\end{proof}
	
    \appendix

    \section{Proof of Proposition \ref{prop:alpha_near_bdry}}\label{app:proof_prop:alpha_near_bdry}

    %For $z \in D$ and $\delta>0$ let us recall from \eqref{eq:loop_sets_disk} that $\bar{A}_{\delta, D}(z)$ is the collection of simple        loops in $D$ covering $z$ that are not contained in the disk $B(z, \delta)$ and $\bar{\alpha}_{\delta, D}(z)$ its $\mu_D$-mass. 
              
        (i) Let $r>0$ be fixed. Because of invariance of $\mu_{\bH}$ under the horizontal translation maps $z \mapsto z+ a$ (for $a \in \R$), it is enough to consider the case of an arbitrary point $i\cdot y$ ($y >0$) on the imaginary axis and show that $\alpha_{ry, \bH}(i y) = \alpha_{r, \bH}(i)= : C_{\bH}(r)$. But this follows from the fact that under the conformal map $f : \bH \to \bH, f(w) = y^{-1}\cdot w$ we have,
        \begin{align}
            \mu_{\bH}(A_{r, \bH} (i)) = & f\circ \mu_{\bH} (A_{r, \bH} (i)) = \mu_{\bH} (A_{ry, \bH} (i\cdot y)).    
        \end{align}
        Here the left equality is just a consequence of the conformal invariance \eqref{eq:conf_inv} of $\mu_{\bH}$ and the right equality can be obtained from the definition of $f \circ \mu_{\bH}$. Since all the loops in $A_{ry, \bH} (i\cdot y)$ have their diameters bounded below and necessarily intersect the segment $(0, iy)$ of the imaginary axis, thinness implies that the above quantity is finite. 

        By the same argument as above, there is a finite constants $\bar{C}_{\bH}(r)>0$ such that, 
        \begin{align}\label{eq:alpha_H_near_bdry}
           \bar{C}_{\bH}(r) =\bar{\alpha}_{r, \bH}(i) =  \bar{\alpha}_{ry, \bH}(i\cdot y)
        \end{align}
        and the above is finite by thinness. These prove \eqref{eq:alpha_H_near_bdry_0}. The asymptotics of the constants follow from \eqref{eq:mu_size_uv_ir} and \cite[Corollary 3]{nw11}.

        (ii) Let us prove the claim for the unit disk $\D$. It is easy to see that the constant $C_{\D}$, if it exists, has to be independent of the limit point on $\partial \D$ due to rotational symmetry of the disk. Also, because of this reason it is enough to prove the claim for the specific point $z_0=-1 \in \partial \D$ and assume that the sequence $\{z_n\}_n$ lies entirely on the real line. (If a point $z_n$ is not on the real line, by a rotation we map it to a point $z_n' \in \R$ such that $d_{z_n} = d_{z'_n}$. The rotational invariance of $\mu_{\D}$ then gives $\alpha_{d_{z_n}, \D}(z_n) = \alpha_{d_{z'_n}, \D}(z'_n)$.) For the sake of concreteness we let the sequence be $z_n = -1+1/n$ and the general case will follow similarly. We are going to show that, 
        %for a sequence $\{z_n\}_n$ in $\D$ with $z_n \to -1$ as $n\to \infty$, we have
         \begin{align}\label{eq:alpha_D_near_bdry_0}
            \limsup_{n\to \infty }\bar{\alpha}_{d_{z_n}/2, \D} (z_n) = \limsup_{n\to \infty } \mu_{\D} (\bar{A}_{d_{z_n}/2, \D} (z_n)) = \bar{C}_{\D}< \infty.
        \end{align}
        By the relation \eqref{eq:diam_disk_rel_2}, the above will imply $\limsup_{n \to \infty} \alpha_{d_{z_n}, \D} (z_n) \le \bar{C}_{\D} < \infty$, which is the required result. 
         
        Let $\varphi: \D \to \bH, w \mapsto i\frac{1+w}{1-w}$ be the conformal equivalence between the unit disk and the upper-half plane. We note that $\varphi$ can be extended locally to $\partial \D$ in a neighborhood of $-1$ and that, $\varphi(z_n) = \varphi(-1+1/n) = \frac{i}{2n-1}$ and $\varphi'(z_n) =\varphi'(-1+1/n) = \frac{2i}{(2-1/n)^2}$ for all $n \ge 1$. We have        \begin{align}\label{eq:alpha_D_near_bdry}
            & \bar{\alpha}_{1/2n, \D} (z_n) =   \mu_{\D} (\bar{A}_{1/2n, \D} (-1+1/n) )
            = \varphi^{-1}\circ \mu_{\bH} (\bar{A}_{1/2n, \D} (-1+1/n) ) 
            \nonumber \\
            = &  \mu_{\bH} (\gamma \in S_{\bH} \mid \varphi^{-1}(\gamma) \in \bar{A}_{1/2n, \D} (-1+1/n)  )
            \nonumber \\
            = & \mu_{\bH} (\gamma \in S_{\bH} \mid \varphi(-1+1/n) \in \bar{\gamma}, \gamma \nsubseteq \varphi (B(-1+1/n ; 1/2n) ) )
            \nonumber \\
            = & \mu_{\bH} \left(\gamma \in S_{\bH} \mid \frac{i}{2n-1} \in \bar{\gamma}, \gamma \nsubseteq B(\frac{i}{2n-1}  ; |\varphi'(-1+1/n)| \frac{1}{2n}) \right) + o(1)
            \nonumber \\
            = & \bar{\alpha}_{\frac{1/n}{(2-1/n)^2}, \bH} \left( z_n \right) + o(1),
        \end{align}
        as $n \to \infty$. The second equality comes from the conformal invariance of $\mu_{\bH}$, the fifth equality comes from \cite[Lemma 4.2]{cgk16} and the rest are obtained simply by applying the relevant definitions. 
        %Let $\psi_n: \bH \to \bH$ be a conformal map that takes the disk $B(\frac{i}{2n-1}; \frac{1}{2(2n-1)})$ to $B(\frac{i}{2n-1}; \frac{1}{(2-1/n)(2n-1)})$ and fixes $\frac{i}{2n-1}$. Then by a calculation as in the above and by again using \cite[Lemma 4.2]{cgk16}, we have,
        Now we observe that, 
        \begin{align}
            & \left|\bar{\alpha}_{\frac{1/n}{(2-1/n)^2}, \bH} \left( z_n \right) -  \bar{\alpha}_{\frac{1}{2(2n-1)}, \bH} \left( z_n \right) \right| 
            = & \bar{\alpha}_{\frac{1}{2(2n-1)}, \frac{1/n}{(2-1/n)^2}, \bH} \left( \frac{i}{2n-1} \right)  = \frac{1}{5} \log \frac{2}{2-1/n}.
        \end{align}
        The r.h.s. follows from \eqref{eq:mu_size_uv_ir} and clearly vanishes as $n \to \infty$.
        %above vanishes in the limit because the loops contained in the above set are those that stay inside the annulus between the circles $\partial B(\frac{i}{2n-1}; \frac{1/n}{(2-1/n)^2})$ and $\partial B(\frac{i}{2n-1}; \frac{1}{2(2n-1)} )$, whose Lebesgue measure goes to $0$ in the limit and the measure $\mu_{\bH}$ is non-atomic.
        By \eqref{eq:alpha_H_near_bdry}, $ \bar{\alpha}_{\frac{1}{2(2n-1)}, \bH} \left( \frac{i}{2n-1} \right) = \bar{C}_{\bH} (1/2)< \infty$ and thus from \eqref{eq:alpha_D_near_bdry} we see that $\limsup_{n \to \infty} \bar{\alpha}_{1/2n, \D}(-1+1/n) \le \bar{C}_{\bH}(1/2)$. Thus \eqref{eq:alpha_D_near_bdry_0} holds.
                  
       (iii) Now we consider the case of the general bounded domain $D$ with a fixed point $z_0 \in \partial D$. Let $\varphi: \D \to D$ be a conformal equivalence. Since $D$ is simply connected and has a smooth boundary by our hypothesis, we know that both $\varphi$ and $\varphi'$ can be continuously extended to the whole $\bar{\D}$ (by \cite[Theorem 3.5]{pommerenke92} and the fact that $\varphi$ has a primitive on $\D$).  Without loss of generality we can assume that $\varphi(-1) = z_0$. 
       
       Let $\{z_n\}_n$ be a sequence in $D$ with $\lim_{n \to \infty}z_n = z_0$. We are going to show that, $\bar{C}_{D}(z_0) = \limsup_{n \to \infty} \bar{\alpha}_{d_{z_n}/2, D} (z_n)$ is finite. Since $\alpha_{d_{z_n}, D}(z_n) \le \bar{\alpha}_{d_{z_n}/2, D} (z_n)$ by \eqref{eq:diam_disk_rel_2}, this will prove the required claim. We will transport the points $z_n$'s via $\varphi^{-1}$ to $\D$ and then use the result proved in the previous part. Since $\mu_D = \varphi \circ \mu_{\D}$ we have,
      \begin{align}\label{eq:alpha_D_gen_near_bdry_1}
          \bar{\alpha}_{d_{z_n}/2, D} (z_n)  &  = \varphi\circ \mu_{\D} (\bar{A}_{d_{z_n}/2, D}(z_n))
          %\nonumber \\
          =  \mu_{\D} (\gamma \in S_{\D} \mid \varphi(\gamma) \in \bar{A}_{d_{z_n}/2, D}(z_n) )
          %\nonumber \\
          %= & \mu_{\D} (\gamma \in S_{\D} \mid  z_n \in \overline{\varphi(\gamma)}, \varphi(\gamma) \nsubseteq B(z_n; d_{z_n}/2))
          \nonumber \\
          = & \mu_{\D} (\gamma \in S_{\D} \mid  \varphi^{-1}(z_n) \in \bar{\gamma}, \gamma \nsubseteq \varphi^{-1}(B(z_n; d_{z_n}/2)))
          \nonumber \\
          = & \mu_{\D} (\gamma \in S_{\D} \mid  \varphi^{-1}(z_n) \in \bar{\gamma}, \gamma \nsubseteq B(\varphi^{-1}(z_n); |(\varphi^{-1})'(z_n)|d_{z_n}/2)) + o(1)
          \nonumber \\
          = & \bar{\alpha}_{|(\varphi^{-1})'(z_n)| d_{z_n}/2, \D} (\varphi^{-1}(z_n))+ o(1)
      \end{align}
      as $n \to \infty$, by \cite[Lemma 4.2]{cgk16}. 
      
      Since $(\varphi^{-1})'$ can be extended continuously to $\partial D$ (by \cite[Theorem 3.5]{pommerenke92}), without any loss of generality we can assume that $|(\varphi^{-1})'(z_0)|<1$. Thus, for large enough $n\ge 1$, we have $|(\varphi^{-1})'(z_n)|<1$. This gives us the equality
      \begin{align}\label{eq:alpha_D_gen_near_bdry_2}
          \bar{\alpha}_{|(\varphi^{-1})'(z_n)| d_{z_n}/2, \D} (\varphi^{-1}(z_n)) = \frac{1}{5} \log \frac{1}{|(\varphi^{-1})'(z_n)|} + \bar{\alpha}_{ d_{z_n}/2, \D} (\varphi^{-1}(z_n)).
      \end{align}
      
    Now let us use the fact that, again by \cite[Theorem 3.5]{pommerenke92}, both $\varphi'$ and $(\varphi^{-1})'$ can be continuously extended to the boundaries of their respective domains. Thus, we must have $\inf_{z \in \partial D} |(\varphi^{-1})'(z)| >0$, i.e. $|(\varphi^{-1})'(z)|$ is bounded away from $0$ on $\partial D$. By \cite[Corollary 3.19]{lawler05} we have, for every $n \ge 1$,
      \begin{align}
          \frac{d_{\varphi^{-1} (z_n)}}{4 d_{z_n}} \le |(\varphi^{-1})'(z_n)| \le \frac{4 d_{\varphi^{-1} (z_n)}}{ d_{z_n}}
      \end{align}       
      and so, 
      \begin{align}
           & |\bar{\alpha}_{ d_{z_n}/2, \D} (\varphi^{-1}(z_n)) - \bar{\alpha}_{ d_{\varphi^{-1}(z_n)}/2, \D} (\varphi^{-1}(z_n)) | 
          \le \frac{1}{5} \log \frac{d_{z_n}}{d_{\varphi^{-1}(z_n)}} \le \frac{1}{5} \log \frac{4}{|(\varphi^{-1})'(z_n)|} .
      \end{align}
      Combining this with \eqref{eq:alpha_D_gen_near_bdry_1} and \eqref{eq:alpha_D_gen_near_bdry_2} we have,
      \begin{align}
          \bar{C}_D(z_0) := \limsup_{n \to \infty} \bar{\alpha}_{d_{z_n}/2, D}(z_n)  = & \limsup_{n \to \infty} \bar{\alpha}_{d_{\varphi^{-1}(z_n)}/2, \D} (\varphi^{-1}(z_n)) + \frac{1}{5} \log \frac{1}{|(\varphi^{-1})'(z_0)|}
          \nonumber \\
          = & \bar{C}_{\D} + \frac{1}{5} \log \frac{4}{|(\varphi^{-1})'(z_0)|^2}  < \infty
      \end{align}
      where the r.h.s. follows from the previous part of this proposition. Since $|(\varphi^{-1})'(z)|, z \in \partial D$ are bounded away from $0$ and $+\infty$, the last claim follows.
    
    \section{Proofs from Section \ref{sec:PLF_GLF_exist}} \label{app:proofs_corr_fns}

    \subsection{Proof of Lemma \ref{lem:1,2point}}\label{app:proof_lem:1,2point}

		(i)  Recall the notations introduced in \eqref{eq:loop_sets} and the definition of $V^{\delta}$ from \eqref{eq:layering_number_uv}. We have,
		\begin{align}
			\E (V^{\delta}_{\beta, D}(z)) 
            %\E\left( e^{\beta N^{\delta}_D (z) }\right)
            % & \E  \left[ e^{\beta \sum_{\gamma \in \eta \cap A_{\delta, D}(z)} X_{\gamma}}  \right]			
			%=  \E \left( e^{\beta \sum_{\gamma \in \eta^+ \cap A_{\delta, D}(z)} (+1)} \right) \E \left( e^{\beta \sum_{\gamma \in \eta^- \cap A_{\delta, D}(z)} (-1)} \right) 
            = & \E \left( e^{\beta \mc{L}_{\lambda/2}^+(A_{\delta, D}(z))}\right) \E \left( e^{-\beta \mc{L}_{\lambda/2}^-(A_{\delta, D}(z))}\right) 
            %\\
            %= & e^{\frac{\lambda}{2}\alpha_{\delta, D} (z) (e^{\beta}-1) } e^{\frac{\lambda}{2}\alpha_{\delta, D} (z) (e^{-\beta}-1)}
			%= e^{\lambda  \alpha_{\delta, D} (z) (\cosh(\beta) -1)} 
            = e^{10\Delta(\beta) \alpha_{\delta, D}(z) }.
		\end{align}
		by independence of $\mc{L}_{\lambda/2}^+$ and $\mc{L}_{\lambda/2}^-$ and because $\mc{L}_{\lambda/2}^{\pm}(A_{\delta, D}(z))$ are Poisson random variables with parameter $\frac{\lambda}{2}\alpha_{\delta, D} (z)$. 	For the second relation we use the fact that, since $d(z, \partial D)>R>\delta>0$, we must have $A_{\delta, R, D}(z) = A_{\delta, R, \C}(z)$. From this and \eqref{eq:mu_size_uv_ir} we can calculate the mean of the layering field as above to obtain $\E \left( V^{\delta, R} _{\beta, D}(z)\right) = \exp [\lambda  \alpha_{\delta, R, D} (z) (\cosh(\beta) -1)] = \left(\frac{R}{\delta}\right)^{2\Delta(\beta)}$.

        (ii) Since $\delta < d_z = d(z, \partial D)$ implies $\alpha_{\delta, D}(z) = \frac{1}{5} \log\frac{d_z}{\delta} + \alpha_{d_z, D}(z)$, for such $\delta>0$,
		%\begin{align}
		%	\alpha_{\delta, D}(z) = \alpha_{\delta, d_z, D} (z) + \alpha_{d_z, D}(z) =  \alpha_{\delta, d_z, \C} (z) + \alpha_{d_z, D}(z) = \frac{1}{5} \log\frac{d_z}{\delta} + \alpha_{d_z, D}(z).
		%\end{align}
		\begin{align}\label{eq:conv_1_pt}
			\delta^{2\Delta(\beta)} \E (V^{\delta}_{\beta, D}(z)) = & \delta^{2\Delta(\beta)} e^{10\Delta(\beta) \alpha_{\delta, D}(z)} 
			%\nonumber \\
			%= & \delta^{2\Delta(\beta)} e^{10\Delta(\beta) \alpha_{\delta, d_z, D}(z)} e^{10\Delta(\beta)\alpha_{d_z, D}(z)}
			%\nonumber \\
			%= & \delta^{2\Delta(\beta)} e^{2\Delta(\beta) \log\frac{d_z}{\delta}} e^{10\Delta(\beta) \alpha_{d_z, D}(z)}
			%\nonumber \\
			=  d_z^{2\Delta(\beta)} e^{10\Delta(\beta) \alpha_{d_z, D}(z)} 
		\end{align}
        and the limit follows.				
		
		(iii) From the definition of $V^{\delta}_{\beta} = V^{\delta}_{\beta, D}$, we can write,
		\begin{align}\label{eq:2point_1}
			 \E \left[ V^{\delta}_{\beta}(z) V^{\delta'}_{\beta'}(w)\right] 
             %= & \E\left[ e^{\beta N^{\delta} (z) } e^{\beta' N^{\delta'} (w) } \right] 
            %\nonumber \\
            %= & \E \left[ e^{\beta \mc{L}_{\lambda/2}^+(A_{\delta}(z)) - \beta \mc{L}_{\lambda/2}^-(A_{\delta}(z))} e^{\beta' \mc{L}_{\lambda/2}^+(A_{\delta'}(w)) - \beta' \mc{L}_{\lambda/2}^-(A_{\delta'}(w)) } \right]
            %\nonumber \\
            = & \E \left[ e^{\beta \mc{L}_{\lambda/2}^+(A_{\delta}(z))  +  \beta' \mc{L}_{\lambda/2}^+(A_{\delta'}(w)) } \right] 
            \E \left[ e^{- (\beta \mc{L}_{\lambda/2}^-(A_{\delta}(z))  +  \beta' \mc{L}_{\lambda/2}^-(A_{\delta'}(w))) } \right]
		\end{align}
        since $\mc{L}_{\lambda/2}^+$ and $\mc{L}_{\lambda/2}^-$ are independent point processes. Because the sets $A_{\delta'} (w)\setminus A_{\delta}(z)$, $A_{\delta}(z)\cap A_{\delta'}(w)$ and $A_{\delta}(z)\setminus A_{\delta'} (w)$ are disjoint, we have,   
        \begin{align}\label{eq:2point_2}
			& \E \left[ e^{\pm[ \beta \mc{L}_{\lambda/2}^{\pm}(A_{\delta}(z))  +  \beta' \mc{L}_{\lambda/2}^{\pm}(A_{\delta'}(w))] } \right] 
            \nonumber \\
            %= & \E \left[e^{  \pm\beta  \mc{L}_{\lambda/2}^{\pm} (A_{\delta}(z)\setminus A_{\delta'} (w) )} e^{ \pm (\beta+\beta')  \mc{L}_{\lambda/2}^{\pm} (A_{\delta}(z)\cap A_{\delta'} (w))} e^ {\pm\beta'  \mc{L}_{\lambda/2}^{\pm}(A_{\delta'} (w)\setminus A_{\delta}(z))} \right]
			%\nonumber \\
			= & \E \left[e^{ \pm\beta  \mc{L}_{\lambda/2}^{\pm}( A_{\delta}(z)\setminus A_{\delta'} (w) )}\right] \E \left[e^{ \pm(\beta+\beta')  \mc{L}_{\lambda/2}^{\pm}(A_{\delta}(z)\cap A_{\delta'} (w))}\right] \E \left[e^ {\pm\beta'  \mc{L}_{\lambda/2}^{\pm}(A_{\delta'} (w)\setminus A_{\delta}(z))} \right]
		\end{align}		
		Clearly $A_{\delta}(z)\cap A_{\delta'}(w) = A_{|z-w|}(z,w) = : A(z,w)$ since $\delta, \delta'<|z-w|$. Thus the middle term is,
        \begin{align}\label{eq:2point_2.1}
            \E \left[e^{ \pm (\beta+\beta')  \mc{L}_{\lambda/2}^{\pm}(A_{\delta}(z)\cap A_{\delta'} (w))}\right]  = e^{\frac{\lambda}{2}\alpha(z,w) \left(e^{\pm (\beta+\beta')}-1\right)}.
        \end{align}
        %as $\mc{L}_{\lambda/2}^{\pm} (A(z,w))$ are $\Poi(\frac{\lambda}{2}\alpha(z,w))$ random variables. 		
        To simplify calculations we assume without any loss of generality that $\delta' < \delta$. This gives us, 
		\begin{align}
			A_{\delta}(z)\setminus A_{\delta'} (w) 
            %= & \{ \gamma  \mid z \in \bar{\gamma}, d(\gamma)\ge \delta \text{ and } (w \notin \bar{\gamma} \text{ or } d(\gamma) < \delta')\}	
			%\nonumber \\
			= & \{ \gamma  \mid z \in \bar{\gamma}, d(\gamma)\ge \delta \text{ and } w \notin \bar{\gamma}\}	
	       	=A_{\delta}(z|w), 
		\end{align}
		and therefore,
		\begin{align}\label{eq:2point_3}
			\E \left[e^{ \pm\beta \mc{L}_{\lambda/2}^{\pm}( A_{\delta}(z)\setminus A_{\delta'} (w) )}\right] = e^{\frac{\lambda}{2} \alpha_{\delta}(z|w)(e^{\pm \beta} -1)}.
		\end{align}
		Lastly, since $\delta < |z-w|$ by assumption, $A_{\delta', \delta} (w) \subset A_{\delta'}(w|z)$. This implies
		\begin{align}
			A_{\delta'} (w)\setminus A_{\delta}(z) 
            %= & \{ \gamma  \mid w \in \bar{\gamma}, d(\gamma)\ge \delta' \text{ and } (z \notin \bar{\gamma} \text{ or } d(\gamma) < \delta)\}				
			= & \{ \gamma  \mid w \in \bar{\gamma}, d(\gamma)\ge \delta' \text{ and } z \notin \bar{\gamma}\}\cup \{ \gamma \mid w \in \bar{\gamma}, \delta>d(\gamma)\ge \delta' \}
			\nonumber \\
			= & A_{\delta'}(w|z) \cup A_{\delta', \delta}(w) = A_{\delta'}(w|z).
		\end{align}
		And so,
		\begin{align}\label{eq:2point_4}
			\E \left[e^ {\pm\beta' \mc{L}_{\lambda/2}^{\pm}(A_{\delta'} (w)\setminus A_{\delta}(z))} \right] = e^{\frac{\lambda}{2} \alpha_{\delta'}(w|z)(e^{\pm \beta'} -1)}.
		\end{align}		
		Combining the expressions \eqref{eq:2point_2}, \eqref{eq:2point_2.1}, \eqref{eq:2point_3} and \eqref{eq:2point_4} in \eqref{eq:2point_1} we get
		%\begin{align}
			% \E \left[ V^{\delta}_{\beta}(z) V^{\delta'}_{\beta'}(w)\right]
			% =  e^{\lambda \alpha_{\delta}(z|w)(\cosh(\beta) -1)} e^{ \lambda\alpha(z,w) \left(\cosh(\beta+\beta')-1\right)}  e^{\lambda \alpha_{\delta'}(w|z)(\cosh(\beta') -1)},
		%\end{align}
		the required expression.

        (iv) Let $0<\delta < d_{z,w}$ and $0<\delta' <d_{w,z}$. Since $\alpha_{\delta} (z|w) = \alpha_{\delta, d_{z,w}}(z) + \alpha_{d_{z,w}}(z|w)$ and $
        \alpha_{\delta'} (w|z) = \alpha_{\delta', d_{w,z}}(w) + \alpha_{d_{w,z}}(w|z)$, using \eqref{eq:2point_0} we have,
		\begin{align}
			  \E \left( V^{\delta}_{\beta}(z) V^{\delta'}_{\beta}(w)\right) 
			%\\
			%= & \exp\left[ \lambda \left\{ \alpha_{|z-w|, D}(z,w) \left(\cosh(2\beta)-1\right) +  \alpha_{\delta, D}(z|w)(\cosh(\beta) -1) + \alpha_{\delta', D}(w|z)(\cosh(\beta) -1) \right\} \right]
			%\\
			= & e^{ \lambda \alpha(z,w) \left(\cosh(2\beta)-1\right)}   e^{ \lambda\alpha_{d_{z,w}}(z|w)(\cosh(\beta) -1)}
            e^{\lambda\alpha_{d_{w,z}}(w|z)(\cosh(\beta) -1) }
			\nonumber 
            \\
			& \hspace{1cm} e^{\lambda  \alpha_{\delta,d_{z,w}}(z)(\cosh(\beta) -1)} e^{\lambda \alpha_{\delta' ,d_{w,z}}(w)(\cosh(\beta) -1)} .	
		\end{align}
		%Using the formulae $\alpha_{\delta,d_{z,w}, D}(z) = \alpha_{\delta, d_{z,w}, \C} = \frac{1}{5}\log \left(\frac{d_{z,w}}{\delta}\right)$, $\alpha_{\delta',d_{w,z}, D}(z) = \frac{1}{5}\log \left(\frac{d_{w,z}}{\delta'}\right)$
        % we can write the second term as
		%\begin{align}
		%	& \exp\left[ \lambda \left\{ \alpha_{\delta,d_{z,w}, D}(z)(\cosh(\beta) -1) + \alpha_{\delta' ,d_{w,z}, D}(w)(\cosh(\beta) -1) \right\}\right] \\
		%= & \exp\left[ \lambda \left\{ \frac{1}{5} (\cosh(\beta) -1) \log \left(\frac{d_{z,w}}{\delta}\right)+ \frac{1}{5}(\cosh(\beta) -1) \log \left(\frac{d_{w,z}}{\delta'}\right)\right\}\right]
		%\\
		%= & \left(\frac{d_{z,w}d_{w,z}}{\delta \delta'}\right)^{2 \Delta (\beta)}
		%\end{align}
		By \eqref{eq:mu_size_uv_ir}, 
        \begin{align}
            e^{\lambda  \alpha_{\delta,d_{z,w}}(z)(\cosh(\beta) -1)} e^{\lambda \alpha_{\delta' ,d_{w,z}}(w)(\cosh(\beta) -1)} = \left(\frac{d_{z,w}d_{w,z}}{\delta \delta'}\right)^{2 \Delta (\beta)} 
        \end{align} and thus, as $\delta, \delta' \to 0$, 
		\begin{align}
			 \frac{\E( V^{\delta}_{\beta, D}(z) V^{\delta'}_{\beta, D}(w))}{(\delta \delta') ^{ -2\Delta(\beta)}} 
			\to  e^{10 \Delta(2\beta)  \alpha_{|z-w|}(z,w)}  
            e^{ 10 \Delta(\beta)\alpha_{d_{z,w}}(z|w)} 
            e^{10 \Delta(\beta) \alpha_{d_{w,z}}(w|z) } (d_{z,w}d_{w,z})^{2 \Delta(\beta)}.
		\end{align}			
		%using the notation $\Delta (\beta) = \frac{\lambda}{10}(\cosh(\beta)-1)$.
        As all the terms appearing above are finite, the limit is finite.
    \subsection{Proof of Lemma \ref{lem:L^1_conv}}\label{app:proof_lem:L^1_conv}

    Let us fix $z, w \in D$ and recall that $\alpha(z,w) = \alpha_{|z-w|}(z,w)$. We observe the following simple inequalities, 
		\begin{align}\label{eq:2point_bounds}
			\alpha_{d_{z,w}}(z|w) \le \alpha_{d_{z,w}}(z),
			\alpha_{d_{w,z}}(w|z) \le \alpha_{d_{w,z}}(w)  \text{ and }
            \alpha(z,w) \le 
            %& \alpha_{d_{z,w}}(z) \wedge \alpha_{d_{w,z}}(w) \le 
            \frac{1}{2}(\alpha_{d_{z,w}}(z) + \alpha_{d_{w,z}}(w))
		\end{align}
		using the fact that $d_{z,w}, d_{w,z} \le |z-w|$. 
        %From the above lemma, we observe that whenever $\delta, \delta' < d_{z,w}$ we have,
		%\begin{align}\label{eq:L^1_1}
		%	& \E ( \tilde{V}^{\delta}_{\beta, D}(z) \tilde{V}^{\delta'}_{\beta, D}(w))
		%	\nonumber \\
		%	= & (d_{z,w}d_{w,z})^{2 \Delta(\beta)} \exp\{ 10 \Delta(2\beta) \alpha_{|z-w|, D}(z,w) + 10 \Delta(\beta) [\alpha_{d_{z,w}, D}(z|w) + \alpha_{d_{w,z}, D} (w|z)]\}.
		%\end{align}			 
		From the proof of Lemma \ref{lem:1,2point}(iv) we know that $\E ( \tilde{V}^{\delta}_{\beta}(z) \tilde{V}^{\delta'}_{\beta}(w))$ equals the r.h.s. of \eqref{eq:conv_2_pt_0} whenever $0< \delta, \delta' < d_{z,w} \wedge d_{w,z}$. Therefore for such $\delta, \delta'$, by applying the inequalities noted above, we have 
		\begin{align}\label{eq:L^1_2}
			& \E ( \tilde{V}^{\delta}_{\beta}(z) \tilde{V}^{\delta'}_{\beta}(w))
			%\nonumber\\
			%\le & (d_{z,w}d_{w,z})^{2 \Delta(\beta)} e^{5 \Delta (2\beta)(\alpha_{d_{z,w}, D}(z) + \alpha_{d_{w,z}, D}(w))} e^{10 \Delta(\beta) (\alpha_{d_{z,w}, D}(z) + \alpha_{d_{w,z}, D}(w))}
			%\nonumber \\
			\le (d_{z,w}d_{w,z})^{2 \Delta(\beta)} \exp[ (5\Delta(2\beta) + 10\Delta(\beta))(\alpha_{d_{z,w}}(z) + \alpha_{d_{w,w}}(w))] 
			%\nonumber \\
			%\le & d(D)^{4 \Delta (\beta)} \exp\{ 15\Delta(2\beta) (\alpha_{d_{z,w}, D}(z) + \alpha_{d_{w,z}, D}(w))\}  
		\end{align}
		%where for the second inequality we used the facts that $\Delta (\beta) \ge 0$ for all $\beta \in \R$ and that $D$ is bounded. %In the final line we use the relation $\cosh(2\beta) \ge \cosh(\beta)$ and thus $\Delta (2\beta) \ge \Delta (\beta)$ for all $\beta \in \R$.		
		Since we already know the point-wise convergence from Lemma \ref{lem:1,2point}(iv), by the dominated convergence theorem, it is therefore enough to show that the r.h.s. of \eqref{eq:L^1_2} is integrable. Using the notation $\Lambda = \{ (z,w) \in D\x D \mid d_{z,w}\wedge d_{w,z}<1\}$ let us split the integral as,
		\begin{align}\label{eq:L^1_3}
			I 
            %= & \int_{D\x D} (d_{z,w}d_{w,z})^{2 \Delta(\beta)}  \exp\{ (5\Delta(2\beta) + 10\Delta(\beta)) (\alpha_{d_{z,w}, D}(z) + \alpha_{d_{w,z}, D}(w))\} \, dz \, dw
			%\nonumber \\
			= & \left(\int_{\Lambda}  + \int_{D^2 \setminus \Lambda} \right) (d_{z,w}d_{w,z})^{2 \Delta(\beta)} \exp\{ (5\Delta(2\beta) + 10\Delta(\beta)) (\alpha_{d_{z,w}}(z) + \alpha_{d_{w,z}}(w))\} \, dz \, dw 
			\nonumber \\
			= & I_1+ I_2.
		\end{align}
		
	   %We here note here that for all $z \in D$, we have the relation, 
		%\begin{align}\label{eq:L^1_4}
		%	\alpha_{1, D}(z) =   \mu_D (\gamma \mid  \gamma \subset D, z\in\bar{\gamma}, d(\gamma) \ge 1)
		%	\le  \alpha_{1, d(D), \C} (z)
		%	=  \frac{1}{5} \log d(D).		
		% \end{align}
		%using the fact that $\mu_{D}$ is the restriction of $\mu_{\C}$. Therefore,
		%\begin{align}
		%	\sup_{z \in D} \alpha_{1, D} (z) \le \frac{1}{5} \log d(D).
		%\end{align} 		
		We can compute the second integral of \eqref{eq:L^1_3} as follows. We have,
		\begin{align}\label{eq:L^1_5}
			I_2 = & \int_{\Lambda^c}  (d_{z,w}d_{w,z})^{2 \Delta(\beta)}  \exp\{ (5\Delta(2\beta) + 10\Delta(\beta)) (\alpha_{d_{z,w}, D}(z) + \alpha_{d_{w,z}, D}(w))\} \, dz \, dw 
			\nonumber \\
			\le & d(D)^{4\Delta(\beta)} \int_{\Lambda^c}  \exp\{ (5\Delta(2\beta) + 10\Delta(\beta)) [\alpha_{1, D}(z) + \alpha_{1, D} (w)]\}  \, dz \, dw 
			\nonumber \\
			%\le & \exp\left(6 \Delta(2\beta) \sup_{z \in D} \alpha_{\epsilon, D}(z) \right) | D^2\setminus \Lambda_{\epsilon}| 
			%\nonumber \\
			\le & d(D)^{4\Delta(\beta)} e^ {2(\Delta(2\beta) + 2\Delta(\beta)) \log d(D))} |D^2| 
			%\le  d(D)^{2(\Delta(2\beta) + 2\Delta(\beta)) + 4\Delta(\beta)}|D^2| 
            < \infty.
		\end{align}
		where in the second line we have used the facts that $d_{z,w}, d_{w,z} \ge 1$ in this case, $\Delta(\beta), \Delta(2\beta)\ge 0$ and $d(D) \ge d_{z, w}, d_{w, z}$. The third inequality comes from the fact that $\alpha_{1, D}(z) \le (1/5) \log d(D)$ by \eqref{eq:mu_D_size}. 
		
		For the first integral $I_1$, we have $d_{z,w} \wedge d_{w,z} < 1$. Thus we can have three situations, as follows
		\begin{align}
			\Lambda = &  \{(z,w)  \mid d_{z,w}\vee d_{w,z} < 1 \} \cup \{(z,w) \mid d_{z,w} < 1, d_{w,z} \ge 1\} \cup  \{(z,w)  \mid d_{w,z} < 1, d_{z, w} \ge 1 \} 
			\nonumber \\
			=: & \Lambda_1 \cup \Lambda_2 \cup \Lambda_3.
		\end{align} 
		We consider the cases separately. On $\Lambda_1$ we have $d_{z,w} < 1$ and $d_{w,z}< 1$ and therefore one can write, $\alpha_{d_{z,w}}(z) = \alpha_{d_{z,w}, 1} (z) +\alpha_{1} (z) \le (1/5) \log (1/d_{z,w}) + \alpha_{1} (z)$ and likewise for $\alpha_{d_{w,z}}(w)$ (by \eqref{eq:mu_D_size}). So,
		\begin{align}\label{eq:L^1_6}
			I_{11} =& \int_{\Lambda_1}  (d_{z,w}d_{w,z})^{2 \Delta(\beta)} \exp\{(5\Delta(2\beta) + 10\Delta(\beta)) (\alpha_{d_{z,w}}(z) + \alpha_{d_{w,z}}(w))\} \, dz \, dw 
			\nonumber \\
			\le &\int_{\Lambda_1} (d_{z,w}d_{w,z})^{2 \Delta(\beta)} e^{ (5\Delta(2\beta) + 10\Delta(\beta)) \left[\frac{1}{5} \log \frac{1}{d_{z,w}d_{w,z}} \right]} \, e^{(5\Delta(2\beta) + 10\Delta(\beta))[\alpha_{1}(z) + \alpha_{1}(w)]} \, dz \, dw
			%\nonumber \\
			%\le & e^{(5\Delta(2\beta) + 10\Delta(\beta)) \frac{2}{5} \log d(D)}  \int_{\Lambda_1} (d_{z,w}d_{w,z})^{2 \Delta(\beta)} (d_{z,w}d_{w,z})^{-(\Delta(2\beta) + 2\Delta(\beta))} \, dz \, dw 	
			\nonumber \\
			= & 	d(D)^{2 (\Delta(2\beta) + 2\Delta(\beta))} \int_{\Lambda_1} (d_{z,w}d_{w,z})^{-\Delta(2\beta) } \, dz \, dw  .
		\end{align}
		 For $\Lambda_2$ we have,
		\begin{align}\label{eq:L^1_7}
			I_{12} = &  \int_{\Lambda_2} (d_{z,w}d_{w,z})^{2 \Delta(\beta)}  \exp\{ (5\Delta(2\beta) + 10\Delta(\beta)) (\alpha_{d_{z,w}}(z) + \alpha_{d_{w,z}}(w))\} \, dz \, dw 
			\nonumber \\
			\le & d(D)^{2\Delta(\beta)}\int_{\Lambda_2} d_{z,w}^{2 \Delta(\beta)}  e^{(5\Delta(2\beta) + 10\Delta(\beta)) \left[ \frac{1}{5}\log \frac{1}{d_{z,w}} +\alpha_{1}(z)\right]} e^{(5\Delta(2\beta) + 10\Delta(\beta)) [ \alpha_{1}(w)]} \, dz \,  dw
			\nonumber \\
			= & d(D)^{2\Delta(\beta)}\int_{\Lambda_2} d_{z,w}^{2 \Delta(\beta)} \left( \frac{1}{d_{z,w}}\right)^{(\Delta(2\beta) + 2\Delta(\beta))} e^{(5\Delta(2\beta) + 10\Delta(\beta)) [ \alpha_{1}(z)+ \alpha_{1}(w)]}\, dz \,  dw
			\nonumber \\
			\le & d(D)^{2\Delta(\beta)}\exp\left(2(\Delta(2\beta) + \Delta(\beta)) \log d(D) \right)  \int_{\Lambda_2} d_{z,w}^{-\Delta(2\beta) } \, dz  \, dw
			\nonumber \\
			= & d(D)^{2(\Delta(2\beta) + \Delta(\beta)) + 2\Delta(\beta)}  \int_{\Lambda_2} d_{z,w}^{-\Delta(2\beta)} \, dz \,  dw.
		\end{align}
		One can similarly show that 
		\begin{align}\label{eq:L^1_8}
			I_{13}  
            %& \int_{\Lambda_3}   (d_{z,w}d_{w,z})^{2 \Delta(\beta)}  e^{ (5\Delta(2\beta) + 10\Delta(\beta)) (\alpha_{d_{z,w}, D}(z) + \alpha_{d_{w,z}, D}(w))} \, dz \, dw 
			%\nonumber \\
			\le  & d(D)^{2(\Delta(2\beta) + \Delta(\beta)) + 2\Delta(\beta)}  \int_{\Lambda_2} d_{w,z}^{-\Delta(2\beta)} \, dz \,  dw.   
		\end{align}
        We now have to show that the integrals appearing in r.h.s. of \eqref{eq:L^1_6}, \eqref{eq:L^1_7} and \eqref{eq:L^1_8} are finite. For this we need the following elementary result.
        
		\begin{lem}
		Let $f : \D \to D$ be a conformal map from the unit disk onto $D$. For $z,w \in \D$ we use the notation, $d^{\D}_{z,w} = |z-w| \wedge d(z, \partial \D) = |z-w| \wedge (1-|z|)$. Then for all $z, w \in \D$ we have,
		\begin{align}\label{eq:L^1_9}
			d^{\D}_{z, w} \le \frac{4}{|f'(z)|} d_{f(z), f(w)}.
		\end{align}       
        \end{lem}
        
        \begin{proof}
		To see this we consider two cases, (I) $|z-w| < 1-|z|$ i.e. $d^{\D}_{z,w} = |z-w|$ and (II) $|z-w| \ge 1-|z|$, i.e. $d^{\D}_{w,z} =1-|z|$. In case (I), let $B = B(z, |z-w|)$ be the disk around $z$ with radius $|z-w|$, and note that $B \subset \D$. By the Koebe 1/4 theorem we have,
		\begin{align}
			d(f(z), \partial f(B)) \ge \frac{|f'(z)|}{4} d(z, \partial B) = \frac{|f'(z)|}{4} |z-w|.
		\end{align} 
		But note that $|f(z) - f(w)| \ge d(f(z), \partial f(B))$ since $f(w) \in \partial f(B)$. We also observe that $d(f(z), D) \ge d(f(z), \partial f(B))$. Combining these we can write,
		\begin{align}
			d_{f(z), f(w)} = d(f(z), D) \wedge |f(z)-f(w)| \ge d(f(z), \partial f(B)) \ge \frac{|f'(z)|}{4} |z-w|,
		\end{align}
		which implies \eqref{eq:L^1_9}. For case (II), we let $B' = B(z, 1-|z|)$ and again by the Koebe 1/4 theorem,
		\begin{align}
			d(f(z), \partial f(B')) \ge \frac{|f'(z)|}{4}  d(z, \partial B') = \frac{|f'(z)|}{4}  (1-|z|).
		\end{align} 
		But since $|f(w)- f(z)| \ge d(f(z), \partial f(B'))$  (as $w\notin B'$, $f(w) \notin f(B')$) and $d(f(z), D) \ge d(f(z), \partial f(B'))$, we can obtain \eqref{eq:L^1_9} as in the previous case.
        \end{proof}
        
		Let $f: \D \to D$ be a conformal map as in the above lemma. As $D$ has a smooth boundary and $f$ is analytic, $f'$ admits a continuous extension from $\bar{\D}$ to $\bar{ D}$ (Ref. \cite[Theorem 3.5]{pommerenke92}). In particular, we have $\lVert f' \rVert_{\infty} := \lVert f' \rVert_{L^{\infty}(\bar{\D})} < \infty$. We also have,
		\begin{align}
			(d^{\D}_{z,w})^{-\Delta(2\beta)} \le |z-w|^{-\Delta(2\beta)} + (1-|z|)^{-\Delta(2\beta)}.
		\end{align}
		since the LHS is the minimum of the two terms of the r.h.s.. Applying these to \eqref{eq:L^1_6} we can compute by change of variables,
		\begin{align} \label{eq:L^1_10}
			I_{11} \le &  c_D \int_{D^2} (d_{z,w}d_{w,z})^{-\Delta(2\beta)} \, dz \, dw
			%\nonumber \\
			\le  c_D \int_{\D^2} (d_{f(z),f(w)}d_{f(w),f(z)})^{-\Delta(2\beta)} |f'(z)f'(w)|^2 \, dz \, dw
			\nonumber \\
			\le & c_D \int_{\D^2} \left(\frac{|f'(z)|}{4}d^{\D}_{f(z),f(w)}\frac{|f'(w)|}{4}d^{\D}_{f(w),f(z)}\right)^{-\Delta(2\beta) } |f'(z)f'(w)|^2 \, dz \, dw
			%\nonumber \\
			%\le & c_D \int_{\D^2} \left(d^{\D}_{f(z),f(w)}d^{\D}_{f(w),f(z)}\right)^{-\Delta(2\beta)} |f'(z)f'(w)|^{2-\Delta(2\beta)} \, dz \, dw
			\nonumber \\
			\le & c_D \lVert f' \rVert^{2(2-\Delta(2\beta))}_{\infty} \int_{\D^2} \left(d^{\D}_{f(z),f(w)}d^{\D}_{f(w),f(z)}\right)^{-\Delta(2\beta)}  \, dz \, dw
			\nonumber \\
			\le & c_{D,f} \int_{\D^2} \left( |z-w|^{-\Delta(2\beta)} + (1-|z|)^{-\Delta(2\beta)} \right) 
			%\nonumber \\
			 \left( |z-w|^{-\Delta(2\beta)} + (1-|w|)^{-\Delta(2\beta)} \right)  \, dz \, dw			
		\end{align}
		Note that 
		\begin{align}
			\int_{\D^2} |z-w|^{-2\Delta(2\beta) } \, dz \, dw  \le & \int_{\D} \,dz \int_{B(z;2)} |z-w|^{-2\Delta(2\beta) } \, dw
			\nonumber \\
			= &   \int_{\D} \,dz \int_0^{2\pi} \int_0^2 r^{1-2\Delta(2\beta)} \, dr \, d\theta < \infty
		\end{align}
		when $\Delta(2\beta) < 1$, which is our assumption. For similar reasons,
		\begin{align}
			& \int_{\D^2} |z-w|^{-\Delta(2\beta)} (1-|z|)^{-\Delta(2\beta)}\, dz \, dw, \quad \int_{\D^2} |z-w|^{-\Delta(2\beta) } (1-|w|)^{-\Delta(2\beta)}\, dz \, dw 
			\nonumber \\
			\text{ and } & \int_{\D^2} (1-|w|)^{-\Delta(2\beta)} (1-|z|)^{-\Delta(2\beta) }\, dz \, dw 
		\end{align}
		are also finite. By \eqref{eq:L^1_10} this implies that $I_{11} < \infty$. 
		
		For $I_{12}$ and $I_{13}$ one can proceed with similar arguments to show that they are also finite. This finishes the proof of the fact that the r.h.s. of \eqref{eq:L^1_2} is integrable.

    \subsection{Proof of Proposition \ref{prop:n_point}} \label{app:proof_prop:n_point}
		
		 (i) We shall use the notation $S_0 = \{ z_1, \ldots, z_n\}$. Let $\mc{L} = \mc{L}_{\lambda, D}$ be the realization of the BLS in $D$ and let $X, \{   X_{\gamma}\}_{\gamma \in \eta}$ be i.i.d. Ber(1/2) random variables taking values $\pm 1$. For a set $S = \{ z_{i_1}, \ldots, z_{i_k}\} \subseteq S_0$, $S \neq \emptyset$, we define $I_S = \{ i_1, \ldots, i_k\}$. Also, for such a set $S$, let 
		\begin{align}\label{eq:conv_n_pt_0.1}
			A(S, S_0) = A_D(S, S_0) = & \{ \gamma \in S(D) \mid S \subset \bar{\gamma}, \, S_0\setminus S \subset D \setminus \bar{\gamma} \bar\}
			\nonumber \\
			A_{\delta}(S, S_0) = A_{\delta, D}(S, S_0) = & \{ \gamma \in A_D(S, S_0) \mid d(\gamma) \ge \delta\}
		\end{align}
		be the collection of all loops in $D$ that cover exactly the points of $S$ and its subcollection containing only loops of diameter larger than $\delta$.
		
		Let $\mc{K}$ be the collection of all functions $\bm{k}: \{ S \subset S_0\mid S\neq \emptyset\} \to \N_0$. For each such $\bm{k}$ we shall define $\mc{E}_{\bm{k}}$ to be the event that the loops in the realization $\mc{L}$ of the BLS cover the points in $S_0$ exactly according to $\bm{k}$. This is precisely defined as follows, 
		\begin{align}
			\mc{E}_{\bm{k}} = \cap_{S \subseteq S_0, S \neq \emptyset}\{  | \mc{L} \cap A_D(S, S_0)| = \bm{k}(S) \}.
		\end{align}
		Similarly, 
		\begin{align}
			\mc{E}^{\delta}_{\bm{k}} = \cap_{S \subseteq S_0, S \neq \emptyset}\{   | \mc{L} \cap A_{\delta, D}(S, S_0)| = \bm{k}(S) \}.
		\end{align}
		Note that $\mc{E}_{\bm{k}_1} \cap \mc{E}_{\bm{k}_2} = \emptyset$ whenever $\bm{k}_1 \neq \bm{k}_2 \in \mc{K}$. 
		
		Since $z_j \in S_0$, we have by definition \eqref{eq:layering_number_uv},
		\begin{align}
			N_{\delta, D}(z_j) = \sum_{\gamma \in A_{\delta, D}(z_j)\cap \mc{L}} X_{\gamma} = \sum_{S: S \subset S_0, z_j \in S} \sum_{\gamma \in A_{\delta}(S, S_0)\cap \mc{L}} X_{\gamma}.
		\end{align}		
		Note that we can write,
		\begin{align}\label{eq:conv_n_pt_1}
			\E\left(\prod_{j=1}^n V^{\delta}_{\beta_j , D}(z_j)\right) = & \E \left( e^{\sum_{j=1}^n \beta_j N_{\delta, D}(z_j)}\right)	
			= \sum_{\bm{k} \in \mc{K}} \E\left[ \ind_{\mc{E}^{\delta}_{\bm{k}} } e^{\sum_{j=1}^n \beta_j N_{\delta, D}(z_j)} \right]		
			\nonumber \\
			= & \sum_{\bm{k} \in \mc{K}} \E\left[  e^{\sum_{j=1}^n \beta_j N_{\delta, D}(z_j)} \mid \mc{E}^{\delta}_{\bm{k}} \right]	\pr(\mc{E}^{\delta}_{\bm{k}}).		
		\end{align}			
		Now for each $\bm{k} \in \mc{K}$,
		\begin{align}
			\E\left[ e^{\sum_{j=1}^n \beta_j N_{\delta, D}(z_j)} \mid \mc{E}^{\delta}_{\bm{k}} \right]	
			%= & \E\left[e^{\sum_{j=1}^n \beta_j \sum_{S: S \subset S_0, z_j \in S} \sum_{\gamma \in A_{\delta}(S, S_0)\cap \mc{L}} X_{\gamma}} \mid \mc{E}^{\delta}_{\bm{k}} \right]			
			= & \E\left[ e^{\sum_{S\subset S_0, S\neq \emptyset} \sum_{\gamma \in A_{\delta}(S, S_0)\cap \mc{L}} \left(\sum_{j \in I_S} \beta_j\right) X_{\gamma}} \mid \mc{E}^{\delta}_{\bm{k}}\right]
			\nonumber \\
			%= & \E\left[ \prod_{S\subset S_0, S\neq \emptyset} e^{\sum_{\gamma \in A_{\delta}(S, S_0)\cap \mc{L}} \left(\sum_{j \in I_S} \beta_j\right) X_{\gamma}} \mid \mc{E}^{\delta}_{\bm{k}}\right]	
			%\\
			= & \prod_{S\subset S_0, S\neq \emptyset} \E\left[ e^{ \sum_{\gamma \in A_{\delta}(S, S_0)\cap \mc{L}} \left(\sum_{j \in I_S} \beta_j\right) X_{\gamma} } \mid \mc{E}^{\delta}_{\bm{k}} \right]
			\nonumber \\
			= &  \prod_{S\subset S_0, S\neq \emptyset} \left[\E \left(e^{ \left(\sum_{j \in I_S} \beta_j\right) X} \right) \right]^{\bm{k}(S)}
			\nonumber \\
			= & \prod_{S\subset S_0, S\neq \emptyset} \left[\cosh\left(\sum_{j \in I_S} \beta_j\right)\right]^{\bm{k}(S)}
		\end{align}
		The first equality is obtained by rearranging the summations. The second uses the fact that if $S_1\neq S_2$ are non-empty subsets of $S_0$, then no loop can appear simultaneously in both $A_{\delta}(S_1, S_0)\cap \mc{L}$ and $A_{\delta}(S_2, S_0)\cap \mc{L}$ and thus the respective sums are independent. The third follows from the definition of $\mc{E}^{\delta}_{\bm{k}}$. Also note that, with the notation $\alpha_{\delta, D}(S, S_0) = \mu_D(A_{\delta, D}(S, S_0))$, the random variables $|\mc{L} \cap A_{\delta, D}(S, S_0)|$ are distributed according to $\Poi(\lambda \alpha_{\delta, D}(S,S_0))$. Using the above mentioned independence, we thus get,
		\begin{align}
			\pr(\mc{E}^{\delta}_{\bm{k}}) =   \prod_{S\subset S_0, S\neq \emptyset} \pr( | \mc{L} \cap A_{\delta, D}(S, S_0)| = \bm{k}(S) )
	       =\prod_{S\subset S_0, S\neq \emptyset} e^{-\lambda \alpha_{\delta, D}(S, S_0)} \frac{(\lambda \alpha_{\delta, D}(S, S_0))^{\bm{k}(S)}}{\bm{k}(S)! }.
		\end{align}		
		Therefore from \eqref{eq:conv_n_pt_1} we have,
		\begin{align}
			\E\left(\prod_{j=1}^n V^{\delta}_{\beta_j , D}(z_j)\right) 	= & \sum_{\bm{k} \in \mc{K}} \E\left[  e^{\sum_{j=1}^n \beta_j N_{\delta, D}(z_j)} \mid \mc{E}^{\delta}_{\bm{k}} \right]	\pr(\mc{E}^{\delta}_{\bm{k}})
			\nonumber \\
			= & \sum_{\bm{k} \in \mc{K}} \prod_{S\subset S_0, S\neq \emptyset} \left[\cosh\left(\sum_{j \in I_S} \beta_j\right)\right]^{\bm{k}(S)} e^{-\lambda \alpha_{\delta, D}(S, S_0)} \frac{(\lambda \alpha_{\delta, D}(S, S_0))^{\bm{k}(S)}}{\bm{k}(S)! }
			%\nonumber \\
			%= & \sum_{\bm{k} \in \mc{K}} \prod_{S\subset S_0, |S|>1} \left[\cosh\left(\sum_{j \in I_S} \beta_j\right)\right]^{\bm{k}(S)} e^{-\lambda \alpha_{\delta, D}(S, S_0)} \frac{(\lambda \alpha_{\delta, D}(S, S_0))^{\bm{k}(S)}}{\bm{k}(S)! }
			\nonumber \\
			& \hspace{1cm} \prod_{j=1}^n \left[\cosh\left(\beta_j\right)\right]^{\bm{k}(z_j)} e^{-\lambda \alpha_{\delta, D}(z_j, S_0)} \frac{(\lambda \alpha_{\delta, D}(z_j, S_0))^{\bm{k}(z_j)}}{\bm{k}(z_j)! } .
		\end{align}
        Rearranging the terms in the r.h.s. we get,
        \begin{align}\label{eq:conv_n_pt_5}
			\E\left(\prod_{j=1}^n V^{\delta}_{\beta_j , D}(z_j)\right) = & \prod_{S\subset S_0, |S|>1} \left[ \sum_{\bm{k}(S)=0}^{\infty} e^{-\lambda \alpha_{\delta, D}(S, S_0)} \frac{\left(\lambda \alpha_{\delta, D}(S, S_0) \cosh\left(\sum_{j \in I_S} \beta_j\right)\right)^{\bm{k}(S)}}{\bm{k}(S)! } \right]
			%\nonumber \\
			%& \prod_{j=1}^n \left[ \sum_{\bm{k}(z_j)=0}^{\infty}   e^{-\lambda \alpha_{\delta, D}(z_j, S_0)} \frac{(\lambda \alpha_{\delta, D}(z_j, S_0) \cosh\left(\beta_j\right) )^{\bm{k}(z_j)}}{\bm{k}(z_j)! } \right]
			\nonumber \\
			= & \prod_{S\subset S_0, |S|>1}  \exp\left[ -\lambda \alpha_{\delta, D}(S, S_0) \left\{ 1- \cosh\left(\sum_{j \in I_S} \beta_j\right)\right\}\right]
			\nonumber \\
			& \hspace{1cm} \prod_{j=1}^n \exp\left[ -\lambda \alpha_{\delta, D}(z_j, S_0) \left\{ 1- \cosh\left( \beta_j\right)\right\}\right]
            \nonumber \\
            = &  \prod_{S\subset S_0, |S|>1}  e^{10 \Delta\left(\sum_{j \in I_S} \beta_j\right) \alpha_{\delta, D}(S, S_0)}
			\prod_{j=1}^n e^{10\Delta(\beta_j)\alpha_{\delta, D}(z_j, S_0)}
		\end{align}
		
		If $m = \min_{i\neq j}|z_i-z_j|\wedge \min_i d(z_i, \partial D)$ then for $\delta < m$ and each $j$ we have, $\alpha_{\delta, D}(z_j, S_0) = \frac{1}{5} \log \frac{m}{\delta} + \alpha_{m, D} (z_j, S_0)$. Applying this observation to \eqref{eq:conv_n_pt_5} we get,
		\begin{align}\label{eq:conv_n_pt_6}
			\E\left(\prod_{j=1}^n V^{\delta}_{\beta_j , D}(z_j)\right) = & \prod_{S\subset S_0, |S|>1}  e^{10 \Delta\left(\sum_{j \in I_S} \beta_j\right) \alpha_{\delta, D}(S, S_0)}
			\prod_{j=1}^n \left(\frac{m}{\delta}\right)^{2\Delta(\beta_j)} e^{10\Delta(\beta_j)\alpha_{m, D}(z_j, S_0)}.
		\end{align}
		Note that, in the above, $\alpha_{\delta , D} (S, S_0) = \alpha_{D}(S, S_0)$ whenever $|S|>1$ and $\delta <m$. Therefore,
		\begin{align}\label{eq:conv_n_pt_7}
			 & \phi_D(z_1, \ldots, z_n ; \bm{\beta}) : =  \lim_{\delta \downarrow 0} \delta^{2\sum_{j=1}^n \Delta(\beta_j) }\E\left(\prod_{j=1}^n V^{\delta}_{\beta_j , D}(z_j)\right)	
			\nonumber \\
            = & m^{2\sum_{j=1}^n \Delta(\beta_j) } \prod_{S\subset S_0, |S|>1}  e^{10 \Delta\left(\sum_{j \in I_S} \beta_j\right) \alpha_{D}(S, S_0)}         
            \prod_{j=1}^n e^{10\Delta(\beta_j)\alpha_{m, D}(z_j, S_0)}			
		\end{align}
		exists.
        
       (ii) Now we can show the conformal covariance of the object obtained above. Let $S'_0 = \{ z'_1, \ldots, z'_n\}$. Similar to the sets defined in \eqref{eq:conv_n_pt_0.1} we can define the sets $A_{D'}(S', S'_0)$, $A_{\delta, D'} (S', S'_0)$ when $S'\subseteq S'_0$ and denote by $\alpha_{D'}(S', S'_0)$, $\alpha_{\delta, D'}(S', S'_0)$ their $\mu_{D'}$-mass. Since $f$ is conformal, we know that $\mu_{D'} = f \circ \mu_D$ and it thus follows from the definition that, $\alpha_{D'}(S', S'_0) = \alpha_{D}(S, S_0)$ whenever $S\subset S_0, |S|>1$ and $S'= f(S)$. 

       Let $m' = \min_{i\neq j}|z'_i - z'_j| \wedge d(z'_j, \partial D')$. When $\delta< m \wedge m'$ we can use the expression obtained in \eqref{eq:conv_n_pt_5} and its analogous expression for $D'$ to get, 
       \begin{align} \label{eq:conv_n_pt_8}
           \frac{\E\left(\prod_{j=1}^n V^{\delta}_{\beta_j , D'}(z'_j)\right) }{\E\left(\prod_{j=1}^n V^{\delta}_{\beta_j , D}(z_j)\right) } =  & \frac{ \prod_{S'\subset S'_0, |S'|>1}  e^{10 \Delta\left(\sum_{j \in I_S} \beta_j\right) \alpha_{\delta, D'}(S', S'_0)}
			\prod_{j=1}^n e^{10\Delta(\beta_j)\alpha_{\delta, D'}(z'_j, S'_0)} }
            {\prod_{S\subset S_0, |S|>1}  e^{10 \Delta\left(\sum_{j \in I_S} \beta_j\right) \alpha_{\delta, D}(S, S_0)}
			\prod_{j=1}^n e^{10\Delta(\beta_j)\alpha_{\delta, D}(z_j, S_0)}}
            \nonumber \\
            = & \prod_{j=1}^n e^{10\Delta(\beta_j)[\alpha_{\delta, D'}(z'_j, S'_0)-\alpha_{\delta, D}(z_j, S_0)]}
       \end{align}

       We therefore need to analyze the differences $\alpha_{\delta, D}(z_j, S_0) - \alpha_{\delta, D'}(z'_j, S'_0)$ appearing in the above expression. This can be found in \cite[p. 498]{cgk16}, where it was shown that,
       \begin{align}
           \alpha_{\delta, D'}(z'_j, S'_0)-\alpha_{\delta, D}(z_j, S_0) = \frac{1}{5} \log|f'(z_j)| - o(1) \text{ as } \delta \to 0.
       \end{align}
       for every $j$. Plugging the above in \eqref{eq:conv_n_pt_8} we have,
       \begin{align}
           \frac{\phi_{D'}(z'_1, \ldots, z'_n; {\bm \beta})}{\phi_{D}(z_1, \ldots, z_n; {\bm \beta})} = \lim_{\delta \to 0} \frac{\E\left(\prod_{j=1}^n V^{\delta}_{\beta_j , D'}(z'_j)\right) }{\E\left(\prod_{j=1}^n V^{\delta}_{\beta_j , D}(z_j)\right) } =\prod_{j=1}^n |f'(z_j)|^{2\Delta(\beta_j)}, 
       \end{align}
       which proves \eqref{eq:conf_cov_n_pt}.

    \subsection{Proof of Proposition \ref{prop:GLF_facts}} \label{app:proof_prop_GLF_facts}
    (i) Since $G_0(A_{\delta, D}(z))$ is a centered Gaussian random variable with variance $\E[G_0(h^{\delta}_z)^2] = \alpha_{\delta}(z)$, clearly $\E[W^{\delta}_{\xi} (z)] = e^{\frac{\xi^2}{2}\alpha_{\delta}(z)}$. Thus for $0<\delta < d_z$, we have by \eqref{eq:mu_size_uv_ir},
		\begin{align}
			\E[\tilde{W}^{\delta}_{\xi}(z)] 
			=  \delta^{2\Delta_{\xi}} e^{\frac{\xi^2}{2} [\alpha_{\delta, d_z }(z) + \alpha_{d_z }(z)]}
			=  \delta^{2\Delta_{\xi}} e^{\frac{\xi^2}{2} \left[\frac{1}{5} \log\frac{d_z}{\delta} + \alpha_{d_z }(z)\right]}
		      =  d_z^{2\Delta_{\xi} } e^{\frac{\xi^2}{2} \alpha_{d_z}(z)},
		\end{align}
		which gives \eqref{eq:W_1_pt_1}.

        (ii) Suppose $|z-w|\wedge d(z, \partial D) \wedge d(w, \partial D) = d_{z,w}\wedge d_{w,z}>\delta, \delta'>0$. By the definition of  $G_0$ we have, $\E[G_0(A_{\delta}(z)) G_0( A_{\delta'}(w))] = \mu_D(A_{\delta}(z) \cap A_{\delta'}(w) ) = \mu_D(A_{\delta\vee \delta'}(z,w)) = \alpha_{|z-w|} (z,w)$. By using \eqref{eq:mu_size_uv_ir} again we get,
		\begin{align}\label{eq:W_conv_3}
			\E (\tilde{W}^{\delta}_{\xi}(z) \tilde{W}^{\delta'}_{\xi}(w))
			%= & (\delta\delta')^{2\Delta_{\xi}} \E \left[e^{\xi (G(h^{\delta}_z) + G(h^{\delta'}_w))}\right]
			%\nonumber \\
			= & (\delta\delta')^{2\Delta_{\xi}} e^{ \frac{\xi^2}{2} \E[(G_0(A_{\delta}(z)) + G_0(A_{\delta'}(w)))^2] }
			\nonumber \\
			%= &  (\delta\delta')^{2\Delta_{\xi}} e^{\frac{\xi^2}{2} \left( \alpha_{\delta} (z) + \alpha_{\delta'}(w) + 2 \alpha_{\delta\vee \delta'}(z,w) \right) }
			%\nonumber \\
			= & (\delta\delta')^{2\Delta_{\xi}} e^{\frac{\xi^2}{2} \left( \alpha_{\delta , d_{z,w}} (z) + \alpha_{\delta', d_{w,z}}(w)\right)} e^{\frac{\xi^2}{2} \left( \alpha_{d_{z,w}}(z) + \alpha_{d_{w,z}}(w)+ 2 \alpha_{|z-w|}(z,w) \right) }
			\nonumber \\
			= & (\delta\delta')^{2\Delta_{\xi}} e^{\frac{\xi^2}{2} \left( \frac{1}{5}\log \frac{d_{z,w}d_{w,z}}{\delta \delta'}\right)} 
                e^{ \frac{\xi^2}{2} \left( \alpha_{d_{z,w}}(z) + \alpha_{d_{w,z}}(w)+ 2 \alpha_{|z-w|}(z,w) \right)}
			\nonumber \\
			%= & (\delta\delta')^{2\Delta_{\xi}} \left(\frac{d_{z,w}d_{w,z}}{\delta \delta'}\right)^{2\Delta_{\xi}} e^{\frac{\xi^2}{2} \left( \alpha_{d_{z,w}}(z) + \alpha_{d_{w,z}}(w)+ 2 \alpha_{|z-w|}(z,w) \right) }
			%\nonumber \\
			= & (d_{z,w}d_{w,z})^{2\Delta_{\xi}} e^{\frac{\xi^2}{2}  \left( \alpha_{d_{z,w},D}(z) + \alpha_{d_{w,z}}(w)+ 2 \alpha_{|z-w|}(z,w) \right) }.
		\end{align}		
		This is the required expression from \eqref{eq:W_2_pt}.	
		
		(iii) To prove that the above convergence can be upgraded from point-wise to convergence in $L^1(D\x D)$ note that, when $d_{z,w}\wedge d_{w,z}>\delta, \delta'>0$, by \eqref{eq:W_conv_3} we have, using the fact that $a\wedge b \le \frac{1}{2}(a+b)$,
		\begin{align}\label{eq:W_conv_6}
			 \E (\tilde{W}^{\delta}_{\xi}(z) \tilde{W}^{\delta'}_{\xi}(w)) 
			=& (d_{z,w}d_{w,z})^{2\Delta_{\xi}} e^{\frac{\xi^2}{2} \left( \alpha_{d_{z,w}}(z) + \alpha_{d_{w,z}}(w)+ 2 \alpha_{|z-w|}(z,w) \right)}
			\nonumber \\
			\le & d(D)^{4\Delta_{\xi}} e^{ \frac{\xi^2}{2} \left( \alpha_{d_{z,w}}(z) + \alpha_{d_{w,z}}(w)+ 2 \alpha_{d_{z,w}}(z)\wedge  \alpha_{d_{w,z}}(w) \right)}
			\nonumber \\
			\le & d(D)^{4\Delta_{\xi}} e^{\xi^2 \left( \alpha_{d_{z,w}}(z) + \alpha_{d_{w,z}}(w)\right)}.
		\end{align}
		Calculations such as the ones appearing in \eqref{eq:L^1_3} - \eqref{eq:L^1_7} show that the r.h.s. of \eqref{eq:W_conv_6} is integrable, i.e. $\int_D \int_D e^{20\Delta_{\xi} \left( \alpha_{d_{z,w},D}(z) + \alpha_{d_{w,z}, D}(w)\right)} \, dz \, dw < \infty$
		for $\Delta_{\xi} < 1/4$. This and \eqref{eq:W_2_pt} prove \eqref{eq:L^1_conv_W_0}.

    \section{Proofs from Section \ref{sec:PLF_to_GLF}} \label{app:proofs_PLF_to_GLF}

   % \subsection{Proof of Lemma \ref{lem:chaos_decomp_layering_uv}}\label{app:proof_chaos_decomp_layering_uv}

    \subsection{Proof of Proposition \ref{prop:chaos_decomp_layering}}\label{app:proof_chaos_decomp_layering}
    We begin by observing that when $\varphi \in C_b(D)$, for each $\delta>0$, clearly $\tilde{V}^{\delta}_{\lambda, \beta}(\varphi)$ is in $L^2(\pr)$. Hence by an application of Proposition \ref{prop:wi_chaos_poisson} and Lemma \ref{lem:chaos_decomp_layering_uv} we know that $\tilde{V}^{\delta}_{\lambda,\beta}(\varphi)$ has a chaos decomposition as in \eqref{eq:chaos_exp_PRM_gen} with kernels     \begin{align}\label{eq:chaos_decomp_layering_4}
			f^{\delta, \varphi}_{q, \lambda, \beta} (x_1, \ldots, x_q) = \frac{1}{q!}  \int_D \varphi(z) \E(\tilde{V}^{\delta}_{\beta}(z)) \prod_{i=1}^{q} \left(e^{\beta h^{\delta}_z(x_i)}-1\right) \, dz
			=:  \frac{1}{q!} \int_D F^{\delta, \varphi}_{q; x_1, \ldots, x_q; \lambda, \beta} (z ) \,  dz .
	\end{align}
    where $x_1, \ldots, x_q \in S_{\pm}$. The form of the kernels follows from the r.h.s. of \eqref{eq:chaos_decomp_layering_uv}. Now, calculations similar to the ones in \eqref{eq:V_delta_cauchy_1} and \eqref{eq:V_delta_cauchy_2} can be used to show that $V_{\lambda,\beta}(\varphi)$ exists as an $L^2(\pr)$-limit of $\tilde{V}^{\delta}_{\lambda,\beta}(\varphi)$ as $\delta \to 0$. In particular, $V_{\lambda, \beta}(\varphi)$ is square-integrable. So by Proposition \ref{prop:wi_chaos_poisson}, $V_{\lambda, \beta}(\varphi)$ has a decomposition as in \eqref{eq:chaos_exp_PRM_gen} and each kernel $f^{\varphi}_{\lambda, \beta, q}$ (for $q\ge 1$) is in $L^2(\nu_{\lambda}^q)$.
    %{\clpl Also, for every $q\ge 1$, as $\delta \to 0$, \begin{align}
    %    f^{\varphi, \delta}_{\lambda, \beta, q} \to f^{\varphi}_{\lambda, \beta, q} \text{ in } L^2(\nu_{\lambda}^q).
    %\end{align}}    
    
    Since $\varphi$, $\lambda$ and $\beta$ are fixed throughout this proof, we will use the notations $f^{\delta}_{q}= f^{\delta, \varphi}_{q, \lambda, \beta}$ and $F^{\delta, \varphi}_{q; x_1, \ldots, x_q; \lambda, \beta} = F^{\delta}_{q; x_1, \ldots, x_q}$. Then by \eqref{eq:chaos_decomp_layering_uv} and Fubini's theorem we have,
		\begin{align}\label{eq:chaos_decomp_layering_3}
			\tilde{V}^{\delta}_{\beta}(\varphi) 
	        = & \delta^{2\Delta(\beta)} \int_D \varphi(z) \E(V^{\delta}_{\beta}(z)) \left( 1 + \sum_{q=1}^{\infty} \frac{1}{q!} I_q^{N_{\lambda}} \left[(e^{\beta h^{\delta}_z(\cdot)}-1)^{\otimes q}\right]\right) \, dz
			%\nonumber \\
			%= & \delta^{2\Delta(\beta)} \int_D   \varphi(z) \E(V^{\delta}_{\beta}(z)) \, dz + \sum_{q=1}^{\infty} \frac{1}{q!} \int_D \varphi(z) \delta^{2\Delta(\beta)}  \E(V^{\delta}_{\beta}(z)) I_q^{N_{\lambda}} \left[(e^{\beta h^{\delta}_z(\cdot)}-1)^{\otimes q}\right] \, dz 
			\nonumber \\
			= & \delta^{2\Delta(\beta)} \E[V^{\delta}_{\beta}(\varphi)] + \sum_{q=1}^{\infty} I_q^{N_{\lambda}} (f^{\delta}_{ q})
		\end{align}	
		As $\delta \to 0$, since $h^{\delta}_z (x) = \epsilon \ind_{A_{\delta}(z)}(\gamma) \to \epsilon \ind_{A_0(z)} (\gamma) =  h_z(x)$ for every $x =(\epsilon, \gamma) \in S_{\pm}$, we have for all $z \in D$, as $\delta \to 0$,
		\begin{align}\label{eq:chaos_decomp_layering_5}
			F^{\delta}_{q; x_1, \ldots, x_q} (z ) = & \varphi(z) \E\left(\tilde{V}^{\delta}_{\beta}(z)\right) \prod_{i=1}^{q} \left(e^{\beta h^{\delta}_z(x_i)}-1\right)
			\to & \varphi(z) \langle V_{\beta}(z) \rangle \prod_{i=1}^q\left(e^{\beta h_z(x_i)} -1\right)
		\end{align}
        applying  eq. \eqref{eq:conv_1_pt_0}. We also have, when $\delta < d_z$, 
		\begin{align}\label{eq:chaos_decomp_layering_6}
			|F^{\delta}_{q; x_1, \ldots, x_q}(z)| 
            %\le & |\varphi(z)| d_z^{2\Delta(\beta)} \exp\left(\lambda \alpha_{d_z, D} (z) (\cosh(\beta)-1)\right) e^{q|\beta|} 
			\le  \lVert \varphi \rVert_{\infty} d_z ^{2\Delta(\beta)} e^{10\Delta(\beta) \alpha_{d_z,D}(z)} e^{q|\beta|}.
		\end{align}
		This can be seen from \eqref{eq:conv_1_pt} and the fact that $|e^{\beta h^{\delta}_z(x_i)}-1| \le |e^{\beta} -1|\vee |e^{-\beta} -1| \le e^{|\beta|}$ for all $i =1, \ldots, q$. The r.h.s. of \eqref{eq:chaos_decomp_layering_6} is integrable, since \eqref{eq:mu_D_size} gives
		\begin{align}\label{eq:chaos_decomp_layering_7}
			%\int_D |F^{\delta, \varphi}_{q; x_1, \ldots, x_q} (z ) |\, dz \le & 
            \int_D d_z ^{2\Delta(\beta)} e^{10\Delta(\beta) \alpha_{d_z,D}(z)} \, dz
		      \le  \int_D d_z ^{2\Delta(\beta)} e^{2 \Delta(\beta)\log\frac{d(D)}{d_z}} \, dz
		    = d(D)^{2\Delta(\beta)} |D| <\infty.
		\end{align}		
		By the dominated convergence theorem, this shows that the point-wise convergence in \eqref{eq:chaos_decomp_layering_5} can be upgraded to $L^1(D, dz)$-convergence and therefore by \eqref{eq:chaos_decomp_layering_4}, for all $x_1, \ldots, x_q \in S_{\pm}$,
		\begin{align}\label{eq:chaos_decomp_layering_8}
			\lim_{\delta \to 0}f^{\delta}_q(x_1, \ldots, x_q) = \frac{1}{q!} \int_D \varphi(z) \langle V_{\beta}(z) \rangle \prod_{i=1}^q\left(e^{\beta h_z(x_i)} -1\right) \, dz,
		\end{align}
		which is the r.h.s. of \eqref{eq:chaos_decomp_layering_2}. An It\^{o} isometry type calculation (see \cite[Lemma 12.2]{lp18}) then implies that $(I^{N_{\lambda}}_q(f^{\delta}_q))_{\delta>0}$ is a Cauchy sequence in $L^2(\pr)$, as $\delta \to 0$. Let us briefly verify this for $q=1$. For $\delta, \delta'>0$ using the triangle inequality and Jensen's inequality we have,
        \begin{align}
             \E\left( |I^{N_{\lambda}}_1(f^{\delta}_1) - I^{N_{\lambda}}_1(f^{\delta'}_1)|^2 \right) 
            = & \int_{S_{\pm}} \nu_{\lambda}(dx) \left(f^{\delta}_1(x) - f^{\delta'}_1(x)\right)^2
            \nonumber \\
            \le & c \int_{S_{\pm}} \nu_{\lambda}(dx) \left[ \int_D \varphi(z) \E(\tilde{V}^{\delta}_{\beta}(z)) (e^{\beta h^{\delta}_z(x)} - e^{\beta h^{\delta'}_z(x)}) \right]^2 
            \nonumber  \\
            & + c \int_{S_{\pm}} \nu_{\lambda}(dx) \left[ \int_D \varphi(z) [\E(\tilde{V}^{\delta}_{\beta}(z)) - \E(\tilde{V}^{\delta}_{\beta}(z)) ] (e^{\beta h^{\delta'}_z(x)}-1)  \right]^2
            \nonumber \\
            \le & c_D \int_{S_{\pm}} \nu_{\lambda}(dx)  \int_D \varphi(z)^2 \E(\tilde{V}^{\delta}_{\beta}(z))^2 (e^{\beta h^{\delta}_z(x)} - e^{\beta h^{\delta'}_z(x)})^2 
           \nonumber  \\
            & + c_D\int_{S_{\pm}} \nu_{\lambda}(dx)  \int_D \varphi(z)^2 [\E(\tilde{V}^{\delta}_{\beta}(z)) - \E(\tilde{V}^{\delta}_{\beta}(z)) ]^2 (e^{\beta h^{\delta'}_z(x)}-1)^2. 
        \end{align}
        Using the definition of $h^{\delta}_z$ from \eqref{eq:locally_exploding_kernel} and our calculations of the one-point functions (from \eqref{eq:1_point_0} and \eqref{eq:conv_1_pt}) we see that, clearly both integrals in the above converge to $0$ as $\delta, \delta' \to 0$. Hence $(I^{N_{\lambda}}_1(f^{\delta}_1))_{\delta >0}$ is Cauchy in $L^2(\pr)$.

        Combining all of the observations made above we can see that, for all $q\ge 1$, 
		\begin{align}
			I^{N_{\lambda}}_q (f^{\delta}_q) \to I^{N_{\lambda}}_q (f_q), \text{ in } L^2(\pr).
		\end{align}
		Therefore the r.h.s. of \eqref{eq:chaos_decomp_layering_3} converges to that of \eqref{eq:chaos_decomp_layering_1}. Since from Proposition \ref{prop:layering_field_conv} we know that $\delta^{2\Delta(\beta)} V^{\delta}_{\beta}(\varphi) \to V_{\beta}(\varphi)$ in $ L^2(\Omega,\pr)$ as $\delta \to 0$, we get the equality of \eqref{eq:chaos_decomp_layering_1}.

    \subsection{Proof of Proposition \ref{prop:GLF_chaos_exp}} \label{app:proof_prop_GLF_chaos_exp}

    Before giving the proof, we state a lemma which is the Gaussian analogue to the chaos decomposition result stated in Proposition \ref{prop:wi_chaos_poisson}. % which was concerned with functionals of Poisson random measures. 
    Recall that $G$ is a Gaussian random measure defined in Section \ref{sec:prelim_glf_gmc} with intensity measure $\nu = \frac{1}{2}(\delta_{-1} + \delta_{+1})\otimes \mu_D$ on $S_{\pm}(D)$.
    \begin{lem}\label{lem:wi_chaos_gaussian}
		Let $h \in L^2(\nu)$ be such that $\nu(h) = 0$. Then, for $\xi \in \R$
		\begin{align}\label{eq:G_decomp}
			e^{\xi G(h) } = e^{\frac{\xi^2}{2} \nu(|h|^2)} \left[1 + \sum_{q=1}^{\infty} \frac{\xi^q}{q!} I^G_q[ h^{\otimes q}]\right], \text{ in } L^2(\pr)
		\end{align}
		where $I^G_q$ was defined in \eqref{eq:mult_Gaussian_integrals}. 
	\end{lem}	

    \begin{proof}
        This is a consequence of \cite[Theorem 8.2.1]{pt11}. Note that, by \cite[Eq. (8.2.11)]{pt11} we have,
        \begin{align}
            e^{\xi G(h/\lVert h \rVert_{L^2(\nu)}) - \frac{\xi^2}{2}} = 1+ \sum_{q=1}^{\infty} \frac{\xi^q}{q!} I^G_q\left( \left[ \frac{h}{\lVert h \rVert_{L^2(\nu)}} \right]^{\otimes q}\right).
        \end{align}
        Now the change of variable $\xi \to \xi / \lVert h \rVert_{L^2(\nu)}$ gives \eqref{eq:G_decomp}.
    \end{proof}

    Now we are ready to prove Proposition \ref{prop:GLF_chaos_exp} and it is similar to the one contained in the previous section. Since $W^{\delta}_{\xi}(z) = e^{\xi G(h^{\delta}_z)}$ and $\nu(|h^{\delta}_z|^2) < \infty$, by \eqref{eq:G_decomp}, we have that for all $\varphi \in C_b (D)$
		\begin{align}\label{eq:W_decomp_1}
			\tilde{W}^{\delta}_{\xi}(\varphi)  
            %& \delta^{2\Delta_{\xi}} \int_D \varphi(z) W^{\delta}_{\xi} (z) \, dz
			%\nonumber \\
			= & \delta^{2\Delta_{\xi}} \int_D \varphi(z) e^{\frac{\xi^2}{2} \nu(|h^{\delta}_z|^2)} \left(1+ \sum_{q=1}^{\infty} \frac{\xi^q}{q!} I^G_q[(h^{\delta}_z)^{\otimes q}]\right) \, dz
			\nonumber \\
			%= & \delta^{2\Delta_{\xi}} \left[\int_D \varphi(z) e^{\frac{\xi^2}{2} \alpha_{\delta, D}(z)} \, dz 
			%+ \sum_{q=1}^{\infty} \frac{\xi^q}{q!}  \int_D e^{\frac{\xi^2}{2} \alpha_{\delta, D}(z)}  \varphi(z) I^G_q[(h^{\delta}_z)^{\otimes q}] \, dz\right]
			%\nonumber \\
			= & \delta^{2\Delta_{\xi}} \left[\int_D \varphi(z) e^{\frac{\xi^2}{2} \alpha_{\delta, D}(z)} \, dz 
			+ \sum_{q=1}^{\infty} I^G_q\left(\frac{\xi^q}{q!}  \int_D e^{\frac{\xi^2}{2} \alpha_{\delta, D}(z)}  \varphi(z) (h^{\delta}_z)^{\otimes q} \, dz\right)\right]
			\nonumber \\
			= & \delta^{2\Delta_{\xi}} \int_D \varphi(z) e^{\frac{\xi^2}{2} \alpha_{\delta, D}(z)} \, dz 
			+ \sum_{q=1}^{\infty} I^G_q\left(w^{\varphi, \delta}_{q, \xi}\right).
		\end{align}
		Here 
		\begin{align}\label{eq:W_decomp_2}
			w^{\varphi, \delta}_{q, \xi}(x_1, \ldots, x_q) = & \delta^{2\Delta_{\xi}}  \frac{\xi^q}{q!}  \int_D e^{\frac{\xi^2}{2} \alpha_{\delta, D}(z)}  \varphi(z) \prod_{i=1}^q h^{\delta}_z (x_i)  \, dz
			=:  \frac{\xi^q}{q!}  \int_D W^{\delta, \varphi}_{q;x_1,\ldots, x_q; \xi}(z)\, dz
		\end{align}		
		when $q\ge 1$, $x_1, \ldots, x_q \in S_{\pm}$ and $\varphi$, $\delta$ are as above. As before, let us drop the notations $\varphi$ and $\xi$. For $0<\delta<d_z$ we have by \eqref{eq:mu_size_uv_ir},
		\begin{align}\label{eq:W_decomp_3}
			W^{\delta}_{q;x_1,\ldots, x_q}(z) 
            %= &   \delta^{2\Delta_{\xi}}  e^{\frac{\xi^2}{2} \alpha_{\delta, D}(z)}  \varphi(z) (h^{\delta}_z)^{\otimes q} (x_1, \ldots, x_q)  
			%\nonumber \\
			=  \delta^{2\Delta_{\xi}}  e^{\frac{\xi^2}{2} [\alpha_{\delta, d_z, D}(z) + \alpha_{ d_z, D}(z)]}  \varphi(z) \prod_{i=1}^q h^{\delta}_z(x_i)
			%\nonumber \\
			%= & \delta^{2\Delta_{\xi}}  e^{\frac{\xi^2}{10} \log\frac{d_z}{\delta} } e^{\frac{\xi^2}{2} \alpha_{ d_z, D}(z)}  \varphi(z) \prod_{i=1}^q h^{\delta}_z(x_i)
			%\nonumber \\
			%= & \delta^{2\Delta_{\xi}} \left(\frac{d_z}{\delta}\right)^{2\Delta_{\xi}}  e^{\frac{\xi^2}{2} \alpha_{ d_z, D}(z)}  \varphi(z) \prod_{i=1}^q h^{\delta}_z(x_i)
			%\nonumber \\
			=  d_z^{2\Delta_{\xi}} e^{\frac{\xi^2}{2} \alpha_{ d_z, D}(z)}  \varphi(z) \prod_{i=1}^q h^{\delta}_z(x_i).
		\end{align}
		Since as $\delta \to 0$, $h^{\delta}_z (x) = h^{\delta}_z (\epsilon, \gamma) = \epsilon \ind_{A_{\delta}(z)}(\gamma) \to \epsilon \ind_{A_0(z)}(\gamma) = h_z(x)$ for every $x \in S_{\pm}$, this shows,
		\begin{align}
			W^{\delta}_{q;x_1,\ldots, x_q}(z) \to & d_z^{2\Delta_{\xi}} e^{\frac{\xi^2}{2} \alpha_{ d_z, D}(z)}  \varphi(z) \prod_{i=1}^q h_z(x_i) \text{ as } \delta \to 0.
		\end{align}
		From the above we also have, for small enough $\delta>0$,
		\begin{align}\label{eq:W_decomp_4}
			|W^{\delta}_{q;x_1,\ldots, x_q}(z) | = & d_z^{2\Delta_{\xi}} e^{\frac{\xi^2}{2} \alpha_{ d_z, D}(z)}  |\varphi(z)|
			\le     d_z^{2\Delta_{\xi}} e^{2\Delta_{\xi} \log\frac{d(D)}{d_z} }  |\varphi(z)|
            \le  d(D)^{2\Delta_{\xi}} \lVert \varphi\rVert_{\infty}.
		\end{align}
		The r.h.s. of the above is clearly integrable on $D$. Therefore, from \eqref{eq:W_decomp_2}, we get
		\begin{align}
			\lim_{\delta \to 0} w^{ \delta}_q(x_1, \ldots, x_q) 
            = 
			     \frac{\xi^q}{q!} \int_D \varphi(z) d_z^{2\Delta_{\xi}} e^{\frac{\xi^2}{2} \alpha_{ d_z, D}(z)}  (h_z)^{\otimes q} (x_1, \ldots, x_q) \, dz,
		\end{align}
		which matches the required expression in \eqref{eq:W_conv_2}. One can show similarly as in the proof of Proposition \ref{prop:chaos_decomp_layering} that, for every $q\ge 1$, $I^{G}_q(w^{\delta}_q)$ is a Cauchy sequence in $L^2(\pr)$ as $\delta \to 0$. This implies that, as $\delta \to 0$, $I^G_q(w^{\delta}_q) \to I^G_q(w_q)$ in $L^2(\pr)$ for all $q\ge 1$. Thus \eqref{eq:W_decomp_1} gives us \eqref{eq:W_conv_1}.

    \subsection{Proof of Proposition \ref{prop:prelim_main_thm}} \label{app:proof_prop:prelim_main_thm}
		(a) Since $\lambda \beta^2 \to \xi^2$, we have $\Delta(\lambda, \beta) = \frac{\lambda}{10}(\cosh(\beta) -1) \to \frac{\xi^2}{20} = \Delta_{\xi}$. Therefore, as $\varphi \in C_b(D)$, by \eqref{eq:conv_1_pt_0}, \eqref{eq:W_1_pt_1} and the dominated convergence theorem,
		\begin{align}
			\langle V_{\lambda, \beta}(\varphi)\rangle =  \int_D \varphi(z) d_z^{2\Delta(\lambda, \beta)} e^{10\Delta(\lambda, \beta) \alpha_{d_z, D}(z)} \, dz	
			\to  \int_D \varphi(z) d_z^{2\Delta_{\xi}} e^{10\Delta_{\xi} \alpha_{d_z, D}(z)} \, dz = \langle  W_{\xi}(\varphi)\rangle.
		\end{align}

		(b) For each $q\ge 1$, applying the definition from \eqref{eq:W_conv_2} and Fubini's theorem,
		\begin{align}\label{eq:w_summable_1}
			& q! \lVert w^{\varphi}_{q, \xi} \rVert_{L^2(\nu^q)}^2 =   q! \int_{S_{\pm}^q} |w^{\varphi}_{q, \xi}(x_1, \ldots, x_q)|^2 \nu(dx_1)\cdots \nu(dx_q)
			\nonumber \\
			= & \frac{\xi^{2q}}{q!} \int_{S_{\pm}^q } \left| \int_D \varphi(z)d_z^{2\Delta_{\xi} } e^{\frac{\xi^2}{2} \alpha_{d_z, D}(z) }   (h_z)^{\otimes q} (x_1, \ldots, x_q) \, dz \right|^2 \nu(dx_1)\cdots \nu(dx_q)
			\nonumber \\
			%= & \frac{\xi^{2q}}{q!} \int_{S_{\pm}^q } \left( \int_D \int_D \varphi(z) \varphi(t) (d_z d_t)^{2\Delta_{\xi} }  e^{\frac{\xi^2}{2} [\alpha_{d_z, D}(z) + \alpha_{d_t, D}(t)] }  \prod_{i=1}^q(h_z(x_i) h_t(x_i)) \, dz \, dt  \right) \nu(dx_1)\cdots \nu(dx_q)
			%\nonumber \\
			= & \frac{\xi^{2q}}{q!} \int_{D\x D} \, dz \, dt  \, \varphi(z) \varphi(t) (d_z d_t)^{2\Delta_{\xi} }  e^{\frac{\xi^2}{2} [\alpha_{d_z, D}(z) + \alpha_{d_t, D}(t)] }
            \int_{S_{\pm}^q}  \prod_{i=1}^q(h_z(x_i) h_t(x_i)) \prod_{i=1}^q\nu(dx_i).  
		\end{align}
		Now we note that, as $\nu = \frac{1}{2}(\delta_{+1} + \delta_{-1}) \otimes \mu_D$,
		\begin{align}\label{eq:w_summable_2}
			&\int_{S_{\pm}^q}  \prod_{i=1}^q(h_z(x_i) h_t(x_i)) \prod_{i=1}^q\nu(dx_i) 
            % = &  \prod_{i=1}^q \int_{S_{\pm}}  (h_z(x_i) h_t(x_i)) \nu(dx_i)
			% \nonumber \\
			=   \prod_{i=1}^q \epsilon_i^2 \int_{S} \ind_{A_{0, D}(z)} (\gamma_i) \ind_{A_{0, D}(t)} (\gamma_i) \mu_D(d\gamma_i)
			\nonumber \\
			= & \prod_{i=1}^q \int_{S} \ind_{A_{|z-t|, D}(z, t)} (\gamma_i)  \mu_D(d\gamma_i)
			%\nonumber \\
			=  \prod_{i=1}^q  \mu_D (A_{|z-t|,D}(z, t)) = \alpha_{|z-t|,D} (z,t)^q.
		\end{align}
		Let us plug this into \eqref{eq:w_summable_1} and use the facts that $\alpha_{d_z, D} \le \alpha_{d_z, d(D), \C} = \frac{1}{5} \log \frac{d(D)}{d_z}$, $\alpha_{|z-t|, D}(z,t) \le \alpha_{|z-t|, d(D), \C}(z) = \frac{1}{5} \log \frac{d(D)}{|z-t|}$ from \eqref{eq:mu_D_size}. These give,
		\begin{align}\label{eq:w_summable_3}
			& q! \lVert w^{\varphi}_{q, \xi} \rVert_{L^2(\nu^q)}^2 =  \frac{\xi^{2q}}{q!} \int_{D\x D} \varphi(z) \varphi(t) (d_z d_t)^{2\Delta_{\xi} }  e^{\frac{\xi^2}{2} [\alpha_{d_z, D}(z) + \alpha_{d_t, D}(t)] } \alpha_{|z-t|,D} (z,t)^q \, dz \, dt
			\nonumber \\
			%\le & \frac{\xi^{2q}}{q!} \int_D \int_D  |\varphi(z) \varphi(t)| (d_z d_t)^{2\Delta_{\xi} }  e^{\frac{\xi^2}{10} \log\frac{d(D)^2}{d_z d_t} }  \left( \frac{1}{5}\log\frac{d(D)}{|z-t|}\right)^q \, dz \, dt
			%\nonumber \\
			\le & d(D)^{4\Delta_{\xi}} \int_{D\x D}  |\varphi(z) \varphi(t)|  \frac{1}{q!}\left( \frac{\xi^2}{5}\log\frac{d(D)}{|z-t|}\right)^q  \, dz \, dt
		\end{align}
		Thus,
		\begin{align}\label{eq:w_summable_4}
			& \sum_{q=1}^{\infty} q! \lVert w^{\varphi}_{q, \xi} \rVert_{L^2(\nu^q)}^2 \le  d(D)^{4\Delta_{\xi}} \int_{D\x D}   \, dz \, dt \, |\varphi(z) \varphi(t)|  \sum_{q=1}^{\infty}\frac{1}{q!}\left( \frac{\xi^2}{5}\log\frac{d(D)}{|z-t|}\right)^q 
			\nonumber \\
			\le & d(D)^{4\Delta_{\xi}} \int_{D\x D}   \, dz \, dt \, |\varphi(z) \varphi(t)|  \exp\left( \frac{\xi^2}{5}\log\frac{d(D)}{|z-t|} \right)
			\nonumber \\
			%= & d(D)^{8\Delta_{\xi}} \int_D \int_D  \, dz \, dt \, |\varphi(z) \varphi(t)| \, |z-t|^{-4\Delta_{\xi}}
			%\nonumber \\
			\le &  d(D)^{8\Delta_{\xi}} \lVert \varphi \rVert_{L^{\infty}(D)}^2 \int_{D\x D}  |z-t|^{-4\Delta_{\xi}} \, dz \, dt. 
		\end{align}
		Since $4\Delta_{\xi} < 1$ by our assumption, the above integral is finite (cf. \cite[Lemma A1]{cgpr21}).
		
		(c) Using the expressions from \eqref{eq:chaos_decomp_layering_2} and \eqref{eq:W_conv_2} we can write, for all $q \ge 1$ and $x_1, \ldots, x_q \in S_{\pm}$,
		\begin{align} \label{eq:f_w_conv_1}
			& \left| \lambda^{q/2} f^{\varphi}_{q, \lambda, \beta}(x_1, \ldots, x_q) - w^{\varphi}_{q, \xi}(x_1, \ldots, x_q)\right|^2
			\nonumber \\
			= & \left| \frac{\lambda^{q/2}}{q!} \int_D \varphi(z) \E[V_{\lambda, \beta}(z)] \prod_{i=1}^q (e^{\beta h_z(x_i) }-1) \, dz  - \frac{\xi^q}{q!} \int_D \varphi(z) \E[W_{\xi}(z)] \prod_{i=1}^q h_z(x_i) \, dz\right|^2
			\nonumber \\
			\le & 2 \left| \frac{1}{q!} \int_D  \varphi(z) \E[W_{\xi}(z)] \left\{ \prod_{i=1}^q \sqrt{\lambda} (e^{\beta h_z(x_i)}-1) - \prod_{i=1}^q \xi h_z(x_i)\right\} \, dz  \right|^2 
			\nonumber \\ 
			& + 2\left|  \frac{\lambda^{q/2}}{q!} \int_D  \varphi(z)( \E[V_{\lambda, \beta}(z)] - \E[W_{\xi}(z)])  \prod_{i=1}^q \sqrt{\lambda} (e^{\beta h_z(x_i)}-1)\, dz \right|^2
			\nonumber \\
			=: & \frac{2}{q!} |I_1(x_1, \ldots, x_q)|^2 + \frac{2 \lambda^{q/2}}{q!} |I_2(x_1, \ldots, x_q)|^2.
		\end{align}
		%where 
		%\begin{align}
		%	I_1(x_1, \ldots, x_q) = \int_D  \varphi(z) \E[W_{\xi}(z)] \left\{ \prod_{i=1}^q \sqrt{\lambda} (e^{\beta h_z(x_i)}-1) - \prod_{i=1}^q \xi h_z(x_i)\right\} \, dz
		%\end{align}
		%and 
		%\begin{align}
		%	I_2(x_1, \ldots, x_q) = \int_D  \varphi(z)( \E[V_{\lambda, \beta}(z)] - \E[W_{\xi}(z)])  \prod_{i=1}^q \sqrt{\lambda} (e^{\beta h_z(x_i)}-1)\, dz.
		%\end{align}
		
		Now observe that, for each $q \ge 1$, $x_1 = (\epsilon_1, \gamma_1), \ldots, x_q = (\epsilon_q, \gamma_q) \in S_{\pm}$ and $z \in D$, we have,
		\begin{align}
			\prod_{i=1}^q\sqrt{\lambda} (e^{\beta h_z(x_i)} - 1) - \prod_{i=1}^q\xi h_z(x_i) 
            %= & \prod_{i=1}^q\sqrt{\lambda} (e^{\beta \epsilon_i \ind_{A_{0,D}(z)}(\gamma_i) } -1 ) - \prod_{i=1}^q\xi \epsilon_i \ind_{A_{0,D}(z)}(\gamma_i)			
			\le & \left|\prod_{i=1}^q\sqrt{\lambda}(e^{\beta \epsilon_i}-1) - \prod_{i=1}^q\xi \epsilon_i \right| \left(\prod_{i=1}^q  \ind_{A_{0,D}(z)}(\gamma_i)\right)
		\end{align}
		since the LHS is non-zero only when $z \in \bar{\gamma_i}$ for all $i = 1,\ldots, q$. Applying this observation in the following and using \eqref{eq:W_1_pt_1},
		\begin{align} \label{eq:f_w_conv_2}
			& \int_{S_{\pm}^q} |I_1(x_1, \ldots, x_q)|^2 \nu(dx_1)\cdots \nu(dx_q)
			\nonumber \\
			= & \int_{S_{\pm}^q} \left| \int_D  \varphi(z) \E[W_{\xi}(z)] \left\{ \prod_{i=1}^q \sqrt{\lambda} (e^{\beta h_z(x_i)}-1) - \prod_{i=1}^q \xi h_z(x_i)\right\} \, dz \right|^2 \nu(dx_1)\cdots \nu(dx_q) 
			\nonumber \\
			= & \int_{D\x D} \, dz \, dt \, \varphi(z) \varphi(t) \E[W_{\xi}(z)] \E[W_{\xi}(t)] 
			\nonumber \\
			& 	\left[ \int_{S_{\pm}^q}	\left( \prod_{i=1}^q \sqrt{\lambda} (e^{\beta h_z(x_i)}-1) - \prod_{i=1}^q \xi h_z(x_i)\right)\left( \prod_{i=1}^q \sqrt{\lambda} (e^{\beta h_t(x_i)}-1) - \prod_{i=1}^q \xi h_t(x_i)\right)   \prod_{i=1}^q\nu(dx_i)  \right] 
			\nonumber \\
			\le & \int_{D\x D} \, dz \, dt \, |\varphi(z) \varphi(t)| (d_z d_t)^{2\Delta_{\xi} } e^{\frac{\xi^2}{2} [\alpha_{d_z, D}(z) + \alpha_{d_t, D}(t)] } 	\left( \prod_{i=1}^q  \int_{S} \ind_{A_{0,D}(z)} (\gamma_i) \ind_{A_{0,D}(t)}(\gamma_i) \mu_D(d\gamma_i)  \right)
			\nonumber \\
			& \hspace{5cm} \left[ \frac{1}{2^q}\sum_{\epsilon_1, \ldots, \epsilon_q = \pm 1}\left|\prod_{i=1}^q\sqrt{\lambda}(e^{\beta \epsilon_i}-1) - \prod_{i=1}^q\xi \epsilon_i \right|^2  \right]
			\nonumber \\
			= & \int_{D\x D} \, dz \, dt \, |\varphi(z) \varphi(t)| (d_z d_t)^{2\Delta_{\xi} } e^{\frac{\xi^2}{2} [\alpha_{d_z, D}(z) + \alpha_{d_t, D}(t)] } \alpha_{|z-t|,D} (z,t)^q
			\nonumber \\
			& \hspace{5cm} \left[ \frac{1}{2^q}\sum_{\epsilon_1, \ldots, \epsilon_q = \pm 1}\left|\prod_{i=1}^q\sqrt{\lambda}(e^{\beta \epsilon_i}-1) - \prod_{i=1}^q\xi \epsilon_i \right|^2  \right].
		\end{align}
		Observe that the integral appearing in the above can be seen to be finite by the calculations contained in \eqref{eq:w_summable_3}-\eqref{eq:w_summable_4}. Also, in the $\lambda \to\infty, \lambda \beta^2 \to \xi^2$ limit we have $\left|\prod_{i=1}^q\sqrt{\lambda}(e^{\beta \epsilon_i}-1) - \prod_{i=1}^q\xi \epsilon_i \right| \to 0$ for all $\epsilon_i = \pm 1$. Therefore we have shown that,
		\begin{align} \label{eq:f_w_conv_2.5}
			\int_{S_{\pm}^q} |I_1(x_1, \ldots, x_q)|^2 \nu(dx_1)\cdots \nu(dx_q)	\to 0.
		\end{align}
		
		Similarly, again using the fact that $(e^{\beta h_z(x_i)}-1) \le |e^{\beta \epsilon_i}-1| \ind_{A_{0,D}(z)}(\gamma_i)$ for $x_i = (\epsilon_i, \gamma_i) \in S_{\pm}$, we get
		\begin{align} \label{eq:f_w_conv_3}
			& \int_{S_{\pm}^q} |I_2(x_1, \ldots, x_q)|^2 \nu(dx_1)\cdots\nu(dx_q) 
			\nonumber \\
			= & \int_{S_{\pm}^q} \left|\int_D \varphi(z) (\E[V_{\lambda, \beta}(z)] - \E[W_{\xi}(z)]) \prod_{i=1}^q \sqrt{\lambda} (e^{\beta h_z(x_i) }-1) \, dz \right|^2\nu(dx_1) \cdots \nu(x_q)
			\nonumber \\
			= & \int_{D\x D} \, dz \, dt \, \varphi(z) \varphi(t) (\E[V_{\lambda, \beta}(z)] - \E[W_{\xi}(z)]) (\E[V_{\lambda, \beta}(t)] - \E[W_{\xi}(t)]) 
			\nonumber \\
			& \hspace{3cm} \int_{S_{\pm}^q} \prod_{i=1}^q \sqrt{\lambda} (e^{\beta h_z(x_i) }-1) \sqrt{\lambda}(e^{\beta h_t(x_i)}-1)  \nu(dx_1) \cdots \nu(x_q)
			\nonumber \\
			\le & \int_{D\x D} \, dz \, dt \, |\varphi(z) \varphi(t)| \, |\E[V_{\lambda, \beta}(z)] - \E[W_{\xi}(z)]| \, |\E[V_{\lambda, \beta}(t)] - \E[W_{\xi}(t)]| \, \alpha_{|z-t|, D}(z,t)^q 
			\nonumber \\
			& \hspace{6cm} \left[ \frac{1}{2} \sum_{\epsilon = \pm 1} [\sqrt{\lambda}(e^{\beta \epsilon}-1)]^2 \right]^q. 
		\end{align}
		Since $\lambda \to \infty, \lambda\beta^2 \to \xi^2 < 5$, for large $\lambda$ and small $\beta$ 
		\begin{align}
			\lambda \sum_{\epsilon = \pm 1} (e^{\beta \epsilon}-1)^2 \le 2\lambda |e^{|\beta|}-1|^2 < 2 \eta
		\end{align}
		for some $\eta < 5$. Therefore for these $\lambda, \beta$,
		$\left[ \frac{1}{2} \sum_{\epsilon = \pm 1} [\sqrt{\lambda}(e^{\beta \epsilon}-1)]^2 \right]^q \le \eta^q$, 
		and thus from \eqref{eq:f_w_conv_3} we have, 
		\begin{align}\label{eq:f_w_conf_3.5}
			& \int_{S_{\pm}^q} |I_2(x_1, \ldots, x_q)|^2 \nu(dx_1)\cdots\nu(dx_q) 
			%\nonumber \\
			%\le & \eta^q \int_{D\x D} \, dz \, dt \, |\varphi(z) \varphi(t)| \, |\E[V_{\lambda, \beta}(z)] - \E[W_{\xi}(z)]| \, |\E[V_{\lambda, \beta}(t)] - \E[W_{\xi}(t)]| \, \alpha_{|z-t|, D}(z,t)^q 	
			\nonumber \\
			\le & \eta^q \int_{D\x D} \, dz \, dt \, |\varphi(z) \varphi(t)| \, |\E[V_{\lambda, \beta}(z)] - \E[W_{\xi}(z)]| \, |\E[V_{\lambda, \beta}(t)] - \E[W_{\xi}(t)]| \, \left(\frac{1}{5} \log \frac{d(D)}{|z-t|}\right)^q
			\nonumber \\
			\le & q! \int_{D\x D} \, dz \, dt \, |\varphi(z) \varphi(t)| \, |\E[V_{\lambda, \beta}(z)] - \E[W_{\xi}(z)]| \, |\E[V_{\lambda, \beta}(t)] - \E[W_{\xi}(t)]|  \left( \frac{d(D)}{|z-t|}\right)^{\eta/5},
		\end{align}
		as $(\frac{\eta}{5} \log \frac{d(D)}{|z-t|})^q \le q! \exp( \frac{\eta}{5} \log \frac{d(D)}{|z-t|}) = q!  ( \frac{d(D)}{|z-t|})^{\eta/5}$. Since $\eta< 5$, $|\E[V_{\lambda, \beta}(z)] - \E[W_{\xi}(z)]| \to 0$ and $|\E[V_{\lambda, \beta}(t)] - \E[W_{\xi}(t)]| \to 0$ by part (a) of this proposition, the r.h.s. of the above converges to $0$. Thus we have shown that, 
		\begin{align}
			\int_{S_{\pm}^q} |I_2(x_1, \ldots, x_q)|^2 \nu(dx_1)\cdots\nu(dx_q)  \to 0, \text { as } \lambda \to \infty, \lambda \beta^2 \to \xi^2.
		\end{align}
        The above, together with \eqref{eq:f_w_conv_1} and \eqref{eq:f_w_conv_2.5}, proves the required convergence. 		
		
		(d) Let $\lambda$ be large enough and $\beta$ small enough so that, $\lambda |e^{|\beta|}-1|^2 \le \eta < 5$. Let us fix this $\eta$. Using the expression for $f^{\varphi}_{q, \lambda, \beta}$ from \eqref{eq:chaos_decomp_layering_2}, the estimate $|e^{\beta} - 1|\vee |e^{-\beta}-1| \le |e^{|\beta|}-1|$, similarly as in the above calculations we have for $N\ge 1$,
		\begin{align}
			& \sum_{q= N+1}^{\infty} q! \lVert \lambda^{q/2} f^{\varphi}_{q,\lambda, \beta}\rVert_{L^2(\nu^q)} 
			%\nonumber \\
			%= &  \sum_{q= N+1}^{\infty}  q! \lambda^q\int_{S_{\pm}^q} \nu(d_1)\cdots\nu(dx_q) |f^{\varphi}_{q, \lambda, \beta}(x_1, \ldots, x_q)|^2
			\nonumber \\
			= & \sum_{q= N+1}^{\infty} q! \lambda^q \int_{S_{\pm}^q} \nu(d_1)\cdots\nu(dx_q)  \left|\frac{1}{q!}\int_D \varphi(z) \E [V_{\lambda, \beta} (z)] \prod_{i=1}^q \left(e^{\beta h_z(x_i)} -1 \right) \, dz\right|^2
			%\nonumber \\
			%= & \sum_{q= N+1}^{\infty} \frac{\lambda^q}{q!} \int_{D\x D} \, dz \, dt \,  \varphi(z) \varphi(t) \E [V_{\lambda, \beta} (z)] \E [V_{\lambda, \beta} (t)] \int_{S_{\pm}^q} \nu(d x_1)\cdots\nu(dx_q) \prod_{i=1}^q \left(e^{\beta h_z(x_i)} -1 \right) \left(e^{\beta h_t(x_i)} -1 \right)
			\nonumber \\
			\le & \sum_{q= N+1}^{\infty} \frac{1}{q!} \left[ \lambda |e^{|\beta|}-1|^2  \right]^q  \int_{D\x D} \, dz \, dt \,   | \varphi(z) \varphi(t) | \E [V_{\lambda, \beta} (z)] \E [V_{\lambda, \beta} (t)] \alpha_{|z-t|, D} (z,t)^q
			\nonumber \\
			\le & \sum_{q= N+1}^{\infty} \frac{\eta^q}{q!}   \int_{D\x D} \, dz \, dt \, | \varphi(z) \varphi(t) | \E [V_{\lambda, \beta} (z)] \E [V_{\lambda, \beta} (t)] \alpha_{|z-t|, D} (z,t)^q
			\nonumber \\
			= & \sum_{q= N+1}^{\infty} \frac{\eta^q}{q!} \int_{D\x D} \, dz \, dt \, | \varphi(z) \varphi(t) |  (d_z d_t)^{2\Delta(\lambda, \beta)} e^{10\Delta(\lambda, \beta) [\alpha_{d_z, D}(z) + \alpha_{d_t, D}(t)]}  \alpha_{|z-t|, D} (z,t)^q.
		\end{align}
		Thus,
		\begin{align}
			& \sum_{q= N+1}^{\infty} q! \lVert \lambda^{q/2} f^{\varphi}_{q,\lambda, \beta}\rVert_{L^2(\nu^q)} 
			\nonumber \\
			\le &  \int_{D\x D} \, dz \, dt \, | \varphi(z) \varphi(t) | (d_z d_t)^{2\Delta(\lambda, \beta)} e^{10\Delta(\lambda, \beta) [\alpha_{d_z, D}(z) + \alpha_{d_t, D}(t)]}  \sum_{q= N+1}^{\infty} \frac{1}{q!}  (\eta \alpha_{|z-t|, D} (z,t))^q
            %\nonumber \\
            %\le & \lVert \varphi\rVert_{\infty}^2 \int_{D\x D} \, dz \, dt \,(d_z d_t)^{2\Delta(\lambda, \beta)} \left(\frac{d(D)^2}{d_z d_t} \right)^{2\Delta(\lambda, \beta)} \sum_{q= N+1}^{\infty} \frac{1}{q!}  (\eta \alpha_{|z-t|, D} (z,t))^q 
            \nonumber \\
            = &  \lVert \varphi\rVert_{\infty}^2  d(D)^{4\Delta(\lambda, \beta)} \int_{D\x D} \, dz \, dt \sum_{q= N+1}^{\infty} \frac{1}{q!}  (\eta \alpha_{|z-t|, D} (z,t))^q.  
		\end{align} 
		  Now using the fact that $\alpha_{|z-t|, D} (z,t) \le \alpha_{|z-t|, D}(z) \le \frac{1}{5} \log \frac{d(D)}{|z-t|} $ we see that the sum in the r.h.s. of the above is bounded by $ e^{\eta \alpha_{|z-t|, D} (z,t)} \le  d(D)^{\eta/5} |z-t|^{-\eta/5}$. So the r.h.s. is integrable on $D\x D$ by \cite[Lemma A1]{cgpr21} as $\eta<5$ and $D$ is bounded. Also, since $\eta$ is fixed, we have 
		\begin{align}
			\limsup_{\lambda \to \infty, \lambda\beta^2 \to \xi^2}\sum_{q= N+1}^{\infty} q! \lVert \lambda^{q/2} f^{\varphi}_{q,\lambda, \beta}\rVert_{L^2(\nu^q)} 
            \le C_{D, \varphi} \int_{D\x D} \, dz \, dt \sum_{q= N+1}^{\infty} \frac{1}{q!}  (\eta \alpha_{|z-t|, D} (z,t))^q < \infty,
		\end{align}
		for every $N \ge 1$. Since the integrand goes to $0$ as $N \to \infty$ for all $z, t \in D$, by dominated convergence we have the required result. 
	
	%\newpage

 \bibliographystyle{alpha}
\bibliography{BLS_ref}

\begin{thebibliography}{vdBCL18}

\bibitem[Ber17]{bersetycki17}
Nathana\"el Berestycki.
\newblock An elementary approach to {G}aussian multiplicative chaos.
\newblock {\em Electron. Commun. Probab.}, 22:Paper No. 27, 12, 2017.

\bibitem[Bog07]{bogachev07}
V.~I. Bogachev.
\newblock {\em Measure theory. {V}ol. {I}, {II}}.
\newblock Springer-Verlag, Berlin, 2007.

\bibitem[Cam17]{camia17}
Federico Camia.
\newblock Scaling limits, {B}rownian loops, and conformal fields.
\newblock In {\em Advances in disordered systems, random processes and some
  applications}, pages 205--269. Cambridge Univ. Press, Cambridge, 2017.

\bibitem[CGK16]{cgk16}
Federico Camia, Alberto Gandolfi, and Matthew Kleban.
\newblock Conformal correlation functions in the {B}rownian loop soup.
\newblock {\em Nuclear Phys. B}, 902:483--507, 2016.

\bibitem[CGPR21]{cgpr21}
Federico Camia, Alberto Gandolfi, Giovanni Peccati, and Tulasi~Ram Reddy.
\newblock Brownian loops, layering fields and imaginary {G}aussian
  multiplicative chaos.
\newblock {\em Comm. Math. Phys.}, 381(3):889--945, 2021.

\bibitem[JSW20]{jsw20}
Janne Junnila, Eero Saksman, and Christian Webb.
\newblock Imaginary multiplicative chaos: moments, regularity and connections
  to the {I}sing model.
\newblock {\em Ann. Appl. Probab.}, 30(5):2099--2164, 2020.

\bibitem[Law05]{lawler05}
Gregory~F. Lawler.
\newblock {\em Conformally invariant processes in the plane}, volume 114 of
  {\em Mathematical Surveys and Monographs}.
\newblock American Mathematical Society, Providence, RI, 2005.

\bibitem[LJ10]{lejan10}
Yves Le~Jan.
\newblock Markov loops and renormalization.
\newblock {\em Ann. Probab.}, 38(3):1280--1319, 2010.

\bibitem[LP18]{lp18}
G\"unter Last and Mathew Penrose.
\newblock {\em Lectures on the {P}oisson process}, volume~7 of {\em Institute
  of Mathematical Statistics Textbooks}.
\newblock Cambridge University Press, Cambridge, 2018.

\bibitem[LTF07]{lt07}
Gregory~F. Lawler and Jos\'e{}~A. Trujillo~Ferreras.
\newblock Random walk loop soup.
\newblock {\em Trans. Amer. Math. Soc.}, 359(2):767--787, 2007.

\bibitem[LW04]{lw04}
Gregory~F. Lawler and Wendelin Werner.
\newblock The {B}rownian loop soup.
\newblock {\em Probab. Theory Related Fields}, 128(4):565--588, 2004.

\bibitem[NW11]{nw11}
\c~Serban Nacu and Wendelin Werner.
\newblock Random soups, carpets and fractal dimensions.
\newblock {\em J. Lond. Math. Soc. (2)}, 83(3):789--809, 2011.

\bibitem[Pom92]{pommerenke92}
Ch. Pommerenke.
\newblock {\em Boundary behaviour of conformal maps}, volume 299 of {\em
  Grundlehren der mathematischen Wissenschaften [Fundamental Principles of
  Mathematical Sciences]}.
\newblock Springer-Verlag, Berlin, 1992.

\bibitem[PT11]{pt11}
Giovanni Peccati and Murad~S. Taqqu.
\newblock {\em Wiener chaos: moments, cumulants and diagrams}, volume~1 of {\em
  Bocconi \& Springer Series}.
\newblock Springer, Milan; Bocconi University Press, Milan, 2011.
\newblock A survey with computer implementation, Supplementary material
  available online.

\bibitem[RV14]{rv14}
R\'emi Rhodes and Vincent Vargas.
\newblock Gaussian multiplicative chaos and applications: a review.
\newblock {\em Probab. Surv.}, 11:315--392, 2014.

\bibitem[vdBCL18]{vcl18}
Tim van~de Brug, Federico Camia, and Marcin Lis.
\newblock Spin systems from loop soups.
\newblock {\em Electron. J. Probab.}, 23:Paper No. 81, 17, 2018.

\bibitem[Wer06]{werner06}
Wendelin Werner.
\newblock Some recent aspects of random conformally invariant systems.
\newblock In {\em Mathematical statistical physics}, pages 57--99. Elsevier B.
  V., Amsterdam, 2006.

\bibitem[Wer08]{werner08}
Wendelin Werner.
\newblock The conformally invariant measure on self-avoiding loops.
\newblock {\em J. Amer. Math. Soc.}, 21(1):137--169, 2008.

\bibitem[WP21]{wp21}
Wendelin Werner and Ellen Powell.
\newblock Lecture notes on the gaussian free field, 2021.

\end{thebibliography}
%\printbibliography

\end{document}